\documentclass{article}
\usepackage{amssymb, amsmath, amscd, amsthm, amsfonts, amscd, mathrsfs}
\usepackage[pdftex]{graphicx}

\newcommand \eval{\mathrm{eval}}

\newcommand \del {\partial}

\newcommand \modk {{\mathcal M}_{\kappa}}

\newcommand \Oomega {{\tilde {\mathcal V}({\mathcal R})}}

\newcommand \oomega {{\mathscr X}}
\newcommand \HH {{\mathcal H}}

\newcommand \la {\lambda}

\newcommand \R {{\mathcal R}}

 \newcommand \Prob {\mathbb P}
\newcommand \prob {\mathbb P}

\newcommand \ee {\mathbb E}

\newcommand \U {{\mathcal U}}

\newcommand \Y {{\mathcal  Y}}

\newcommand \BB {\mathfrak B}

\newcommand \B {{\mathfrak B}}
\newcommand \F {{\mathcal F}}

\newcommand \mm {{\mathfrak m}}
\newcommand \MM {{\mathfrak M}}

\newtheorem{theorem}{Theorem}

\newtheorem {lemma} {Lemma}[section]

\newtheorem{proposition}[lemma]{Proposition}

\newtheorem{corollary}[lemma]{Corollary}

\newtheorem{assumption} [lemma]{Assumption}

\title{Limit Theorems for Translation Flows}
\author{Alexander I. Bufetov\footnote{Laboratoire d'Analyse, Topologie, Probabilit{\'e}s, 
Aix-Marseille Universit{\'e}, CNRS,  the  Steklov Institute of Mathematics, the Institute for Information Transmission Problems,
National Research University Higher School of Economics,
the Independent University of Moscow, Rice University.}}
\date{{\it To William Austin Veech}}
\begin{document}
\maketitle

\begin{abstract}

The aim of this paper is to obtain an asymptotic expansion for ergodic integrals of
translation flows on flat surfaces of higher genus (Theorem \ref{multiplicmoduli})
and to give a limit theorem for these flows (Theorem \ref{limthmmoduli}).

\end{abstract}

\tableofcontents
\section{Introduction.}

\subsection{Outline of the main results.}

A compact Riemann surface endowed with an abelian differential admits two natural flows, called, respectively,
{\it horizontal} and {\it vertical}. One of the main objects of this paper is the space ${\mathfrak B}^+$
of H{\"o}lder cocycles over the vertical flow, invariant under the holonomy by the horizontal flow.
Equivalently, cocycles in ${\mathfrak B}^+$ can be viewed, in the spirit of R. Kenyon \cite{kenyon} and F. Bonahon \cite{bonahon1},  \cite{bonahon2},
as finitely-additive transverse invariant measures for the horizontal foliation of our abelian differential.
Cocycles in ${\mathfrak B}^+$ are closely connected to the invariant distributions for translation flows in the sense
of G.Forni \cite{F2}.

The space ${\mathfrak B}^+$ is finite-dimensional, and for a generic abelian
differential the
dimension of ${\mathfrak B}^+$ is equal to the genus of the underlying surface.
Theorem \ref{multiplicmoduli}, which extends earlier work of A.Zorich \cite{Z} and G. Forni \cite{F2},
states that the  time integral of  a Lipschitz function under the vertical flow
can be uniformly approximated by a suitably chosen cocycle from ${\mathfrak B}^+$ up to an error that grows
more slowly than any power of time.
The renormalizing action of the Teichm{\"u}ller flow on the space of H{\"o}lder cocycles now allows one
to obtain limit theorems for translation flows on flat surfaces (Theorem \ref{limthmmoduli}).

The statement of Theorem \ref{limthmmoduli} can be informally summarized as follows.
Taking the leading term in the asymptotic expansion of Theorem \ref{multiplicmoduli},
to a generic abelian differential one assigns a compactly supported probability measure on
the space of continuous functions on the unit interval.  The normalized distribution
of the time integral of a Lipschitz function converges, with respect to weak topology, to the trajectory of the
corresponding ``asymptotic distribution'' under the action of the Teichm{\"u}ller flow.
Convergence is exponential with respect to both the L{\'e}vy-Prohorov and the Kantorovich-Rubinstein metric.

\subsection{H{\"o}lder cocycles over translation flows.}
Let $\rho\geq 2$ be an integer, let $M$ be a compact orientable surface of genus $\rho$,
and let $\omega$ be a holomorphic one-form on $M$.
Denote by
$
\nu=i(\omega\wedge {\overline \omega})/2
$
the area form induced by $\omega$ and assume that $\nu(M)=1$.

Let $h_t^+$ be the {\it vertical} flow on $M$ (i.e., the flow corresponding to $\Re(\omega)$);
let $h_t^-$ be the {\it horizontal} flow on $M$ (i.e., the flow corresponding to $\Im(\omega)$).
The flows $h_t^+$, $h_t^-$ preserve the area $\nu$.

Take $x\in M$, $t_1, t_2\in {\mathbb R}_+$ and assume that the closure of the set
\begin{equation}
\label{admrectpsan}
\{h^+_{\tau_1} h^{-}_{\tau_2}x, 0\leq \tau_1< t_1, 0\leq \tau_2< t_2\}
\end{equation}
does not contain zeros of the form $\omega$.  The set (\ref{admrectpsan})
is then called {\it an admissible rectangle} and denoted $\Pi(x, t_1, t_2)$.
Let ${\overline {\mathfrak C}}$ be the semi-ring of admissible rectangles.

Consider the linear space ${\mathfrak B}^+$ of H{\"o}lder cocycles $\Phi^+(x,t)$ over the vertical
flow $h_t^+$ which are invariant under horizontal holonomy. More precisely, a function
$\Phi^+(x,t): M\times {\mathbb R}\to {\mathbb R}$ belongs to the space ${\mathfrak B}^+$
if it satisfies:
\begin{assumption}
\label{bplusx}
\begin{enumerate}
\item  $\Phi^+(x,t+s)=\Phi^+(x,t)+\Phi^+(h_t^+x, s)$;
\item  There exists $t_0>0$, $\theta>0$ such that
$|\Phi^+(x,t)|\leq t^{\theta}$ for all $x\in M$ and all $t\in {\mathbb R}$ satisfying $|t|<t_0$;
\item  If $\Pi(x, t_1, t_2)$ is an admissible rectangle, then
$\Phi^+(x, t_1)=\Phi^+(h_{t_2}^-x, t_1)$.
\end{enumerate}
\end{assumption}
A cocycle $\Phi^+\in\BB^+$ can equivalently be thought of as a finitely-additive H{\"o}lder measure defined on all
arcs $\gamma=[x, h_t^+x]$ of the vertical flow and
invariant  under the horizontal flow. It will often be convenient to identify the cocycle with the corresponding
finitely-additive  measure.
For example, let $\nu^+$ be the Lebesgue measure on leaves of the vertical foliation;
the corresponding cocycle $\Phi_1^+$  defined by $\Phi_1^+(x,t)=t$ of course belongs to ${\BB}^+$.

In the same way  define the space ${\mathfrak B}^-$ of H{\"o}lder cocycles $\Phi^-(x,t)$ over
the horizontal flow $h_t^-$  which are invariant under vertical holonomy.
A cocycle $\Phi^-\in\BB^-$ can equivalently be thought of as a finitely-additive H{\"o}lder measure defined on all
arcs ${\tilde \gamma}=[x, h_t^-x]$ of the horizontal flow and
invariant  under the vertical flow.
Let $\nu^-$ be the Lebesgue measure on leaves of the horizontal foliation;
the corresponding cocycle $\Phi_1^-$  is defined by the formula $\Phi_1^-(x,t)=t$; of course,
$\Phi_1^-\in{\BB}^-$.

Given $\Phi^+\in {\mathfrak B}^+$, $\Phi^-\in {\mathfrak B}^-$, a finitely additive
measure $\Phi^+\times \Phi^-$ on the semi-ring ${\overline {\mathfrak C}}$ of admissible rectangles
is introduced by the formula
\begin{equation}
\Phi^+\times \Phi^-(\Pi(x, t_1, t_2))=\Phi^+(x,t_1)\cdot \Phi^-(x, t_2).
\end{equation}

In particular, for $\Phi^-\in {\mathfrak B }^-$, set $m_{\Phi^-}=\nu^+\times\Phi^-$:
\begin{equation}
\label{mphipsan}
m_{\Phi^-}(\Pi(x, t_1, t_2))=t_1\Phi^-(x, t_2).
\end{equation}
For any $\Phi^-\in {\mathfrak B}^-$ the measure
$m_{\Phi^-}$ satisfies $(h_t^+)_*m_{\Phi^-}=m_{\Phi^-}$ and is an invariant distribution in the
sense of G.~Forni \cite{F1}, \cite{F2}. For instance, $m_{\Phi_1^-}=\nu$.

An ${\mathbb R}$-linear pairing between ${\mathfrak B}^+$ and ${\mathfrak B^-}$ is given, for
$\Phi^+\in {\mathfrak B}^+$, $\Phi^-\in {\mathfrak B}^-$,
by the formula
\begin{equation}
\label{psanpairing}
\langle \Phi^+, \Phi^-\rangle=\Phi^+\times \Phi^-(M).
\end{equation}

\subsection{Characterization of cocycles.}

For an abelian differential ${\bf X}=(M, \omega)$
let $\B^+_c({\bf X})$ be the space  of continuous holonomy-invariant cocycles:
more precisely, a function
$\Phi^+(x,t): M\times {\mathbb R}\to {\mathbb R}$ belongs to the space ${\mathfrak B}_c^+({\bf X})$
if it satisfies conditions 1 and 3 of Assumption \ref{bplusx}, while condition 2
is replaced by the following weaker version:

For any $\varepsilon>0$ there exists $\delta>0$  such that
$|\Phi^+(x,t)|\leq \varepsilon$ for all $x\in M$ and all $t\in {\mathbb R}$ satisfying $|t|<\delta$.

Given an abelian differential ${\bf X}=(M, \omega)$, we now construct, following Katok \cite{katok},
an explicit  mapping of $\BB_c^+(M, \omega)$ to $H^1(M, {\mathbb R})$.

A continuous closed curve $\gamma$ on $M$ is called {\it rectangular} if
$$
\gamma=\gamma_1^+\sqcup\dots \sqcup \gamma_{k_1}^+\bigsqcup \gamma_1^-\sqcup\dots \sqcup \gamma_{k_2}^-,
$$

where $\gamma_i^+$ are arcs of the flow $h_t^+$, $\gamma_i^-$ are arcs of the flow $h_t^-$.

For $\Phi^+\in\BB_c^+$ define
$$
\Phi^+(\gamma)=\sum_{i=1}^{k_1} \Phi^+(\gamma_i^+);
$$
similarly, for $\Phi^-\in\BB_c^-$ write
$$
\Phi^-(\gamma)=\sum_{i=1}^{k_2} \Phi^-(\gamma_i^-).
$$

Thus, a cocycle $\Phi^+\in\BB_c$ assigns a number $\Phi^+(\gamma)$ to every closed rectangular curve $\gamma$.
It is shown in Proposition \ref{acthomology} below that if $\gamma$ is homologous to $\gamma^{\prime}$, then
$\Phi^+(\gamma)=\Phi^+(\gamma^{\prime})$. For an abelian differential ${\bf X}=(M, \omega)$,
we thus obtain  maps
\begin{equation}
\label{maptocohomology}
{\check {\mathcal I}}_{\bf X}^+: \B_c^+({\bf X})\to H^1(M, {\mathbb R}), \ {\check {\mathcal I}}_{\bf X}^-: \B_c^-({\bf X})\to H^1(M, {\mathbb R}).
\end{equation}

For a generic abelian differential, the image of $\B^+$ under the map ${\check {\mathcal I}}_{\bf X}^+$ is the strictly unstable space of
the Kontsevich-Zorich cocycle over the Teichm{\"u}ller flow.

More precisely, let $\kappa=(\kappa_1, \dots, \kappa_{\sigma})$ be
a nonnegative integer vector such that  $\kappa_1+\dots+\kappa_{\sigma}=2\rho-2$.
Denote by $\modk$ the moduli space of pairs $(M, \omega)$, where $M$ is
a Riemann surface of genus $\rho$ and $\omega$ is a holomorphic differential
of area $1$ with singularities of orders $\kappa_1, \dots, \kappa_{\sigma}$.
The space $\modk$ is often called the {\it stratum} in the moduli space of abelian differentials.

The Teichm{\"u}ller flow ${\bf g}_s$ on $\modk$ sends the
modulus of a pair $(M, \omega)$ to the modulus of the pair
$(M, \omega^{\prime})$, where $\omega^{\prime}=e^s\Re(\omega)+ie^{-s}\Im(\omega)$;
the new complex structure on $M$ is uniquely determined by the requirement that the form $\omega^{\prime}$
 be holomorphic.
As shown by Veech, the space $\modk$ need not be connected;
let $\HH$ be a connected component of $\modk$.

Let ${\mathbb H}^1(\HH)$ be the fibre bundle over $\HH$ whose fibre at a point $(M, \omega)$ is the cohomology group
$H^1(M, {\mathbb R})$. The bundle ${\mathbb H}^1(\HH)$ carries the {\it Gauss-Manin connection} which declares continuous integer-valued sections of our bundle to be flat and is uniquely defined by that requirement. Parallel transport with respect to the Gauss-Manin connection
along the orbits of the Teichm{\"u}ller flow yields a cocycle over the Teichm{\"u}ller flow, called the {\it Kontsevich-Zorich cocycle}
and denoted  ${\mathbf A}={\mathbf A}_{KZ}$.

Let $\Prob$ be a ${\bf g}_s$-invariant ergodic probability
measure on $\HH$. For ${\bf X}\in\HH$, ${\bf X}=(M, \omega)$, let ${\mathfrak B}_{\bf X}^+$, ${\mathfrak B}_{\bf X}^-$ be the
corresponding spaces of H{\"o}lder cocycles.

Denote by $E_{\bf X}^u\subset H^1(M, {\mathbb R})$
the space spanned by vectors corresponding to the positive Lyapunov exponents of the Kontsevich-Zorich
cocycle, by $E_{\bf X}^s\subset H^1(M, {\mathbb R})$ the space spanned by vectors corresponding to
the negative exponents of the Kontsevich-Zorich cocycle.

\begin{proposition}
For $\Prob$-almost all ${\bf X}\in\HH$ the map
${\check {\mathcal I}}_{\bf X}^+$ takes $\BB_{\bf X}^+$ isomorphically onto $E_{\bf X}^u$, the map
${\check {\mathcal I}}_{\bf X}^-$ takes $\BB_{\bf X}^-$ isomorphically onto $E_{\bf X}^s$.

The pairing $\langle, \rangle$ is nondegenerate and is taken by the isomorphisms ${\mathcal I}_{\bf X}^+$, ${\mathcal I}_{\bf X}^-$ to the cup-product in the cohomology $H^1(M, {\mathbb R})$.
\end{proposition}

{\bf Remark.} In particular, if $\Prob$ is the Masur-Veech ``smooth" measure \cite{masur, veech}, then
$\dim {\BB}_{\bf X}^+=\dim {\BB}_{\bf X}^-=\rho$.

{\bf Remark.} The isomorphisms ${\check {\mathcal I}}_{\bf X}^+$, ${\check {\mathcal I}}_{\bf X}^-$
are analogues of G. Forni's isomorphism \cite{F2} between his space of invariant distributions and the
unstable space of the Kontsevich-Zorich cocycle.

Now recall that to every cocycle $\Phi^-\in \BB_{\bf X}^-$ we have assigned a
finitely-additive H{\"o}lder measure $m_{\Phi^-}$ invariant under the flow $h_t^+$.
Considering these measures  as distributions in the sense of Sobolev and Schwartz,
we arrive at the following proposition.

\begin{proposition}
\label{forni-classif} Let $\mathbb{P}$ be an ergodic
$g_s-invariant$ probability measure on $\mathcal{H}.$ The for
$\mathbb{P}-almost$ every abelian differential $(M,\omega)$ the
space $\{m_{\Phi^-},\Phi^-\in\mathfrak{B}^-(M,\omega)\}$ coincides
with the space of $h^{+}_{t}-invariant$ distributions belonging to
the Sobolev space $H^{-1}.$
\end{proposition}

\begin{proof} By definition for any
$\Phi^+\in\mathfrak{B}^+$ the distribution $m_{\Phi^+}$ is
$h^-_t$-invariant and belongs to the Sobolev space $H^{-1}$.
G.Forni has shown that for any $g_s$-invariant ergodic measure
$\mathbb{P}$ and $\mathbb{P}$-almost every abelian differential
$(M,\omega),$ the dimension of the space of $h_t^-$-invariant
distributions belonging to the Sobolev space $H^{-1}$ ${\it does\
not\ exceed}$ the dimension of the strictly expanding Oseledets
subspace of the Kontsevich-Zorich cocycle (under mild additional
assumption on the measure $\mathbb{P}$ G.Forni proved that these
dimensions are in fact equal, see Theorem 8.3  and Corollary
$8.3^{\prime}$ in \cite{F2}; note, however, that the proof of
the {\it upper} bound in Forni's Theorem only uses ergodicity of the measure).
Since the dimension of the space
$\{m_{\Phi^-},\Phi^-\in\mathfrak{B}^-\}$ equals that of the strictly expanding space for the
Kontsevich-Zorich cocycle for $\mathbb{P}$-almost all
$(M,\omega),$ the proposition is proved completely.
\end{proof}

Consider the inverse isomorphisms
$$
{\mathcal I}_{\bf X}^+=\left({\check {\mathcal I}}_{\bf X}^+\right)^{-1}; \ {\mathcal I}_{\bf X}^-=\left({\check {\mathcal I}}_{\bf X}^-\right)^{-1}.
$$
Let $1=\theta_1>\theta_2>\dots>\theta_{l}>0$
be the distinct positive Lyapunov exponents of the Kontsevich-Zorich cocycle ${\bf A}_{KZ}$, and let
$$
E_{\bf X}^u=\bigoplus\limits_{i=1}^l E_{{\bf X}, \theta_i}^u
$$
be the corresponding Oseledets decomposition at ${\bf X}$.
\begin{proposition}
\label{hoeldergrowth}
Let
$v\in E_{{\bf X}, \theta_i}^u$, $v\neq 0$, and denote $\Phi^+={\mathcal I}_{\bf X}^+(v)$.
Then for any $\varepsilon>0$ the cocycle $\Phi^+$ satisfies the
H{\"o}lder condition with exponent $\theta_i-\varepsilon$ and for any $x\in M({\bf X})$ such that
$h_t^+x$ is defined for all $t\in {\mathbb R}$
we have
$$
\limsup\limits_{T\to\infty} \frac{\log|\Phi^+(x,T)|}{\log T}=\theta_i; \  \limsup\limits_{T\to 0} \frac{\log|\Phi^+(x,T)|}{\log T}=\theta_i.
$$
\end{proposition}

\begin{proposition}
\label{cochyperb}
If the Kontsevich-Zorich cocycle does not have zero Lyapunov exponent with respect
to $\Prob$, then $\B^+_c({\bf X})=\B^+({\bf X})$.
\end{proposition}
{\bf Remark.} The condition of the absence of zero Lyapunov exponents can be weakened:
it suffices to require that the Kontsevich-Zorich cocycle act isometrically on the
neutral Oseledets subspace corresponding to the Lyapunov exponent zero.
Isometric action means here that there exists an inner product which
depends measurably on the point in the stratum and which is invariant under the
Kontsevich-Zorich cocycle. In all known examples (see, e.g., \cite{FMZ})
the action of the Kontsevich-Zorich cocycle on its neutral Lyapunov subspace is
isometric; note, however, that the examples of \cite{FMZ} mainly concern measures
invariant under the action of the whole group $SL(2, \mathbb R)$.

{\bf Question.} Does  there exist a ${\bf g}_s$-invariant ergodic probability measure $\Prob^{\prime}$
 on $\HH$ such that the inclusion $\BB^+\subset \BB^+_c$ is proper almost surely with respect to $\Prob^{\prime}$?

{\bf Remark.} G.Forni has made the following
remark. To a cocycle $\Phi^+\in\mathfrak{B}^+$ assign a $1$-current
$\beta_{\Phi^+},$ defined, for a smooth 1-form $\eta$ on the
surface $M,$ by the formula
$$\beta_{\Phi^+}(\eta)=\int_M\Phi^+\wedge\eta,$$ where the
integral in the right hand side is defined as the limit of Riemann
sums. The resulting current $\beta_{\Phi^+}$ is a ${\it basic\
current}$ for the horizontal foliation.

The mapping of H\"{o}lder cocycles into the cohomology
$H^1(M,\mathbb{R})$ of the surface corresponds to G. Forni's map
that to each basic current assigns its cohomology class (the
latter is well-defined by the de Rham Theorem). In particular, it
follows that for any ergodic $g_s$-invariant probability measure
$\mathbb{P}$ on $\mathcal{H}$ and  $\mathbb{P}$-almost every
abelian differential $(M,\omega)$ every basic current from the
Sobolev space $H^{-1}$ is induced by a H\"{o}lder cocycle
$\Phi^+\in\mathfrak{B}^+(M,\omega).$

\subsection{Approximation of weakly Lipschitz functions.}
\subsubsection{The space of weakly Lipschitz functions.}

The space of Lipschitz functions is not invariant under $h_t^+$, and a larger
function space $Lip_w^+(M, \omega)$ of weakly Lipschitz functions is introduced as follows.
A bounded measurable function $f$ belongs to $Lip_w^+(M, \omega)$ if there exists a
constant $C$, depending only on $f$, such that for any admissible rectangle $\Pi(x, t_1, t_2)$ we
have
\begin{equation}
\label{weaklippsan}
\left| \int_0^{t_1} f(h_t^+x)dt -\int_0^{t_1}  f(h_t^+(h^-_{t_2}x)dt \right|\leq C.
\end{equation}
Let $C_f$ be the infimum of all $C$ satisfying (\ref{weaklippsan}). We norm $Lip_w^+(M, \omega)$ by setting
$$
||f||_{Lip_w^+}=\sup_M f+C_f.
$$

By definition, the space  $Lip_w^+(M, \omega)$ contains all Lipschitz functions on $M$ and is
invariant under $h_t^+$. If $\Pi$ is an admissible rectangle, then its characteristic function
$\chi_{\Pi}$  is weakly Lipschitz (I am grateful to C. Ulcigrai for this remark).

We denote by $Lip_{w,0}^+(M, \omega)$ the
subspace of  $Lip_w^+(M, \omega)$ of
functions whose integral with respect to $\nu$ is $0$.

For any $f\in Lip_{w}^+(M, \omega)$ and any $\Phi^-\in \BB^-$ the
integral $\int_M f dm_{\Phi^-}$ can be defined as the limit of Riemann sums.

\subsubsection{The cocycle corresponding to a weakly Lipschitz function.}
If the pairing $\langle, \rangle$ induces an isomorphism between $\BB^+$ and the
dual $(\BB^-)^*$, then one can assign to a function $f\in Lip_{w}^+(M, \omega)$ the functional $\Phi_f^+$ by
the formula
\begin{equation}
\label{defphifteich}
\langle \Phi_f^+, \Phi^-\rangle =\int\limits_M f dm_{\Phi^-}, \Phi^-\in \BB^-.
\end{equation}

By definition,  $\Phi^+_{f\circ h_t^+}=\Phi_f^+$.
We are now proceed to the formulation of the first main result of this paper, the Approximation Theorem \ref{multiplicmoduli}.
\begin{theorem}
\label{multiplicmoduli}
Let $\Prob$ be an ergodic probability ${\bf g}_s$-invariant measure on $\HH$.
For any $\varepsilon>0$ there exists
a constant $C_{\varepsilon}$ depending only on $\Prob$ such that for
$\Prob$-almost every ${\bf X}\in\HH$, any $f\in Lip_w^+({\bf X})$, any $x\in M$ and any $T>0$ we have
$$
\left| \int_0^T
f\circ h_t^+(x) dt -\Phi^+_f(x,T)\right|\leq
C_{\varepsilon}||f||_{Lip_w^+}(1+T^{\varepsilon}).
$$
\end{theorem}

\subsubsection{Invariant measures with simple Lyapunov spectrum.}

Consider the case in which the
Lyapunov spectrum of
the Kontsevich-Zorich cocycle is simple in restriction
to the space $E^u$ (as, by the Avila-Viana theorem \cite{AV},
is the case with the Masur-Veech smooth measure).
Let  $l_0={\rm dim} E^u$ and let
\begin{equation}
1=\theta_1>\theta_2>\dots>\theta_{l_0}
\end{equation}
be the corresponding simple expanding Lyapunov exponents.

Let $\Phi_1^+$ be given by the formula $\Phi_1^+(x,t)=t$
and introduce a basis
\begin{equation}
\Phi_1^+, \Phi_2^+, \dots, \Phi_{l_0}^+
\end{equation}
in ${\mathfrak B}^+_{\bf X}$ in such a
way that ${\check {\mathcal I}}_{\bf X}^+(\Phi_i^+)$ lies in the Lyapunov subspace with exponent $\theta_i$.
By Proposition \ref{hoeldergrowth}, for any $\varepsilon>0$ the cocycle $\Phi_i^+$ satisfies the H{\"o}lder condition with
exponent $\theta_i-\varepsilon$, and for any $x\in M({\bf X})$ we have
$$
\limsup\limits_{T\to\infty} \frac{\log|\Phi_i^+(x,T)|}{\log T}=\theta_i; \ \limsup\limits_{T\to 0} \frac{\log|\Phi_i^+(x,T)|}{\log T}=\theta_i.
$$

Let $\Phi_1^-, \dots, \Phi_{l_0}^-$ be the dual basis in ${\mathfrak  B}^-_{\bf X}$.
Clearly, $\Phi_1^-(x,t)=t$.

By definition, we have
\begin{equation}
\Phi_f^+=\sum_{i=1}^{l_0} m_{\Phi_i^-}(f)\Phi_i^+.
\end{equation}

Noting that by definition we have
$$m_{\Phi_1^-}=\nu,$$

we derive from
Theorem \ref{multiplicmoduli} the following corollary.

\begin{corollary}

Let $\Prob$ be an invariant ergodic probability measure for the Teichm{\"u}ller
flow such that with respect to $\Prob$ the
Lyapunov spectrum of
the Kontsevich-Zorich cocycle is simple in restriction to its strictly
expanding subspace.

Then for any $\varepsilon>0$ there exists
a constant $C_{\varepsilon}$ depending
only on $\Prob$ such that for
$\Prob$-almost every ${\bf X}\in\HH$, any
$f\in Lip_w^+({\bf X})$, any $x\in {\bf X}$
and any $T>0$ we have
$$
\left| \int_0^T f\circ h_t^+(x) dt - T(\int_M fd\nu)-
\sum\limits_{i=2}^{l_0} m_{\Phi_i^-}(f)\Phi_i^+(x,T)
\right|\leq C_{\varepsilon}||f||_{Lip_w^+}(1+T^{\varepsilon}).
$$
\end{corollary}
For horocycle flows a related asymptotic expansion has been obtained by Flaminio and Forni \cite{flafo}.

{\bf Remark.}
If $\Prob$ is the Masur-Veech smooth measure on $\HH$, then
it follows from the work of G.Forni \cite{F1}, \cite{F2}, \cite{forniarxiv}
and S. Marmi, P. Moussa, J.-C. Yoccoz \cite{MMY}
that the left-hand side is bounded for any $f\in C^{1+\varepsilon}(M)$
(in fact, for any $f$ in the Sobolev space $H^{1+\varepsilon}$).
In particular, if $f\in C^{1+\varepsilon}(M)$ and $\Phi_f^+=0$,
then $f$ is a coboundary.

\subsection{Holonomy invariant transverse finitely-additive measures for oriented measured foliations.}

Holonomy-invariant cocycles assigned to an abelian differential can be interpreted
as transverse invariant measures for its foliations in the spirit of Kenyon \cite{kenyon} and Bonahon
\cite{bonahon1}, \cite{bonahon2}.

Let $M$ be a compact oriented surface of genus at least two, and
let $\F$ be a minimal oriented measured foliation on $M$.
Denote by ${\bf m}_{\F}$ the transverse invariant measure of $\F$.
If $\gamma=\gamma(t), t\in [0,T]$ is a  smooth curve on $M$, and $s_1, s_2$ satisfy
$0\leq s_1< s_2\leq T$,  then we denote by ${\rm res}_{[s_1, s_2]}\gamma$ the curve
$\gamma(t), t\in [s_1, s_2]$.

Let  $\B_c(\F)$ be the space of uniformly continuous finitely-additive transverse invariant
measures for $\F$. In other words,  a map $\Phi$ which to every
smooth arc $\gamma$ transverse to $\F$ assigns a real number $\Phi(\gamma)$
belongs to the space ${\B}_c(\F)$
if it satisfies the following:
\begin{assumption}
\label{finaddmeasfoliation}
\begin{enumerate}
\item ( finite additivity) For $\gamma=\gamma(t), t\in [0,T]$ and any $s\in (0,T)$, we have
$$
\Phi(\gamma)=\Phi({\rm res}_{[0, s]}\gamma)+\Phi({\rm res}_{[s, T]}\gamma);
$$
\item ( uniform continuity) for any $\varepsilon>0$ there exists $\delta>0$ such that
for any transverse arc $\gamma$ satisfying ${\bf m}_{\F}(\gamma)<\delta$ we have $|\Phi(\gamma)|<\varepsilon$;
\item ( holonomy invariance)  the value $\Phi(\gamma)$ does not change if $\gamma$ is deformed in
such a way that it stays transverse to $\F$ while the endpoints of $\gamma$ stay on their respective leaves.
\end{enumerate}
\end{assumption}
A measure $\Phi\in \B_c(\F)$ is called H{\"o}lder with exponent $\theta$
if there exists $\varepsilon_0>0$ such that for any transverse arc $\gamma$
satisfying ${\bf m}_{\F}(\gamma)<\varepsilon_0$ we have
$$
|\Phi(\gamma)|\leq \left({\bf m}_{\F}(\gamma)\right)^{\theta}.
$$

Let $\B(\F)\subset \B_c(\F)$ be the subspace of H{\"o}lder transverse measures.

As before, we have a natural map
$$
{\mathcal I}_{\F}: \B_c(\F)\to H^1(M, {\mathbb R})
$$
defined as follows.  For a smooth closed curve $\gamma$ on
$M$ and a measure $\Phi\in \B_c(\F)$  the integral
$\int_{\gamma} d\Phi$ is well-defined as the limit of Riemann sums; by holonomy-invariance and continuity of $\Phi$, this operation
descends to homology and assigns to $\Phi$ an element of $H^1(M, {\mathbb R})$.

Now take an abelian differential ${\bf X}=(M, \omega)$ and let $\F^-_{\bf X}$ be its horizontal foliation.
We have a ``tautological'' isomorphism between $\BB_c(\F^-_{\bf X})$ and $\BB_c^+({\bf X})$: every
transverse measure for the horizontal foliation induces a cocycle for the vertical foliation and
vice versa; to a H{\"o}lder measure corresponds a H{\"o}lder cocycle.
For brevity, write ${\mathcal I}_{\bf X}={\mathcal I}_{\F^-_{\bf X}}$.
Denote by $E^u_{\bf X}\subset H^1(M, {\mathbb R})$ the unstable subspace of the
Kontsevich-Zorich cocycle of the abelian differential ${\bf X}=(M, \omega)$.

Theorem \ref{multiplicmoduli} and Proposition \ref{cochyperb} yield the following
\begin{corollary}
\label{finaddmeasfol}
Let $\Prob$ be a Borel probability measure on $\HH$ invariant and ergodic under the action of the Teichm{\"u}ller flow ${\bf g}_t$.
Then for almost every abelian differential ${\bf X}\in\HH$ the map ${\mathcal I}_{\bf X}$ takes $\BB(\F^-_{\bf X})$ isomorphically
onto $E^u_{\bf X}$.

If the Kontsevich-Zorich cocycle does not have zero Lyapunov exponents with respect to $\Prob$,
then for almost all ${\bf X}\in\HH$ we have $\B_c(\F_{\bf X})=\B(\F_{\bf X})$.
\end{corollary}

In other words, in the absence of zero Lyapunov exponents
all continuous transverse finitely-additive invariant measures are in fact H{\"o}lder.

{\bf Remark.} As before, the condition of the absence of zero Lyapunov exponents can be weakened:
it suffices to require that the Kontsevich-Zorich cocycle act isometrically on the Oseledets
subspace corresponding to the Lyapunov exponent zero.

By definition, the space $\BB(\F^-_{\bf X})$ only depends on the horizontal foliation
of our abelian differential; so does $E^u_{\bf X}$.

\subsection{Finitely-additive invariant measures for interval exchange transformations.}
\subsubsection{The space of invariant continuous finitely-additive measures.}

Let $m\in {\mathbb N}$. Let $\Delta_{m-1}$ be the standard unit simplex
$$
\Delta_{m-1}=\{\la\in {\mathbb R}^m_+,\ \la=(\la_1, \dots,\la_m),\la_i>0, \sum\limits_{i=1}^m \la_i=1\}.
$$

Let $\pi$ be a permutation of $\{1, \dots, m\}$ satisfying the {\it irreducibility} condition:
we have $\pi\{1, \dots, k\}=\{1, \dots, k\}$ if and only if $k=m$.

On the half-open interval $I=[0, 1)$ consider
the points
$$\beta_1=0,\quad \beta_i=\sum_{j<i}\la_j,\quad \beta_1^{\pi}=0,\quad \beta_i^{\pi}=\sum_{j<i}\la_{\pi^{-1}j}$$
and
denote $I_i=[\beta_i, \beta_{i+1})$, $I_i^{\pi}=[\beta_i^{\pi}, \beta_{i+1}^{\pi})$.
The length of $I_i$ is $\la_i$, while the length of $I_i^{\pi}$  is $\la_{\pi^{-1}i}$.
Set
$$
{\bf T}_{(\la,\pi)}(x)=x+\beta_{\pi i}^{\pi}-\beta_i {\rm \ for\ } x\in I_i.
$$
The map ${\bf T}_{(\la,\pi)}$ is called {\it an interval exchange transformation} corresponding to $(\la, \pi)$.
By definition, the map ${\bf T}_{(\la,\pi)}$ is invertible and preserves the Lebesgue measure on $I$.
By the theorem of Masur \cite{masur} and Veech \cite{veech}, for any irreducible permutation $\pi$ and
for Lebesgue-almost all $\la\in\Delta_{m-1}$, the corresponding interval exchange transformation ${\bf T}_{(\la,\pi)}$
is uniquely ergodic: the Lebesgue measure is the only invariant probability measure for ${\bf T}_{(\la,\pi)}$.

Consider the space of complex-valued continuous finitely-additive invariant measures for ${\bf T}_{(\la,\pi)}$.

More precisely, let $\BB_c({\bf T}_{(\la,\pi)})$ be the space of all continuous functions $\Phi: [0,1]\to {\mathbb R}$
satisfying
\begin{enumerate}
\item  $\Phi(0)=0$;
\item if $0\leq t_1\leq t_2<1$ and ${\bf T}_{(\la,\pi)}$ is continuous on $[t_1, t_2]$, then
$\Phi(t_1)-\Phi(t_2)=\Phi({\bf T}_{(\la,\pi)}(t_1))-\Phi({\bf T}_{(\la,\pi)}(t_2))$.
\end{enumerate}

Each function $\Phi$ induces a finitely-additive measure on $[0,1]$ defined on the semi-ring
of subintervals (for instance, the function $\Phi_1(t)=t$ yields the Lebesgue measure on $[0,1]$).

Let $\BB({\bf T}_{(\la,\pi)})$ be the subspace of H{\"o}lder functions $\Phi\in \BB_c({\bf T}_{(\la,\pi)})$.

The classification of H{\"o}lder cocycles over translation flows and the asymptotic formula of
Theorem \ref{multiplicmoduli} now
yield the classification of the space $\BB({\bf T}_{(\la,\pi)})$ and
an asymptotic expansion for time averages of almost all interval exchange maps.

\subsubsection{The approximation of ergodic sums}

Let ${\bf X} = (M,\omega)$ be an abelian differential, and let $I\subset M$ be a closed interval lying on a leaf of a horizontal foliation. The vertical flow $h^+_t$ induces an interval exchange map $\mathbf{T}_I$ on $I$, namely, the Poincar\'{e} first return map of the flow. By definition,  there is a natural tautological identification of the spaces $\mathfrak{B}_c(\mathbf{T}_I)$ and $\mathfrak{B}_c^-(\mathbf{X})$, as well as of the spaces $\mathfrak{B}(\mathbf{T}_I)$ and $\mathfrak{B}^-(\mathbf{X})$.

For $x\in M$, let $\tau_I(x)=\min\left\{t\geqslant0:h^+_{-t}x\in I\right\}$. Note that the function $\tau_I$ is uniformly bounded on $M$. Now take a Lipschitz function $f$ on $I$, and introduce a function $\widetilde{f}$ on $M$ by the formula
$$\widetilde{f}(x)=\frac{f(h^+_{-\tau_I(x)} x)}{\tau_I(x)}$$
(setting $\widetilde{f}(x)=0$ for points at which $\tau_I$ is not defined).

By definition, the function $\widetilde{f}$ is weakly Lipschitz, and Theorem \ref{multiplicmoduli} is applicable to $\widetilde{f}$.

The ergodic integrals of $\widetilde{f}$ under $h^+_t$ are of course closely related to ergodic sums of $f$ under $\mathbf{T}_I$, and for any $N\in\mathbb{N}$, $x\in I$, there exists a time $t(x,N)\in\mathbb{R}$ such that
$$\int\limits^{t(x,N)}_0\widetilde{f}\circ h^+_s(x)\,ds=\sum\limits^{N-1}_{k=0}f\circ \mathbf{T}_I^k(x)\:.$$

By the Birkhoff--Khintchine Ergodic Theorem  we have
$$\lim\limits_{N\rightarrow\infty}\frac{t(x,N)}{N}=\frac{1}{Leb(I)}\:,$$
where $Leb(I)$ stands for the length of $I$.

Furthermore, Theorem \ref{multiplicmoduli} yields the existence of constants $C(I)>0$, $\theta\in (0,1)$, such that for all $x\in I$, $N\in\mathbb{N}$, we have
\begin{equation}\label{txn-est}
\left|t(x,N)-\frac{N}{Leb(I)}\right|\leqslant C(I)\cdot N^{\theta}.
\end{equation}

Indeed, the interval $I$ induces a decomposition of our surface into weakly admissible rectangles $\Pi_1,\ldots,\Pi_m$; denote by $h_i$ the height of the rectangle $\Pi_i$, and introduce a weakly Lipschitz function that assumes the constant value $\frac{1}{h_i}$ on each rectangle $\Pi_i$.
Applying Theorem \ref{multiplicmoduli}
 to this function we arrive at desired estimate.

In view of the estimate (\ref{txn-est}), Theorem \ref{multiplicmoduli} applied to the function $\widetilde{f}$ now yields the following Corollary.

\begin{corollary} \label{iet-approx}
Let $\mathbb{P}$ be a
${\bf{g}}_s-invariant$ ergodic probability measure on
${\mathcal{H}}.$ For any $\varepsilon>0$ there exists
$C_{\varepsilon}>0$ depending only on $\mathbb{P}$ such that the
following holds.

For almost every abelian differential
${\bf{X}}\in{\mathcal{H}},{\bf{X}}=(M,\omega),$ any horizontal
closed interval $I\subset M$, any Lipschitz function
$f:I\to\mathbb{R}$, any $x\in I$ and all $n\in\mathbb{N}$ we have
$$\left|\sum_{k=0}^{N-1}f\circ{\bf{T}}_I^k(x)-\Phi_{\widetilde{f}}^+(x,N)\right|
\leqslant C_{\varepsilon}||f||_{Lip}N^{\varepsilon}.$$
\end{corollary}

{\bf Proof.} Applying Theorem \ref{multiplicmoduli} to $\widetilde{f}$, using the estimate
$(\ref{txn-est})$ and noting that the weakly Lipschitz norm of
$\widetilde{f}$ is majorated by the Lipschitz norm of $f,$ we
obtain the desired Corollary.

Let $\theta_1>\theta_2>\ldots>\theta_{l_0}>0$ be the distinct positive
Lyapunov exponents of the measure $\mathbb{P},$ and let
$d_1=1,d_2,\ldots,d_{l_0}$ be the dimensions of the corresponding
subspaces. The tautological identification of
$\mathfrak{B}({\bf{T}}_{I})$ and $\mathfrak{B}^-({\bf{X}})$
together with the results of the previous Corollary now implies
Zorich-type estimates for the growth of ergodic sums of
${\bf{T}}_I.$ More precisely, we have the following

\begin{corollary} \label{iet-logasympt}
In the assumptions of the
preceding Corollary, the space $\mathfrak{B}({\bf{T}}_I)$ admits a
flag of subspaces
$$0=\mathfrak{B}_0\subset\mathfrak{B}_1=\mathbb{R}Leb_I\subset\mathfrak{B}_2
\subset\ldots\subset\mathfrak{B}_{l_0}=\mathfrak{B}({\bf{T}}_I)$$
such that any finitely-additive measure $\Phi\in\mathfrak{B}_i$ in
H\"{o}lder with exponent $\frac{\theta_i}{\theta_1}-\varepsilon$
for any $\varepsilon>0$ and that for any Lipschitz function
$f:I\to\mathbb{R}$ and for any $x\in I$  we have
$$\lim_{N\to\infty}\sup\frac{\log\left|\sum^{N-1}_{k=0}f{\bf{T}}_I^k(x)\right|}{\log
N}=\frac{\theta_{i(f)}}{\theta_1},$$ where $i(f)=1+\max\{j:\int_I fd\Phi=0$ for all
$\Phi\in\mathfrak{B}_j\}$ and by convention we set
$\theta_{l_0+1}=0.$

If with respect to the measure $\Prob$ the  Kontsevich-Zorich cocycle acts isometrically on its neutral
subspaces, then we also have
$\mathfrak{B}_c({\bf{T}}_I)=\mathfrak{B}({\bf{T}}_I).$
\end{corollary}

{\bf{Remark}}.
Corollaries \ref{iet-approx}, \ref{iet-logasympt} thus yield the asymptotic expansion
in terms of H\"{o}lder cocycles as well as Zorich-type logarithmic
estimates for almost all interval exchange transformations with
respect to any conservative ergodic measure $\mu$ on the space of
interval exchange transformations, invariant under the Rauzy-Veech
induction map and such that the Kontsevich-Zorich cocycle is
log-integrable with respect to $\mu.$

In particular, for the Lebesgue measure, if we
let $\R$ be the Rauzy class of the permutation $\pi$, then, using the simplicity of
the Lyapunov spectrum given by the Avila-Viana theorem \cite{AV}, we obtain
\begin{corollary}
For any irreducible permutation $\pi$ and for Lebesgue-almost all $\la$
all continuous finitely-additive measures are H{\"o}lder: we have
 $$\BB({\bf T}_{(\la,\pi)})=\BB_c({\bf T}_{(\la,\pi)}).$$
For any irreducible permutation $\pi$ there exists
a natural number $\rho=\rho(\R)$
depending only on the Rauzy class of $\pi$ and
such that
\begin{enumerate}
\item for Lebesgue-almost all $\la$ we have $\dim \BB(\la,\pi)=\rho$;
\item  all the spaces ${\mathfrak B}_i$ are one-dimensional and $l_0=\rho$.
\end{enumerate}
\end{corollary}

The second statement of Corollary \ref{iet-logasympt} recovers, in the case of the Lebesgue measure on the space of interval exchange transformations, the  Zorich logarithmic asymptotics for ergodic sums \cite{Z}, \cite{zorichwind}.

{\bf Remark.} Objects related to finitely-additive measures for interval exchange
transformations have been studied by X. Bressaud, P. Hubert and A. Maass in \cite{bhm} and by
S. Marmi, P. Moussa and J.-C. Yoccoz in \cite{MMY2}. In particular, the ``limit shapes''
of \cite{MMY2} can be viewed as graphs of  the cocycles $\Phi^+(x,t)$ considered as  functions in $t$.

\subsection{Limit theorems for translation flows.}
\subsubsection{Time integrals as random variables.}
As before, $(M, \omega)$ is an abelian differential, and
$h_t^+$, $h_t^-$ are, respectively, its vertical and horizontal flows. Take
$\tau\in [0,1]$, $s\in {\mathbb R}$, a real-valued
$f\in Lip_{w,0}^+(M, \omega)$ and introduce the function
\begin{equation}
\label{sfstaux}
{\mathfrak S}[f,s;\tau, x]=\int_0^{\tau\exp(s)} f\circ h^{+}_t(x)dt.
\end{equation}

For fixed $f$, $s$ and $x$ the quantity ${\mathfrak S}[f,s;\tau, x]$
is a continuous function of $\tau\in [0,1]$; therefore, as $x$ varies
in the probability space $(M, \nu)$, we obtain a random element
of $C[0,1]$. In other words, we have a
 random variable
\begin{equation}
{\mathfrak S}[f,s]: (M, \nu)\to C[0,1]
\end{equation}
defined by the formula (\ref{sfstaux}).

For any fixed $\tau\in [0,1]$  the formula (\ref{sfstaux}) yields
a real-valued random variable
\begin{equation}
{\mathfrak S}[f,s; \tau]: (M, \nu)\to {\mathbb R},
\end{equation}
whose expectation, by definition, is zero.

Our first aim is to estimate the growth of its variance as $s\to\infty$.
Without losing generality, one may take $\tau=1$.

\subsubsection{The growth rate of the variance in the case of a simple second Lyapunov exponent.}

Let $\Prob$ be an invariant ergodic probability measure for the
Teichm{\"u}ller flow such that with respect to $\Prob$
the second Lyapunov exponent $\theta_2$ of
the Kontsevich-Zorich cocycle is positive and simple (recall that, as Veech and Forni
showed, the  first one, $\theta_1=1$, is always simple \cite{V2, F2} and that,
by the Avila-Viana
theorem \cite{AV}, the second one is simple for the Masur-Veech smooth measure).

For an abelian differential ${\bf X}=(M, \omega)$, denote by $E_{2,{\bf X}}^+$ the one-dimensional
subspace in $H^1(M, {\mathbb R})$ corresponding to the second Lyapunov exponent
$\theta_2$, and let ${\mathfrak B}_{2,{\bf X}}^+={\mathcal I}_{\bf X}^+(E_{2,{\bf X}}^+)$.
Similarly, denote by $E_{2,{\bf X}}^-$ the one-dimensional
subspace in $H^1(M, {\mathbb R})$ corresponding to the Lyapunov exponent
$-\theta_2$, and let ${\mathfrak B}_{2,{\bf X}}^-={\mathcal I}_{\bf X}^-(E_{2,{\bf X}}^-)$.

Recall that the space $H^1(M, {\mathbb R})$ is endowed with the Hodge
norm $|\cdot |_H$; the isomorphisms ${\mathcal I}_{\bf X}^{\pm}$ take the Hodge norm to a norm on ${\BB}_{\bf X}^{\pm}$;
slightly abusing notation, we denote the latter norm by the same symbol.

Introduce a multiplicative cocycle $H_2(s,{\bf X})$ over the Teichm{\"u}ller flow ${\bf g}_s$ by
taking $v\in E_{2,{\bf X}}^+$, $v\neq 0$, and setting
\begin{equation}
\label{h2sx}
H_2(s, {\bf X})=\frac{|{\bf A}(s,{\bf X})v|_H}{|v|_H}.
\end{equation}

The Hodge norm is chosen only for concreteness in (\ref{h2sx}); any other norm can be
used instead.

By definition, we have
\begin{equation}
\lim\limits_{s\to\infty} \frac{\log H_2(s,{\bf X})}{s}=\theta_2.
\end{equation}

Now take $\Phi^+_2\in {\mathfrak B}^+_{2,{\bf X}}$ $\Phi^-_2\in {\mathfrak B}^-_{2,{\bf X}}$ in such a way that
$\langle\Phi_2^+, \Phi_2^-\rangle=1$.

\begin{proposition}
\label{varestfmoduli}
There exists $\alpha>0$ depending only on $\Prob$ and positive measurable  functions
$$
C: \HH\times \HH\to {\mathbb R}_+, \
V: \HH\to {\mathbb R}_+,\
s_0:\HH\to{\mathbb R}_+
$$
such that the following is true for $\Prob$-almost all ${\bf X}\in\HH$.

If $f\in Lip_{w,0}^+({\bf X})$ satisfies $m_{\Phi_2^-}(f)\neq 0$, then for all $s\geq s_0({\bf X})$ we have                                                                                                                                 \begin{equation}\label{varestest}                                                                                                                     \left|\frac{Var_{\nu} {\mathfrak S}(f,x,e^s)}{V({\bf g}_s{\bf X}) (m_{\Phi_2^-}(f)|\Phi_2^+|H_2(s, {\bf X}))^2}-1\right|\leq
C({\bf X}, {\bf g}_s{\bf X})\exp(-\alpha s).
\end{equation}
\end{proposition}
{\bf Remark.} Observe that the quantity $(m_{\Phi_2^-}(f)|\Phi_2^+|)^2$ does not depend on the specific choice of $\Phi_2^+\in\BB^+_{2, {\bf X}}$,
$\Phi_2^-\in\BB^-_{2, {\bf X}}$ such that $\langle \Phi_2^+, \Phi_2^-\rangle=1$.

{\bf Remark.} Note that by theorems of Egorov and Luzin, the estimate (\ref{varestest})
holds {\it uniformly} on compact subsets of $\HH$ of probability arbitrarily close to $1$.

Proposition \ref{varestfmoduli} is based on
\begin{proposition}
There exists  a  positive measurable function $V: \HH\to {\mathbb R}_+$
such that for $\Prob$-almost all ${\bf X}\in\HH$, we have
\begin{equation}
{Var_{\nu({\bf X})} {\Phi_2^+}(x,e^s)}=V({\bf g}_s{\bf X})|\Phi_2^+|^2(H_2(s, {\bf X}))^2.
\end{equation}
\end{proposition}

In particular ${Var_{\nu} {\Phi_2^+}(x,e^s)}\neq 0$ for any $s\in {\mathbb R}$.
The function $V({\bf X})$ is given by
$$
V({\bf X})=\frac{{Var_{\nu({\bf X})} {\Phi_2^+}(x,1)}}{|\Phi^+_2|^2}.
$$
Observe that the right-hand side does not depend on  a particular choice of $\Phi_2^+\in {\BB}_{2, {\bf X}}^+$, $\Phi_2^+\neq 0$.

\subsubsection{The limit theorem in the case of a simple second Lyapunov exponent.}

Go back to the $C[0,1]$-valued random variable ${\mathfrak S}[f,s]$ and denote by
$\mm[f,s]$ the distribution of the normalized random variable
\begin{equation}
\frac{{\mathfrak S}[f,s]}{{\sqrt{Var_{\nu} {\mathfrak S}[f,s;1]}}}.
\end{equation}
The measure $\mm[f,s]$ is thus a probability distribution on the space $C[0,1]$ of continuous functions on the unit interval.

For $\tau\in {\mathbb R}$, $\tau\neq 0$, we also let $\mm[f,s; \tau]$ be the distribution of the ${\mathbb R}$-valued random variable
\begin{equation}
\frac{{\mathfrak S}[f,s; \tau]}{{\sqrt{Var_{\nu} {\mathfrak S}[f,s; \tau]}}}.
\end{equation}

If $f$ has zero average, then, by definition,  $\mm[f,s; \tau]$ is a measure on ${\mathbb R}$ of expectation $0$ and variance $1$.

By definition, $\mm[f,s]$ is a Borel probability measure on $C[0,1]$; furthermore, if
$\xi=\xi(\tau)\in C[0,1]$, then the following natural normalization requirements hold for $\xi$ with respect to  $\mm[f,s]$:
\begin{enumerate}
\item $\xi(0)=0$ almost surely with respect to $m[f,s]$;
\item $\ee_{\mm[f,s]}\xi(\tau)=0$ for all $\tau\in [0,1]$;
\item $Var_{\mm[f,s]}\xi(1)=1$.
\end{enumerate}

We are interested in the weak accumulation points of $\mm[f,s]$ as $s\to\infty$.

Consider the space $\HH^{\prime}$ given by
the formula
$$
\HH^{\prime}=\{{\bf X}^{\prime}=(M, \omega, v), v\in E_2^+(M, \omega), |v|_H=1\}.
$$
By definition, the space $\HH^{\prime}$ is a $\prob$-almost surely defined
two-to-one cover of the space $\HH$. The skew-product flow of the Kontsevich-Zorich cocycle
over the Teichm{\"u}ller
flow yields a flow ${\bf g}_s^{\prime}$ on $\HH^{\prime}$ given by the formula
$$
{\bf g}_s^{\prime}({\bf X}, v)=({\bf g}_s{\bf X}, \frac{{\bf A}(s,{\bf X})v}{|{\bf A}(s,{\bf X})v|_H}).
$$
Given ${\bf X}^{\prime}\in \HH^{\prime}$, set
$$
\Phi_{2,{\bf X}^{\prime}}^+={\mathcal I}^+(v).
$$

Take ${\tilde v}\in E_2^-(M, \omega)$ such that $\langle v, {\tilde v}\rangle=1$ and set
$$
\Phi_{2,{\bf X}^{\prime}}^-={\mathcal I}^-(v), \ m_{2,{\bf X}^{\prime}}^-=m^-_{\Phi_{2,{\bf X}^{\prime}}}.
$$

Let $\MM$ be the space of all probability distributions on $C[0,1]$ and
introduce a $\Prob$-almost surely defined map
$
{\mathcal D}_2^+: \HH^{\prime}\to \MM
$
by setting  ${\mathcal D}_2^+({\bf X}^{\prime})$ to be the distribution
of the $C[0,1]$-valued normalized random variable
$$
\frac{\Phi^+_{2,{\bf X}^{\prime}}(x, \tau)}{\sqrt{Var_{\nu}\Phi^+_{2,{\bf X}^{\prime}}(x, 1)}}, \  \tau\in[0,1].
$$

By definition, ${\mathcal D}_2^+({\bf X}^{\prime})$ is a Borel
probability measure on the space $C[0,1]$; it is, besides,
a compactly supported measure as its support
consists of equibounded H{\"o}lder functions with exponent
$\theta_2/\theta_1-\varepsilon$.

Consider the set
$\MM_1$ of probability measures $\mm$ on $C[0,1]$ satisfying, for $\xi\in C[0,1]$,
$\xi=\xi(t)$,  the conditions:
\begin{enumerate}
\item the equality $\xi(0)=0$ holds $\mm$-almost surely;

\item for all $\tau$ we have $\ee_{\mm}\xi(\tau)=0$:

\item we have $Var_{\mm}\xi(1)=1$ and for any $\tau\neq 0$ we have $Var_{\mm}\xi(\tau)\neq 0$.
\end{enumerate}

It will be proved in what follows that ${\mathcal D}_2^+(\HH^{\prime})\subset \MM_1$.

Consider a semi-flow $J_s$ on the space $C[0,1]$
defined by the formula
$$
J_s\xi(t)=\xi(e^{-s}t), \ s\geq 0.
$$

Introduce a semi-flow $G_s$ on $\MM_1$ by the formula
\begin{equation}
G_s\mm=\frac{(J_s)_*\mm}{Var_{\mm}(\xi(\exp(-s))}, \mm\in \MM_1.
\end{equation}

By definition, the diagram
$$
\begin{CD}
\label{Dtwo}
\HH^{\prime}@ >{\mathcal D}_2^+>> {\MM_1} \\
@VV{\bf g}_sV           @AAG_sA   \\
\HH^{\prime}@ >{\mathcal D}_2^+>> {\MM_1} \\
\end{CD}
$$
is commutative.

Let $d_{LP}$ be the L{\'e}vy-Prohorov metric and let $d_{KR}$ be the Kantorovich-Rubinstein metric
on the space of probability measures on $C[0,1]$ (see \cite{billingsley}, \cite{bogachev} and the Appendix).

We are now ready to formulate the main result of this subsection.

\begin{proposition}
\label{limthmmoduli-simple}
Let $\Prob$ be a ${\bf g}_s$-invariant ergodic probability measure on ${\HH}$
such that the second Lyapunov exponent of the Kontsevich-Zorich
cocycle is positive and simple with respect to $\Prob$.

There exists a positive measurable function $C: \HH\times \HH\to {\mathbb R}_+$
and a positive constant $\alpha$ depending only on $\Prob$
such that
for $\Prob$-almost every ${\bf X}^{\prime}\in\HH^{\prime}$, ${\bf X}^{\prime}=({\bf X}, v)$,
and any $f\in Lip_{w,0}^+({\bf X})$
satisfying
$m_{2, {\bf X}^{\prime}}^-(f)>0$ we have
\begin{equation}
\label{lpexpmoduli}
d_{LP}(\mm[f, s], {\mathcal D}_2^+({\bf g}_s^{\prime}{\bf X}^{\prime}))\leq
C({\bf X}, {\bf g}_s{\bf X})\exp(-\alpha s).
\end{equation}
\begin{equation}
\label{krexpmoduli}
d_{KR}(\mm[f, s], {\mathcal D}_2^+({\bf g}_s^{\prime}{\bf X}^{\prime}))\leq
C({\bf X}, {\bf g}_s{\bf X})\exp(-\alpha s).
\end{equation}
\end{proposition}

Now fix $\tau\in {\mathbb R}$ and let ${\mm}_2({\bf X}^{\prime}, \tau)$ be the distribution
of the ${\mathbb R}$-valued random variable
$$
\frac{\Phi^+_{2,{\bf X}^{\prime}}(x, \tau)}{\sqrt{Var_{\nu}\Phi^+_{2,{\bf X}^{\prime}}(x, \tau)}}.
$$
For brevity, write ${\mm}_2({\bf X}^{\prime}, 1)={\mm}_2({\bf X}^{\prime})$.
\begin{proposition}
For $\Prob$-almost any ${\bf X}^{\prime}\in\HH^{\prime}$, the measure
${\mm}_2({\bf X}^{\prime}, \tau)$ admits atoms for a dense set of times $\tau\in {\mathbb R}$.
\end{proposition}

A more general proposition on the existence of atoms will be formulated in the following subsection.

Proposition \ref{limthmmoduli-simple} implies that the omega-limit set of the family $\mm[f, s]$
can generically assume at most two values. More precisely, the ergodic measure $\Prob$ on $\HH$ is
naturally lifted to its ``double cover''
on the space $\HH^{\prime}$: each point in the fibre is assigned equal weight;
the resulting measure is denoted  $\Prob^{\prime}$. By definition,
the measure $\Prob^{\prime}$  has no more than two ergodic components. We therefore arrive
at the following
\begin{corollary}\label{nonentwo}
Let $\Prob$ be a ${\bf g}_s$-invariant ergodic probability measure on ${\HH}$
such that the second Lyapunov exponent of the Kontsevich-Zorich
cocycle is positive and simple with respect to $\Prob$.

There exist two closed sets ${\mathfrak N}_1, {\mathfrak N}_2\subset {\mathfrak M}$ such that
for $\Prob$-almost every ${\bf X}\in {\mathcal H}$ and any $f\in Lip_{w,0}^+({\bf X})$
satisfying $\Phi_f^+\neq 0$ the omega-limit set of the family  $\mm[f, s]$ either
coincides with ${\mathfrak N}_1$ or with ${\mathfrak N}_2$.
If, additionally, the measure $\Prob^{\prime}$ is ergodic, then ${\mathfrak N}_1={\mathfrak N}_2$.
\end{corollary}

{\bf Question}. Do the sets ${\mathfrak N}_i$ contain measures with non-compact support?

For horocycle flows on compact surfaces of constant negative curvature, compactness of support
 for all weak accumulation points of ergodic integrals has been obtained by Flaminio and Forni \cite{flafo}.

{\bf Question}. Is the measure $\Prob^{\prime}$ ergodic when ${\Prob}$  is the Masur-Veech measure?

As we shall see in the next subsection, in general, the omega-limit sets of the distributions of
the ${\mathbb R}$-valued random variables ${\mathfrak S}[f,s;1]$ contain the delta-measure at zero.
As a consequence, it will develop that, under certain assumptions on the measure $\Prob$, which are satisfied, in particular, for the Masur-Veech smooth measure, for a generic abelian differential the random variables ${\mathfrak S}[f,s;1]$
{\it do not converge} in distribution, as $s\to\infty$, for any
function $f\in Lip_{w,0}$ such that $\Phi_f^+\neq 0$.

\subsubsection{The general case}

While, by the Avila-Viana Theorem \cite{AV}, the Lyapunov spectrum of
the Masur-Veech measure is simple, there are also natural examples of invariant measures with non-simple positive second
Lyapunov  exponent due to Eskin-Kontsevich-Zorich \cite{EKZ}, G. Forni \cite{forni-nonunif},
C. Matheus (see Appendix A.1 in \cite{forni-nonunif}).
A slightly more elaborate, but similar, construction is needed to obtain limit theorems in this  general case.

Let $\mathbb{P}$ be an invariant ergodic probability measure for
the Teichm\"{u}ller flow, and let
$$\theta_1=1>\theta_2>\ldots>\theta_{l_0}>0$$ be
the distinct positive Lyapunov exponents of the Kontsevich-Zorich cocycle
with respect to $\mathbb{P}.$ We assume $l_0\geq 2$.

As before, for ${\bf X}\in {\mathcal H}$ and $i=2, \dots, l_0$, let
$E_i^u({\bf X})$ be the corresponding  Oseledets subspaces, and let
$\mathfrak{B}_{i}^+({\bf X})$ be the corresponding spaces of cocycles.
To make notation lighter, we omit the symbol ${\bf X}$ when the abelian
differential is held fixed.

For $f\in Lip_{w}^+({\bf{X}})$ we now write
$$\Phi_f^+=\Phi_{1,f}^++\Phi_{2,f}^++\ldots +\Phi_{l_0,f}^+,$$
with $\Phi_{i,f}^+\in\mathfrak{B}_{i}^+$ and,
of course, with
$$\Phi^+_{1,f}=(\int\limits_{M}fd\nu)\cdot\nu^+,$$
where $\nu^+$ is the Lebesgue measure on the vertical foliation.

For each $i=2,\ldots,l_0$ introduce a measurable fibre bundle
$${\bf{S}}^{(i)}{\mathcal H}=\{({\bf X},v):{\bf X}
\in{\mathcal H},v\in E_{i}^+,|v|=1\}.$$

The flow ${\bf g}_s$ is naturally lifted to the space
${\bf{S}}^{(i)}\HH$ by the formula
$${\bf g}_s^{{\bf{S}}^{(i)}}({\bf X},v)=\left({\bf g}_s {\bf X},
\frac{\mathbf{A}(s,{\bf X})v}
{|\mathbf{A}(s,{\bf X})v|}\right).$$ The growth of the norm of
vectors $v\in E_i^+$ is controlled by the multiplicative cocycle
$H_i$ over the flow ${\bf g}_s^{{\bf{S}}^{(i)}}$ defined by the formula
$$H_i(s,({\bf X},v))=\frac{\mathbf{A}(s,{\bf X})v}
{|v|}.$$

For
${\bf X}\in {\mathcal H}$ and $f\in
Lip_{w,0}^+({\bf X})$ satisfying $\Phi_f^+\neq 0$, denote
$$
i(f)=\min\{j:\Phi_{f,j}^+\neq0\}.
$$

Define
$v_{f}\in E_{i(f)}^u$ by the formula
$$\mathcal{I}_{{\bf X}}^+(v_{f})=\frac{\Phi_{f,i(f)}^+}{|\Phi_{f,i(f)}^+|}.$$

The growth of the variance of the ergodic integral of a weakly Lipschitz function $f$ is also,
similarly to the case of the simple second Lyapunov exponent, described by the cocycle $H_{i(f)}$ in the following way.

\begin{proposition} There exists $\alpha>0$
depending only on $\mathbb{P}$ and, for any $i=2,\ldots,l_0,$
positive measurable functions
$$V^{(i)}:{\bf{S}}^{(i)}{\mathcal H}\to\mathbb{R}_+,C^{(i)}:{\mathcal H}
\times{\mathcal H}\to\mathbb{R}_+$$ such that for
$\mathbb{P}$-almost every ${\bf X}\in{\mathcal H}$
the following holds. Let $f\in
Lip_{w,0}^+({\bf X})$ satisfy $\Phi_f^+\neq 0$.

Then for all $s>0$ we have
$$\left|\frac{Var_{\nu}(\mathfrak{S}[f,s;1])}
{V^{(i(f))}(g_s^{{\bf{S}}^{(i(f))}}(\overline{\omega},v_{f}))(H_{i(f)}(s, ({\bf X},v_{f})
))^2}-1\right|\leqslant
C^{(i(f))}({\bf X},{\bf g}_s{\bf X})e^{-\alpha s}.$$
\end{proposition}

We proceed to the formulation and the proof of the limit theorem in
the general case. For $i=2,\ldots,l_0$, introduce the map
$$\mathcal{D}_i^+:{\bf{S}}^{(i)}\HH\to\mathfrak{M}_1$$
by setting $\mathcal{D}_i^+({\bf X},v)$ to be the
distribution of the $C[0,1]$-valued random variable
$$\frac{\Phi^+_v(x,\tau)}{\sqrt{Var_{\nu( \overline{\omega})}
(\Phi_v^+(x,1))}},\tau\in[0,1].$$

As before, we have a commutative diagram
$$
\begin{CD}
\label{Dtwo-diagr}
{{\bf S}^{(i)}}\HH@ >{{\mathcal D}_i^+}>> {\MM_1} \\
@VV{\bf g}_s^{{\bf S}^{(i)}}V           @AAG_sA   \\
{{\bf S}^{(i)}}\HH@ >{{\mathcal D}_i^+}>> {\MM_1} \\
\end{CD}
$$

The measure $\mathfrak{m}[f,s]$, as before, stands for the distribution of the
$C[0,1]$-valued random variable
$$\frac
        {\displaystyle\int\limits_0^ {\tau
\exp(s)} f\circ h_t^+(x)dt}
{\sqrt
      {Var_ {\nu}(\displaystyle\int\limits_0^ {
\exp(s)}f\circ h_t^+(x)dt)}},\tau\in[0,1].$$

\begin{theorem}\label{limthmmoduli}
Let $\mathbb{P}$ be an invariant ergodic probability measure for
the Teichm\"{u}ller flow such that the Kontsevich-Zorich cocycle
admits at least two distinct positive Lyapunov exponents with respect to ${\mathbb P}$.
There exists a constant $\alpha>0$ depending only on $\mathbb{P}$
and a positive measurable map
$C:{\mathcal H}\times{\mathcal H}\to\mathbb{R}_+$ such
that for $\mathbb{P}-almost$ every
${\bf X}\in{\mathcal H}$ and any $f\in
Lip_w^+(\bf{X})$ we have
$$d_{LP}(\mathfrak{m}[f,s],{\mathcal D}_{i(f)}^+(g_s^{{\bf{S}}^{(i(f))}}({\bf X},v_f)))
\leqslant C({\bf X},{\bf g}_s {\bf X})e^{-\alpha
s},$$
$$d_{KR}(\mathfrak{m}[f,s],{\mathcal D}_{i(f)}^+(g_s^{{\bf{S}}^{(i(f))}}({\bf X},v_f)))
\leqslant C({\bf X},{\bf g}_s{\bf X})e^{-\alpha
s}.$$
\end{theorem}

\subsubsection{ Atoms of limit distributions.}

For $\Phi^+\in {\mathfrak B}^+({\bf X})$, let
$\mathfrak{m}[\Phi^+,\tau]$ be the distribution of the
$\mathbb{R}$-valued random variable
$$\frac{\Phi^{+}(x,\tau)}{\sqrt{Var_{\nu}\Phi^{+}(x,\tau)}}\ .$$

\begin{proposition} For $\mathbb{P}-almost$ every
${\bf X}\in\mathcal{H},$ there exists a dense set
$T_{atom}\subset\mathbb{R}$ such that if $\tau\in T_{atom},$
 then for any $\Phi^+\in {\mathfrak B}^+({\bf X})$, $\Phi^+\neq 0$, the measure
$\mathfrak{m}(\Phi^+,\tau)$ admits atoms.
\end{proposition}

\subsubsection{Nonconvergence in distribution of ergodic integrals.}

Our next aim is to show that along certain subsequences of times the ergodic integrals of translation flows
converge in distribution  to the measure $\delta_0$, the delta-mass at zero. Weak convergence
of probability measures will be denoted by the symbol $\Rightarrow$.

We need the following additional assumption on the measure $\Prob$.
\begin{assumption}
\label{aczero}
For any $\varepsilon>0$ the set of
abelian differentials ${\bf X}=(M, \omega)$ such that  there exists an
admissible rectangle $\Pi(x, t_1, t_2)\subset M$ with
$t_1>1-\varepsilon$, $t_2>1-\varepsilon$
has positive measure with respect to $\Prob$.
\end{assumption}
Of course, this assumption holds for the Masur-Veech smooth measure.

\begin{proposition} Let $\mathbb{P}$ be an ergodic
${\bf g}_s$-invariant measure on $\mathcal{H}$ satisfying
Assumption \ref{aczero}. Then for $\mathbb{P}$-almost
every ${\bf{X}}\in\mathcal{H}$ there exists a sequence
$\tau_n\in\mathbb{R}_+$ such that for any
$\Phi^+\in{\mathfrak{B}}^+({\bf{X}})$, $\Phi^+\neq 0$,  we have
$$\mathfrak{m}[\Phi^+, \tau_n]\Rightarrow\delta_0\ \mathrm{in} \ {\mathfrak M}({\mathbb R}) \  \mathrm{as}\
n\to\infty.$$
\end{proposition}
Theorem \ref{limthmmoduli} now implies the following
\begin{corollary} Let $\mathbb{P}$ be an ergodic
${\bf g}_s$-invariant measure on $\mathcal{H}$ satisfying
Assumption \ref{aczero}. Then for $\mathbb{P}$-almost
every ${\bf{X}}\in\mathcal{H}$ there exists a sequence
$s_n\in\mathbb{R}_+$ such that for any $f\in
Lip_{w,0}^+({\bf{X}})$ satisfying $\Phi_f^+\neq 0$ we have
$$\mathfrak{m}[f,s_n; 1]\Rightarrow\delta_0\  \mathrm{in} \   {\mathfrak M}({\mathbb R}) \ \mathrm{as}\
n\to\infty.$$

Consequently, if $f\in Lip_{w,0}^+({\bf{X}})$ satisfies
$\Phi_f^+\neq0,$ then the family of measures
$\mathfrak{m}[f,\tau;1] $ does not converge in ${\mathfrak M}({\mathbb R})$
and the family of measures $\mathfrak{m}[f,\tau] $ does not converge in ${\mathfrak M}(C[0,1])$
as $\tau\to\infty.$
\end{corollary}

\subsection{The mapping into cohomology.}
In this subsection we show that for an arbitrary  abelian differential ${\bf X}=(M, \omega)$ the  map
$$
{\check {\mathcal I}}_{\bf X}:\BB_c^+(M, \omega)\to H^1(M, {\mathbb R})
$$
given by (\ref{maptocohomology}) is indeed  well-defined.

\begin{proposition}
\label{acthomology}
Let $\gamma_i$, $i=1, \dots, k$,  be rectangular closed curves such that the cycle $\sum_{i=1}^k \gamma_i$
is homologous to $0$. Then for any $\Phi^+\in\BB_c^+$ we have
$$
\sum_{i=1}^{k} \Phi^+(\gamma_i)=0.
$$
\end{proposition}
Informally, Proposition \ref{acthomology} states that the {\it relative} homology of the surface with respect
to zeros of the form is not needed for the description of cocycles.
Arguments of this type for invariant measures
of translation flows go back to Katok's work \cite{katok}.

We proceed to the formal proof.
Take a fundamental polygon $\Pi$ for $M$ such that all its sides are
simple
closed rectangular curves on $M$. Let $\del\Pi$ be
the boundary of $\Pi$, oriented counterclockwise. By definition,
\begin{equation}
\label{deltapizeroo}
\Phi^+(\del\Pi)=0,
\end{equation}
since each curve of the boundary enters $\del\Pi$ twice and with
opposite signs.

We now deform the curves $\gamma_i$ to the boundary $\del\Pi$ of our fundamental polygon.
\begin{proposition}
\label{simplezero}
Let $\gamma\subset \Pi$ be a simple rectangular closed curve.
Then
$$
\Phi^+(\gamma)=0.
$$
\end{proposition}

Proof of Proposition \ref{simplezero}.

We may assume that $\gamma$ is oriented counterclockwise and does not
contain zeros of the form $\omega$. By Jordan's theorem, $\gamma$ is the boundary of a domain
$N\subset\Pi$. Let $p_1, \dots, p_r$ be zeros of $\omega$ lying inside
$N$; let $\kappa_i$ be the order of $p_i$. Choose an arbitrary
$\varepsilon>0$, take $\delta>0$ such that $|\Phi^+(\gamma)|\leq\varepsilon$ as soon as the length of $\gamma$ does not exceed $\delta$
 and consider a partition of $N$ given by
\begin{equation}
N=\Pi_1^{(\varepsilon)}\bigsqcup\dots\bigsqcup
\Pi_n^{(\varepsilon)}\bigsqcup
{\tilde \Pi}_1^{(\varepsilon)}\bigsqcup\dots\bigsqcup
{\tilde \Pi}_{r}^{(\varepsilon)},
 \end{equation}
where all
$\Pi_i^{(\varepsilon)}$ are admissible rectangles and
${\tilde \Pi}_i^{(\varepsilon)}$ is a $4(\kappa_i+1)$-gon
containing $p_i$ and no other zeros and satisfying the additional
assumption that all its sides are no longer than $\delta$.
Let $\del\Pi_i^{(\varepsilon)}$, $\del{\tilde \Pi}_i^{(\varepsilon)}$
 stand for the boundaries of our polygons oriented counterclockwise.

We have
$$
\Phi^+(\gamma)=\sum \Phi^+(\del\Pi_i^{(\varepsilon)})+
\sum\Phi^+(\del{\tilde \Pi}_i^{(\varepsilon)}).
$$
In the first sum, each term is equal to $0$ by definition of $\Phi^+$,
whereas the second sum does not exceed, in absolute value, the quantity
$$
C(\kappa_1, \dots, \kappa_r)\varepsilon,
$$
where $C(\kappa_1, \dots, \kappa_r)$ is a positive constant depending only
on $\kappa_1, \dots, \kappa_r$.
Since $\varepsilon$ may be chosen arbitrarily small, we have
$$
\Phi^+(\gamma)=0,
$$
which is what we had to prove.

For $A,B\in\del\Pi$, let $\del\Pi_A^B$ be the part of $\del\Pi$ going
counterclockwise from $A$ to $B$.
\begin{proposition}
Let $A,B\in\del\Pi$ and let $\gamma\subset\Pi$ be an
arbitrary rectangular curve going from $A$ to $B$.
Then
$$
\Phi^+(\del\Pi_A^B)=\Phi^+(\gamma).
$$
\end{proposition}

We may assume that $\gamma$ is simple in $\Pi$, since,
by Proposition \ref{simplezero}, self-intersections of $\gamma$ (whose
number is finite) do not change the value of $\Phi^+(\gamma)$. If $\gamma$
is simple, then $\gamma$ and $\Phi^+(\del\Pi_B^A)$ together form a simple
closed curve, and the proposition follows from Proposition
\ref{simplezero}.

\begin{corollary}
If $\gamma\subset \Pi$ is a rectangular curve which yields a closed curve
in $M$
homologous to zero in $M$, then
$$
\Phi^+(\gamma)=0.
$$
\end{corollary}

Indeed, by the previous proposition we need only consider the case when $\gamma\subset
\del\Pi$. Since $\gamma$ is homologous to $0$ by assumption, the cycle $\gamma$ is in fact
a multiple of the cycle $\del\Pi$, for which the statement follows from (\ref{deltapizeroo}).

\subsection{Markovian sequences of partitions.}
\subsubsection{The Markov property}
Let $(M, \omega)$ be an abelian differential.
A rectangle  $\Pi(x, t_1, t_2)=\{h^+_{\tau_1} h^{-}_{\tau_2}x, 0\leq \tau_1< t_1, 0\leq \tau_2< t_2\}$
is called {\it weakly admissible} if for all sufficiently small  $\varepsilon>0$ the rectangle
$\Pi(h_{\varepsilon}^+h_{\varepsilon}^-x, t_1-\varepsilon, t_2-\varepsilon)$ is admissible
(in other words, the boundary of $\Pi$ may contain zeros of $\omega$ but the interior does not).

Assume that we are given a natural number $m$ and  a sequence of partitions $\pi_n$
\begin{equation}
\label{recpart}
\pi_n: M=\Pi_1^{(n)}\sqcup \dots \Pi_m^{(n)}, \ n\in {\mathbb Z},
\end{equation}
where $\Pi_i^{(n)}$ are weakly admissible rectangles.

The sequence $\pi_n$ of partitions of $M$ into $m$ weakly admissible rectangles
will be called a {\it Markovian sequence} of partitions if
for any $n_1, n_2\in {\mathbb Z}$, $i_1,i_2\in \{1,\dots, m\}$, the rectangles
$\Pi_{i_1}^{(n_1)}$ and $\Pi_{i_2}^{(n_2)}$ intersect in a Markov way in the following precise sense.

Take a weakly admissible rectangle $\Pi(x, t_1, t_2)$ and decompose its boundary into four parts:
$$
\partial_h^1(\Pi)={\overline{ \{h^+_{t_1} h^{-}_{\tau_2}x,  0\leq \tau_2< t_2\}}};
$$
$$
\partial_h^0(\Pi)={\overline {\{ h^{-}_{\tau_2}x,  0\leq \tau_2< t_2\}}};
$$
$$
\partial_v^1(\Pi)={\overline {\{h^-_{t_2} h^{+}_{\tau_1}x,  0\leq \tau_1< t_1\}}};
$$
$$
\partial_v^0(\Pi)={\overline {\{h^{+}_{\tau_1}x,  0\leq \tau_1< t_1\}}}.
$$
The sequence of partitions $\pi_n$ has the {\it Markov property} if for any $n\in {\mathbb Z}$ and $i\in \{1, \dots, m\}$
there exist $i_1, i_2, i_3, i_4\in \{1, \dots, m\}$
such that
$$
\partial_h^1(\Pi_i^{(n)})\subset \partial_h^{1}\Pi_{i_1}^{(n-1)};
$$
$$
\partial_h^0(\Pi_i^{(n)})\subset \partial_h^{0}\Pi_{i_2}^{(n-1)};
$$
$$
\partial_v^1(\Pi_i^{(n)})\subset \partial_v^{1}\Pi_{i_3}^{(n+1)};
$$
$$
\partial_v^0(\Pi_i^{(n)})\subset \partial_v^{0}\Pi_{i_4}^{(n+1)}.
$$
\subsubsection{Adjacency matrices.}
To a Markovian sequence of partitions we assign the sequence of $m\times m$
{\it adjacency matrices} $A_n=A(\pi_n, \pi_{n+1})$ defined as follows:
$(A_n)_{ij}$ is the number of connected components of the intersection
$$
(\Pi_i^{(n)})\cap\Pi_{j}^{(n+1)}.
$$
 A Markovian sequence of partitions $\pi_n$ will be called an {\it exact} Markovian
sequence of partitions if
\begin{equation}
\label{minimality}
\lim\limits_{n\to\infty} \max\limits_{i=1, \dots,m} \nu^+(\partial_v\Pi_i^{(n)})=0; \
\lim\limits_{n\to\infty} \max\limits_{i=1, \dots,m} \nu^-(\partial_h\Pi_i^{(-n)})=0.
\end{equation}

For an abelian differential both whose vertical and horizontal flows are minimal, there always exists $m\in {\mathbb N}$
and a sequence of partitions (\ref{recpart}) having the Markov property
and satisfying the exactness condition (\ref{minimality}).
A suitably chosen Markovian sequence of partitions will be essential for
the construction of finitely-additive measures in the following Section.

{\bf Remark.}
An exact Markovian sequence of partitions allows one
to identify our surface $M$ with the space of trajectories of
a non-autonomous Markov chain, or, in other words, a Markov compactum.
The horizontal and vertical foliations then become the asymptotic foliations of the
corresponding Markov compactum; the finitely-additive measures become
finitely-additive measures on one of the asymptotic foliations invariant under holonomy with respect to
the complementary foliation;
the vertical and horizontal flow  also admit a purely symbolic description
as flows along the leaves of the asymptotic foliations according to an order
induced by a Vershik's ordering on the edges of the graphs forming the Markov compactum.
The space of abelian differentials, or, more precisely,
the Veech space of zippered rectangles,
is then represented  as a subspace of the space of Markov compacta.
The space of Markov compacta is a space of bi-infinite sequences
of graphs and is therefore endowed with a natural shift
transformation.  Using Rauzy-Veech expansions of zippered rectangles, one
represents the Teichm{\"u}ller flow as a suspension flow over this shift.
The Kontsevich-Zorich cocycle
is then a particular case of the cocycle over the shift given by
consecutive adjacency matrices of the graphs forming
our Markov compactum. To an abelian differential,
random with respect to a probability measure invariant under
the Teichm{\"u}ller flow, one can thus assign a Markov compactum
corresponding to a sequence of graphs generated according to a stationary
process. The relation between Markov
compacta and abelian differentials is summarized in the table below.

The main theorems of this paper, Theorem \ref{multiplicmoduli} and Theorem \ref{limthmmoduli},
are particular cases of general theorems on the asymptotic behaviour of ergodic averages
of symbolic flows along asymptotic foliations of random Markov compacta;
these generalizations, which will be published in the sequel to this paper,
are proved in the preprint \cite{bufetov-lthm}.

\begin{center}
\begin{tabular}{|p{5.2cm}|p{5.2cm}|}
\hline
\begin{center}\textbf{Markov compacta}\end{center}&\begin{center}\textbf{Abelian differentials}\end{center}\\
\hline
Asymptotic foliations of
a Markov compactum
&
Horizontal and vertical foliations of an abelian differential\\

\hline
Finitely-additive measures on asymptotic foliations  of
a Markov compactum&The spaces $\mathfrak{B}^+$ and $\mathfrak{B}^-$ of H\"{o}lder cocycles\\

\hline
Vershik's automorphisms&Interval exchange transformations\\

\hline
Suspension flows over
Vershik's automorphisms&Translation flows on flat surfaces\\

\hline

The space of  Markov compacta&The moduli space of abelian differentials\\

\hline

The shift on the space of  Markov compacta&The Teichm{\"u}ller flow\\

\hline

The cocycle of adjacency matrices&The Kontsevich-Zorich cocycle\\

\hline
\end{tabular}
\end{center}

{\bf Acknowledgements.}
W. A.~Veech made the suggestion that G. Forni's invariant distributions for the vertical flow
should admit a description in terms of cocycles for the horizontal flow, and I am greatly
indebted to him. G.~Forni pointed out that cocycles are dual objects
to invariant distributions and suggested the analogy with F. Bonahon's work \cite{bonahon1};
I am deeply grateful to him.
I am deeply grateful to A.~Zorich for his kind explanations
of Kontsevich-Zorich theory and for discussions of the relation
between this paper and F.~Bonahon's work \cite{bonahon1}.
I am deeply grateful to H.~Nakada who pointed out the reference
to S.~Ito's work \cite{ito} and to B. ~Solomyak who pointed out the reference
to the work \cite{dumontkamae} of P.~Dumont, T.~Kamae and S.~Takahashi and the work \cite{kamae} of T.~Kamae.

I am deeply grateful to J.~Chaika, P.~Hubert, Yu.S.~Ilyashenko, H.~Kr{\"u}ger,
E.~Lanneau, S.~Mkrtchyan and R.~Ryham for many helpful suggestions on improving the presentation.
I am deeply grateful to A.~Avila, X.~Bressaud, B.M.~Gurevich, A.B.~Katok, A.V.~Klimenko, S.B.~Kuksin, C.~McMullen,  V.I.~Oseledets,
Ya.G.~Sinai, I.V.~Vyugin, J.-C.~Yoccoz for useful discussions. I am deeply grateful to
N. Kozin, D.~Ong, S. Sharahov and R. Ulmaskulov for typesetting
parts of the manuscript. Part of this paper was written while I was visiting the
Max Planck Institute of Mathematics in Bonn and the Institut Henri Poincar{\'e} in Paris.
During the work on this paper,
I was supported in part by an Alfred P. Sloan Research Fellowship, by the
Dynasty Foundation Fellowship and the IUM-Simons Fellowship, by Grant MK-6734.2012.1 of
the President of the Russian Federation,  by the Programme on Dynamical Systems and
Mathematical Control Theory of the
Presidium of the Russian Academy of Sciences, by the Edgar Odell Lovett Fund at Rice University,
by the RFBR-CNRS grant 10-01-93115, by the RFBR grant 11-01-00654 and
by the National Science Foundation under grant DMS~0604386.

\section{Construction of finitely-additive measures}

\subsection{Equivariant sequences of vectors}

Let ${\mathfrak{A}}$ be a bi-invariant sequence of invertible $m\times m$-matrices with non-negative
entries
$$\mathfrak{A}=(A_n),\ n\in\mathbb{Z}.$$
To a vector $v\in\mathbb{R}^n$ we assign the corresponding
$\mathfrak{A}$-{\it equivariant}
sequence ${\bf v}=(v^{(n)}),\
n\in \mathbb{Z},$ given by the formula
$$v^{(n)}=\begin{cases}A_{n-1}\dots A_0v,&n>0;\\
v,&n=0;\\ A_{-n}^{-1}\dots A_{-1}^{-1}v,&n<0\end{cases}$$

We now consider subspaces in $\mathbb{R}^{m}$ consisting of
vectors $v$
such that the corresponding equivariant subsequence
${\bf v}=v^{(n)}$ decays as $n$ tends to $-\infty$.

More formally,we write
\begin{gather*}
\mathfrak{B}_c^+(\mathfrak{A})=\{v\in\mathbb{R}^m:|v^{(-n)}|\to0
\ as\ n\to+\infty\},\\
\mathfrak{B}^+(\mathfrak{A})=\{v\in\mathbb{R}^m: \text{ there
exists }\ C>0,\alpha >0\text{ such that } |v^{(-n)}|\leqslant C
e^{-\alpha n}
\text{ for all } n\geqslant 0\}.
\end{gather*}

It will sometimes be convenient to identify a vector with the
corresponding equivariant sequence, and, slightly abusing
notation, we shall sometimes say that a given equivariant sequence
belongs to the space $\mathfrak{B}_c^+(\mathfrak{A})$ or $\mathfrak{B}^+(\mathfrak{A})$.

\subsection{A canonical system of arcs corresponding to a Markovian sequence of partitions}

As before, let $(M, \omega)$  be an abelian differential.
By an arc of the vertical flow we mean a set of the form
$\{h_t^+x,0\leqslant t\leqslant t_0\}.$ Such a set will also
sometimes be denoted $[x,x'],$ where $x'=h_{t_0}^+x.$

Let $\pi_n,\ n\in\mathbb{Z},$ be an exact Markovian sequence of
partitions $$\pi_n:
M=\Pi_1^{(n)}\sqcup\dots\sqcup\Pi_m^{(n)},n\in\mathbb{Z},$$ into
weakly admissible rectangles.
Take a rectangle
$\Pi_i^{(n)},n\in\mathbb{Z},i\in\{1,\dots,m\},$ and choose an
arbitrary arc $\gamma_i^{(n)}$ of the vertical foliation going all
the way from the lower boundary of $\Pi_i^{(n)}$ to the
upper  boundary. More formally, take
$x\in\partial_h^{(0)}\Pi_i^{(n)},$ take $t$ such that
$h_t^+x\in\partial_h^{(1)}\Pi_i^{(n)},$ and let $\gamma_i^{(n)}$
be the vertical arc $[x,h_t^+x]$. An arc $\gamma_i^{(n)}$ of this form will be called a {\it Markovian} arc.

The family of arcs $$\gamma_i^{(n)},  n\in\mathbb{Z},i\in\{1,\dots,m\},$$
will be called {\it a canonical system of arcs} assigned to the Markov sequence of partitions $\pi_n$.

Of course, there is freedom in the choice of specific arcs $\gamma_i^{(n)}$, but our constructions
will not depend on the specific choice of a canonical system.

Given a finitely-additive
measure $\Phi^+\in \mathfrak{B}_c^+(M,\omega),$ introduce a
sequence of vectors $v^{(n)}\in {\mathbb R}^m, n\in\mathbb{Z},$
by setting

\begin{equation}
\label{mestoseq}
v_i^{(n)}=\Phi^+(\gamma_i^{(n)}).
\end{equation}

Now let
$\mathfrak{A}=(A_n),n\in\mathbb{Z},$ $A_n=A(\pi_n,\pi_{n+1})$ be
the sequence of adjacency matrices of the sequence of partitions $\pi_n$, and
assume all $A_n$ to be invertible.

By the horizontal holonomy invariance, the value $v_i^{(n)}$ does
not depend on the
specific choice of the arc $\gamma_i^{(n)}$
inside the rectangle $\Pi_i^{(n)}$. Finite additivity of the
measure $\Phi^+$ implies
that the sequence
$v^{(n)},n\in\mathbb{Z}$, is $\mathfrak{A}$-equivariant.

Exactness of the sequence of partitions
$\pi_n,n\in\mathbb{Z}$,
implies that that the equivariant sequence $v^{(n)}$ corresponding
to a finitely-additive
measure
$\Phi^+\in\mathfrak{B}_c^+(X,\omega)$ satisfies
$v^{(0)}\in\mathfrak{B}_c^+(\mathfrak{A}).$

We have therefore obtained a map
$$\eval_0^+:\mathfrak{B}_c^+(X,\omega)\to
\mathfrak{B}_c^+(\mathfrak{A}).$$

It will develop that under certain natural additional assumptions
this map is indeed an isomorphism.

We now take an abelian differential $(M,\omega)$ whose vertical
flow is uniquely ergodic and show that if the heights of the
rectangles $\Pi_i^{(n)}$ decay exponentially as $n\to-\infty$,
then the map $\eval_0^+$ sends $\mathfrak{B}^+(X,\omega)$ to
$\mathfrak{B}^+(\mathfrak{A})$.
We proceed to precise
formulations.

Introduce a sequence $h^{(n)}=(h^{(n)}_1,\dots,h^{(n)}_m),$ by
setting $h^{(n)}_i$ to be the height
of the rectangle
$\Pi_i^{(n)},i=1,\dots,m.$

By the Markov property, the sequence $h^{(n)}$ is $\mathfrak{A}$-equivariant;
unique ergodicity of the vertical flow
and exactness of the sequence $\pi_n,n\in\mathbb{Z}$, imply
that
a positive $\mathfrak{A}$-equivariant sequence is unique up
to scalar multiplication.
By definition and, again, by exactness,
we have $$h^{(0)}\in\mathfrak{B}_c^+(\mathfrak{A}).$$

\begin{proposition}
If $h^{(0)}\in\mathfrak{B}^+(\mathfrak{A}),$ then $$\eval_0^+(\mathfrak{B}^+(X,\omega))\subset
\mathfrak{B}^+(\mathfrak{A})$$.
\end{proposition}Proof.
Let a canonical family of vertical arcs $\gamma_i^{(n)}$ corresponding to the Markovian sequence of partitions $\pi_n$ be chosen
as above. The condition $h^{(0)}\in\mathfrak{B}^+(\mathfrak{A})$
precisely means the existence of constants $C>0,\alpha>0$ such
that for all $n\in\mathbb{N},i=1,\dots,m,$ we have
$\nu^+(\gamma_i^{(n)})\leqslant C e^{-\alpha n}.$

Now if $\Phi^+\in \mathfrak{B}^+(M,\omega),$ then there exists $\theta>0$
such that for all sufficiently large n and all $i=1,\dots,m$ we
have
$$|\Phi^+(\gamma_i^{(n)})|
\leqslant(\nu^+(\gamma_i^{(n)}))^{\theta}.$$

Consequently, $|\Phi^+(\gamma_i^{(n)})|\leqslant
\tilde{C}e^{-\tilde{\alpha}n}$ for some
$\tilde{C}>0,\tilde{\alpha}>0$ and all
$n\in\mathbb{N},i=1,\dots,m,$ which is what we had to prove.

The scheme of the proof of the reverse inclusion can informally be
summarized as follows. We start with an equivariant sequence
$v^{(n)}\in\mathfrak{B}^+(\mathfrak{A}),$ and we wish to recover a
measure $\Phi^+\in\mathfrak{B}^+(X,\omega).$ The equivariant
sequence itself determines the values of the finitely-additive
measure $\Phi^+$ on all {\it Markovian} arcs, that is, arcs going
from the lower horizontal to the upper horizontal boundary of a
rectangle of one of the proposition $\pi_n$.
To extend the measure $\Phi^+$ from Markovian arcs to all vertical
arcs, we approximate an arbitrary arc by Markovian ones (similar approximation lemmas were used by
Forni \cite{F2} and Zorich \cite{Z}).
The
approximating series will be seen to converge because the number
of terms at each stage of the approximation grows at most
sub-exponentially, while the contribution of each term decays
exponentially. For this argument to work, we assume that the norms
of the matrices $A_n$ grow sub-exponentially.

The measure $\Phi^+$ is thus extended to all vertical arcs.
To check the H{\"o}lder property for $\Phi^+$, one needs additionally to
assume that the heights of the Markovian rectangles
$\Pi_i^{(n)}$ decay not faster than exponentially.
More precise Oseledets-type assumptions on the sequence ${\mathfrak A}$
of adjacency matrices are used in order to obtain lower bounds on the
H{\"o}lder exponent for $\Phi^+$ and to derive the logarithmic asymptotics
of the growth of $\Phi^+$ at infinity.
All our assumptions
are verified for Markov sequences of partitions induced by
Rauzy-Veech expansions  of zippered rectangles as soon as one uses
the Veech method of considering expansions corresponding to
occurrences of a fixed renormalization matrix with positive
entries.

\subsection{Strongly biregular sequences of matrices}

A sequence $\mathfrak{A}=(A_n),n\in\mathbb{Z},$ of $m\times
m$-matrices will be called {\it balanced} if all entries of all
matrices $A_n$ are positive, and, furthermore, there exists a
positive constant C such that for any $n\in\mathbb{Z}$ and any
$i_1,j_1,i_2,j_2\in\{1,\dots,m\},$ we have
$$\frac{(A_n)_{i_1j_1}}{(A_n)_{i_2j_2}}<C.$$

A sequence $\mathfrak{A}=(A_n),n\in\mathbb{Z},$ of $m\times
m$-matrices with non-negative entries will be said to have {\it
sub-exponential growth} if for any $\varepsilon>0$ there exists
$C_{\varepsilon}$ such that for all $n\in\mathbb{Z}$ we have
$$\sum_{i,j=1}^{m}(A_n)_{ij}\leqslant C_{\varepsilon}
e^{\varepsilon|n|}.$$

In order to formulate our next group of assumptions, we need to
consider $\mathfrak{A}$-{\it reverse equivariant} sequences of
vectors. To a vector {\~{v}}$\in\mathbb{R}^m$ we assign the
sequence ${\bf{\tilde{v}}}=(\tilde{v}^{(n)}),n\in\mathbb{Z},$
given by the formula
$$\tilde{v}^{(n)}=\begin{cases}(A_n^t)^{-1}\dots(A_1^t)^{-1}\tilde{v}^{(n)},&n>0;\\
       \tilde{v}^{(n)},&n=0;\\
       A_{n+1}^t\dots A_{-1}^tA_0^t\tilde{v}^{(n)},&n<0.\end{cases}$$

By definition, if $\tilde{v}^{(n)}$ is an
$\mathfrak{A}$-equivariant sequence, while $\tilde{v}^{(n)}$ is an
$\mathfrak{A}$-reverse equivariant sequence, then the inner
product $$\left< v^{(n)},\tilde{v}^{(n)}\right>=\sum_{i=1}^{m}
v_i^{(n)}\tilde{v}_i^{(n)}$$ does not depend on $n\in\mathbb{Z}$.
In analogy to the spaces ${\mathfrak B}^+(\mathfrak{A})$ and
${\mathfrak B}_c^+(\mathfrak{A})$, we introduce the spaces
\begin{gather*}
{\mathfrak B}_c^-(\mathfrak{A})=\{\tilde{v}:|\tilde{v}^{(n)}|\to0\ as\ n\to\infty\},\\
{\mathfrak B}^-(\mathfrak{A})=\{\tilde{v}:\ {\mathrm{there}}\ {\mathrm{exists}}\ C>0,\
\alpha>0,\ {\mathrm{such}}\ {\mathrm{that}}\ |\tilde{v}^{(n)}|\leqslant Ce^{-\alpha n}\
{\mathrm{as}}\ n\to\infty\}.
\end{gather*}

The unique ergodicity of the vertical and the horizontal flow
admits the following reformulation in terms of the spaces
${\mathfrak B}_c^+(\mathfrak{A})$, ${\mathfrak B}_c^-(\mathfrak{A})$.

\begin{assumption}\label{a-u-e}. The space
${\mathfrak B}_c^+(\mathfrak{A})$ contains an equivariant sequence
$(h^{(n)}),n\in\mathbb{Z}$ such that $h^{(n)}_i>0$ for all
$n\in\mathbb{Z}$ and all $i\in\{1,\dots,m\}.$ The space
${\mathfrak B}_c^-(\mathfrak{A})$ contains a reverse equivariant
sequence $(\lambda^{(n)}),n\in\mathbb{Z},$ such that
$\lambda^{(n)}_i>0$ for all $n\in\mathbb{Z}$ and all
$i\in\{1,\dots,m\}.$

The sequences $(h^{(n)})$ and $(\lambda^{(n)})$ are unique up to
scalar multiplication.

\end{assumption}

A convenient normalization for us will be: $$|\lambda^{(0)}|=1,\
<\lambda^{(0)},h^{(0)}>=1.$$

Our next assumption is the requirement of Lyapunov regularity for
the sequence of matrices $\mathfrak{A}=(A_n),n\in\mathbb{Z}.$ For
renormalization matrices of Rauzy-Veech expansions this assumption
will be seen to hold by the Oseledets Theorem applied to the
Kontsevich-Zorich cocycle.

In fact, we will assume the validity of all the statements of the
Oseledets-Pesin Reduction Theorem (Theorem 3.5.5 on p.77 in
\cite{pesinbarreira}). We proceed to the precise formulation.
\begin{assumption}\label{a-lyap-reg} There exists
$l_0\in\mathbb{N},$ positive numbers
$$\theta_1>\theta_2>\dots>\theta_{l_0}>0,$$ and, for any
$n\in\mathbb{Z},$ direct-sum decompositions
$$\mathbb{R}^m=E_n^1\oplus\dots\oplus E_n^{l_0}\oplus E_n^{cs}$$
$$\mathbb{R}^m=\tilde{E}_n^1\oplus\dots\oplus \tilde{E}_n^{l_0}\oplus
\tilde{E}_n^{cs}$$
such that the following holds.

\begin{enumerate}
\item for all $n\in\mathbb{Z}$ we have $$E_n^1=\mathbb{R}h^{(n)},
\tilde{E}_n^1=\mathbb{R}{\lambda}^{(n)},$$

\item for all $n\in\mathbb{Z}$ and all $i=1,\dots,l_0$ we have

$$A_nE_n^i=E_{n+1}^i,\ A_n^t\tilde{E}^i_{n+1}=\tilde{E}^i_n$$

\item for all $n\in\mathbb{Z}$, every $i=1,\dots,l_0$, and any
$v\in E_n^i\setminus\{0\}$, we have
$$\lim_{k\to\infty}\frac{\log|A_{n+k-1}\dots A_n
v|}{k}=\lim_{k\to\infty}\frac{-\log|A_{n-k}^{-1}\dots
A_{n-1}^{-1}v|}{k}=\theta_i,$$ and the convergence is uniform on
the sphere $\{v\in E_n^i:|v|=1\};$

\item for all $n\in\mathbb{Z}$, every $i=1,\dots,l_0$, and any
$v\in E_n^i\setminus \{0\}$, we have
$$\lim_{k\to\infty}\frac{\log|A^t_{n-k}\dots A_{n-1}^tv|}{k}=
\lim_{k\to\infty}\frac{-\log|{(A_{n+k-1}^t)}^{-1}\dots{A_{n}^t}^{-1}v|}{k}=\theta_i,$$
and the convergence is uniform on the sphere
$\{v\in\tilde{E}_n^i:|v|=1\};$

\item for all $n\in\mathbb{Z}$ we have
$$A_nE_n^{cs}=E_{n+1}^{cs},A_n^tE_{n+1}^{cs}=E_n^{cs}$$

\item for any $\varepsilon>0$ there exists $C_{\varepsilon}$ such
that for any $n\in\mathbb{Z}$ and $k\in\mathbb{N}$ we have

$$||A_{n+k-1}\dots A_n|_{E_n^{cs}}||\leqslant C_{\varepsilon}
e^{\varepsilon(k+|n|)};$$

$$||A_{n-k}^{-1}\dots
A_{n-1}^{-1}|_{E_n^{cs}}||^{-1}\leqslant
C_{\varepsilon}e^{\varepsilon(k+|n|)};$$

$$||A_{n-k}^{t}\dots
A_{n-1}^{t}|_{\tilde{E}_n^{cs}}||\leqslant
C_{\varepsilon}e^{\varepsilon(k+|n|)};$$

$$||(A_{n+k-1}^{t})^{-1}\dots
(A_{n}^{t})^{-1}|_{\tilde{E}_n^{cs}}||^{-1}\leqslant
C_{\varepsilon}e^{\varepsilon(k+|n|)}.$$

\item for all $n\in\mathbb{Z}$ we have $\dim E_n^{cs}=
\dim \tilde{E}_n^{cs}$ and, for any $i=1,\dots,l_0$ we also have
$\dim E_n^i=\dim \tilde{E}_n^{i}$. If $v,\tilde{v}\in\mathbb{R}^m$
satisfy $<v,\tilde{v}>\neq0,$ then $v\in E_n^i$ implies
$\tilde{v}\in\tilde{E}_n^i,$ while $v\in E_n^{cs}$ implies
$\tilde{v}\in\tilde{E}_n^{cs},$ and vice versa.

\end{enumerate}
\end{assumption}

A balanced sequence $\mathfrak{A}$ of $m\times m$-matrices with positive
entries, having sub-exponential growth and satisfying the unique ergodicity assumption as well as
the Lyapunov regularity assumption, will
be called a {\it strongly biregular} sequence, or, for brevity, an
SB-sequence. Using Markovian sequences of partitions induced by
Rauzy-Veech expansions of zippered rectangles corresponding to
consecutive occurrences of a fixed renormalization matrix with
positive entries and applying the Oseledets Multiplicative Ergodic
Theorem and the Oseledets-Pesin Reduction Theorem (see Theorem 3.5.5 on p.77 in
\cite{pesinbarreira}) to the Kontsevich-Zorich cocycle, we will
establish in the next section the following simple
\begin{proposition} Let $\mathbb{P}$ be an ergodic probability
measure on a connected component $\mathcal{H}$ of the moduli space
${\mathcal{ M}}_{\kappa}$ of abelian differentials. Then
$\mathbb{P}$-almost every abelian differential
$(M,\omega)\in\mathcal{H}$ admits an exact Markov sequence of
partitions whose sequence of adjacency matrices belongs to the class
SB.
\end{proposition}

Let $(M,\omega)$ be an abelian differential whose horizontal
and vertical foliations are both uniquely ergodic. Assume that
$(M,\omega)$ is endowed with an exact Markov sequence
$\pi_n,n\in\mathbb{Z}$, of partitions into weakly admissible
rectangles such that the corresponding sequence $\mathfrak{A}$ of
adjacency matrices belongs to the class SB.

Note that if $\mathfrak{A}$ is an SB-sequence, then
$${\mathfrak B}^+(\mathfrak{A})=E_0^1\oplus\dots\oplus E_0^{l_0}$$
$${\mathfrak B}^-(\mathfrak{A})=\tilde{E}_0^1\oplus\dots\oplus\tilde{E}_0^{l_0}$$

Note also that there exists a constant $C>0$ such that the
positive equivariant sequence $h^{(n)}$ satisfies
$$\frac{h_i^{(n)}}{h_j^{(n)}}\leqslant C$$ for all
$n\in\mathbb{Z},i,j\in\{1,\dots,n\}$. It follows that for any
$\varepsilon>0$ there exists a constant $C_{\varepsilon}$ such
that for all $n>0$ we have $$\min_i h_i^{(-n)}\geqslant
C_{\varepsilon} e^{-(\theta_1+\varepsilon)n}.$$

\subsection{Characterization of finitely-additive measures.}
\subsubsection{The semi-rings of Markovian arcs.}

Given a partition $\pi$ of our surface $M$ into weakly admissible
rectangles $\Pi_1,\dots,\Pi_m,$ we consider the semi-ring
$\mathfrak{C}^+(\pi)$ of arcs of the form $[x,x']$, where
$x\in\partial_h^0\Pi_i,\ x'\in\partial_h^1\Pi_i$ for some $i$
(recall here that $x\in\partial_h^0\Pi_i$ stands for the lower
horizontal boundary of $\Pi_i,\partial_h^1\Pi_i$ for the upper
horizontal boundary).

Our Markov sequence $\pi_n$ thus induces a sequence of semi-rings
$\mathfrak{C}_n^+=\mathfrak{C}^+(\pi_n);$ we write
$\mathfrak{R}_n^+$ for the ring generated by the semi-ring
$\mathfrak{C}_n^+$. Elements of $\mathfrak{R}_n^+$ are finite
unions of arcs from $\mathfrak{C}_n^+$. For an arc $\gamma$ of the
vertical flow, let $\breve{\gamma}_n$ be the largest by inclusion
arc from the ring $\mathfrak{R}_{-n}^+$ contained in $\gamma$, and
let $\hat{\gamma}_n$ be the smallest by inclusion arc from the
ring $\mathfrak{R}_{-n}^+$ containing $\gamma$.

\subsubsection{Extension of finitely-additive measures.}
The following Lemma is immediate from the definitions (note that
similar arc approximation lemmas were used by Forni in \cite{F2} and Zorich in \cite{Z}).
\begin{lemma} \label{arc-dec-lem}
Let $\pi_n$ be an exact  Markovian sequence of partitions such
that the corresponding sequence $\mathfrak{A}$ of adjacency
matrices has sub-exponential growth. Then for any $\varepsilon>0$
there exists $C_{\varepsilon}>0$ such that for any arc $\gamma$ of
the vertical flow and any $n\in\mathbb{N}$ the set $\hat{\gamma}_n
\backslash \breve{\gamma}_n$ consists of at most
$C_{\varepsilon}e^{\varepsilon n}$ arcs from the semi-ring
$\mathfrak{C}_{-n}^+$.
\end{lemma}
\begin{figure}
\begin{center}
\includegraphics{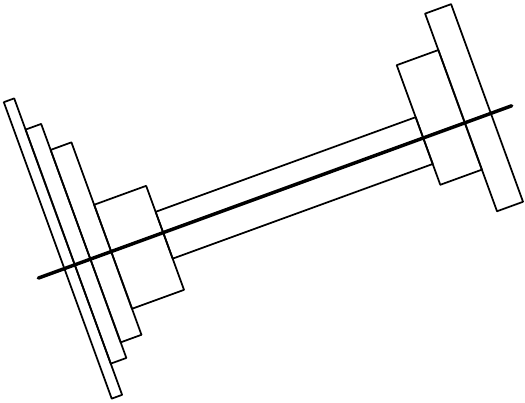}\\
\caption{The number of small arcs grows at most subexponentially}\label{fig:five}
\end{center}
\end{figure}
Informally, Lemma \ref{arc-dec-lem} says that any arc of our
symbolic flow is approximable by Markovian arcs with sub-exponential error;
we illustrate this by Figure \ref{fig:five}.

We are now ready to identify $\mathfrak{B}^+(\mathfrak{A})$ and
$\mathfrak{B}^+(X,\omega)$.

\begin{lemma}
\label{sb-bplus}
Let $\pi_n,n\in\mathbb{Z}$, be a Markov sequence of
partitions such that the corresponding sequence $\mathfrak{A}$ of
adjacency matrices belongs to the class SB. Then for every
equivariant sequence $v^{(n)}\in\mathfrak{B}^+(\mathfrak{A})$
there exists a unique finitely-additive measure
$\Phi^+\in\mathfrak{B}^+(X,\omega)$ such that
$$\eval_0^+(\Phi^+)=v^{(0)}.$$
\end{lemma}

Proof. Indeed, the sequence $v^{(n)}$ itself
prescribes the values of the $\Phi^+$ on all Markovian arcs
$\gamma\in\mathfrak{C}_n^+,n\in\mathbb{Z}.$

For a general arc $\gamma$ of the vertical flow, set
\begin{equation}
\label{approx-arcs}
\Phi^+(\gamma)=\lim_{n\to\infty}\Phi^+(\hat{\gamma}_n)=
\lim_{n\to\infty}\Phi^+(\breve{\gamma}_n),
\end{equation} where the existence
of both limits and the equality of their values immediately follow
from Lemma \ref{arc-dec-lem}.
Finite-additivity of $\Phi^+$ is again a corollary of Lemma  \ref{arc-dec-lem}.
We have thus obtained a finitely-additive measure $\Phi^+$
defined on all vertical arcs. The uniqueness of such a  measure is clear by (\ref{approx-arcs}).
The invariance of the resulting measure under
horizontal holonomy is also clear by definition.
To conclude the proof of Lemma \ref{sb-bplus}, it remains to
check that the obtained finitely-additive measure $\Phi^+$ satisfies
the H{\"o}lder property.

For Markovian arcs the H{\"o}lder upper bound is clear from the {\it
upper} exponential bound $$|v^{(-n)}|\leqslant Ce^{-\alpha n},$$
and the {\it lower} exponential bound $$\min_i h_i^{(-n)}\geqslant
C_1 e^{(-\alpha_1n)}.$$

For general arcs, the H{\"o}lder property follows by
Lemma  \ref{arc-dec-lem}. Lemma \ref{sb-bplus} is proved completely, and we have
thus shown that under its assumptions the map $$\eval_0^+:
\mathfrak{B}^+(X,\omega)\to\mathfrak{B}^+(\mathfrak{A})$$ is
indeed an isomorphism.

{\bf Remark.} Under stronger
assumptions of Lyapunov regularity we will also give a H{\"o}lder {\it lower} bound for the cocycles $\Phi^+\in\BB^+$,
see Proposition \ref{convhoeldprop}.

\subsection {Duality.}
Let $v^{(n)}\in\mathfrak{B}^+(\mathfrak{A}),
\tilde{v}^{(n)}\in\mathfrak{B}^-(\mathfrak{A}),$ and let
$\Phi^+\in\mathfrak{B}^+(X,\omega),\Phi^-\in\mathfrak{B}^+(X,\omega)$
be the corresponding finitely-additive measures. The definitions
directly imply $$\int_M\Phi^+\times\Phi^-=\sum_{i=1}^m
v_i^{(0)}\tilde{v}_i^{(0)}=<v,\tilde{v}>.$$

Duality between the spaces $\mathfrak{B}^+(X,\omega)$ and
$\mathfrak{B}^-(X,\omega)$ follows now from the duality between
 the spaces $\mathfrak{B}^+(\mathfrak{A})$ and
$\mathfrak{B}^-(\mathfrak{A})$, which holds by the Lyapunov
regularity assumption for the sequence $\mathfrak{A}$.

\subsection{Proof of Theorem \ref{multiplicmoduli}.}
\subsubsection{Approximation of almost equivariant sequences.}
 We start with  a sequence of matrices
$A_n,n\geqslant0$ satisfying the following

\begin{assumption}\label{asexpdec} There exists $\alpha>0$ and, for
every $n\geqslant1$, a direct-sum decomposition
$$\mathbb{R}^m=E^u_n\oplus E^n_{cs}$$ satisfying the following.
\begin{enumerate}
\item $A_nE_n^u=E^u_{n+1}$ and $A_n|_{E^u_n}$ is injective.

\item $A_nE_n^{cs}\subset E_{n+1}^{cs}.$

\item  For any $\varepsilon>0$ there exists $C_{\varepsilon}>0$ such
that for any $n\geqslant 1,\ k\geqslant0$ we have
$$\left|\left|(A_{n+k}\cdot\ldots\cdot A_n)^{-1}|_{E^u_{n+k+1}}\right|\right|\leqslant
C_{\varepsilon}e^{\varepsilon n-\alpha k}$$

$$\left|\left|A_{n+k}\cdot\ldots\cdot A_n|_{E^{cs}_{n}}\right|\right|\leqslant
C_{\varepsilon}e^{\varepsilon(n+k)}.$$

\end{enumerate}
\end{assumption}

\begin{lemma}
\label{vngeneral}
Let $A_n$ be a sequence of matrices satisfying Assumption \ref{asexpdec}.
Let $v_1, \dots, v_n,\dots$ be a sequence of vectors such that for any $\varepsilon>0$ a constant
$C_{\varepsilon}$   can be chosen in such a way that for all $n$ we have
$$
|A_nv_n-v_{n+1}|\leq C_{\varepsilon}\exp(\varepsilon n).
$$
Then there exists a unique vector $v\in E_1^u$ such that
$$
|A_n\dots A_1v-v_{n+1}|\leq C^{\prime}_{\varepsilon}\exp(\varepsilon n).
$$
\end{lemma}

Proof: Denote $u_{n+1}=v_{n+1}-A_nv_n$ and
decompose $u_{n+1}=u_{n+1}^+ + u_{n+1}^-$, where
$u_{n+1}^+\in E_{n+1}^u$, $u_{n+1}^-\in E_{n+1}^{cs}$.
Let
$$
v_{n+1}^+=u_{n+1}^+ + A_{n} u_n^+ +
A_{n}A_{n-1}u_{n-1}^+ + \dots + A_{n}\dots A_1u_1^+;
$$
$$
v_{n+1}^-=u_{n+1}^- + A_{n} u_n^- +
A_{n}A_{n-1}u_{n-1}^- + \dots + A_{n}\dots A_1u_1^-.
$$
We have $v_{n+1}\in E_{n+1}^u$, $v_{n+1}^-\in E_{n+1}^{cs}$,
$v_{n+1}=v_{n+1}^+ + v_{n+1}^-$.
Now introduce a vector
$$
v=u_1^+ + A_1^{-1}u_1^+ + \dots + (A_n\dots A_1)^{-1}u_{n+1}^+ +\dots
$$

By our assumptions, the series defining $v$ converges exponentially fast
and, moreover, we have
$$
|A_n\dots A_1v-v_{n+1}^+|\leq C^{\prime}_{\varepsilon}\exp(\varepsilon n)
$$
for some constant $C_{\varepsilon}^{\prime}$.

Since, by our assumptions we also have $|v_{n+1}^-|\leq
C_{\varepsilon}\exp(\varepsilon n)$, the Lemma is proved completely.

Uniqueness of the vector $v$ follows from the fact that, by our assumptions,
for any ${\tilde v}\neq 0$, ${\tilde v}\in E^{u}_0$ we have
$$
|A_n\dots A_1{\tilde v}|\geq C^{\prime\prime} \exp(\alpha n).
$$

\subsubsection{Approximation of weakly Lipschitz functions.}

Let $f:M\to\mathbb{R}$ be a weakly Lipschitz function, and,
as before, introduce a canonical family of Markovian curves
$\gamma_i^n,n\in\mathbb{Z},i=1,\dots,m,$ corresponding to the
Markovian sequence of partitions $\pi_n,n\in\mathbb{Z}$.

Introduce a family of vectors
$v(n)\in\mathbb{R}^m,n\in\mathbb{Z},$ by setting
$$(v(n))_i=\int\limits_{\gamma_i^n}fd\nu^+,\ i=1,\dots,m.$$

By definition of the adjacency matrices $A_n=A(\pi_n,\pi_{n+1}),$
using the weak Lipschitz property of the function $f$ and
sub-exponential growth of the sequence $\mathfrak{A}=(A_n)$, we
arrive for all $n\in\mathbb{N}$ at the estimate
$$\left|A_nv(n)-v(n+1)\right|\leqslant
C_{\varepsilon}||f||_{Lip}\cdot e^{\varepsilon n}.$$

By Lemma \ref{vngeneral}, there exists a unique vector
$v_f^+\in\mathfrak{B}^+(\mathfrak{A})$ such that for all
$n\in\mathbb{N}$ we have $$\left|v(n)-A_{n-1}\cdot\dots\cdot
A_0v_f^+\right|\leqslant C_{\varepsilon}'||f||_{Lip}\cdot
e^{\varepsilon n}.$$

We let $\Phi_f^+\in\mathfrak{B}^+(X,\omega)$ be the
finitely-additive measure corresponding to the vector $v_f^+$, or,
in other words, the unique finitely additive measure in
$\mathfrak{B}^+(X,\omega)$ satisfying
$$\eval_0^+(\Phi_f^+)=v_f^+.$$ The inequality

\begin{equation}\label{a-apr}
\left|\int\limits_0^Tf\circ
h_t^+(x)dt-\Phi_f^+([x,h_T^+x])\right|\leqslant
C_{\varepsilon}''||f||_{Lip}(1+T^{\varepsilon})
\end{equation}

now holds for all $x\in M,T\in\mathbb{R}_+.$

Indeed, if $[x,h_t^+x]$ is a Markovian arc, then (\ref{a-apr}) is
clear by definition of the vector $v_f^+$ and the weak Lipschitz
property of the function $f$, while for general arcs of the vertical flow the
inequality (\ref{a-apr}) follows by Lemma \ref{arc-dec-lem}.

\subsubsection{Characterization of the cocycle $\Phi_f^+$}
Our next step is to check that for every
$\Phi^-\in\mathfrak{B}^-(X,\omega)$ we have
$$<\Phi_f^+\times\Phi^->=\int\limits_Mfdm_{\phi^-},$$ where we
recall that $m_\Phi^-=\nu^+\times\Phi^-.$

As before, let $\gamma_i^{n}$ be a canonical system of Markovian
arcs of the vertical flow, corresponding to the Markov sequence of
partitions $\pi_n, n\in \mathbb{Z},$ and let
$\tilde{\gamma}_i^{n}$ be a canonical system of Markovian arcs of
the horizontal flow corresponding to the Markov sequence of
partitions $\pi_n, n\in \mathbb{Z}.$ By definition, for any
$n\in\mathbb{Z}$ we have

$$\int \limits _M {\Phi_f^{+} \times \Phi^{-}}=
\sum_{i=1}^{m}{\Phi_f^{+}(\gamma_i^{n}) \cdot \Phi^{-}(\tilde{\gamma}_i^{n})}.$$
We now  write the Riemann sum
$$
S(n, f, \Phi^{-})=\sum_{i=1}^{m}{\int
\limits_{\gamma_i^{n}}{fd\nu^{+}}\cdot\Phi^{-}(\tilde{\gamma}_i^{n})}
$$
for the measure $m_{\Phi^-}$ and let $n$ tend to $+\infty$.
By definition, we have

$$\lim_{n\rightarrow+\infty}{S(n, f, \Phi^{-})}=\int\limits_M{fdm_{\Phi^{-}}}.$$
Now for all $n\in\mathbb{N}, i=1,\ldots, m,$ we have

$$\left| \int
\limits_{\gamma_i^{n}}{fd\nu^{+}}-\Phi_f^{+}(\gamma_i^{n})\right|\leqslant
C_\varepsilon e^{\varepsilon n},$$
while, by Lyapunov regularity, the quantity $\max \limits_{i=1, \dots, m}
{|\Phi^{-}(\tilde{\gamma}_i^{n})|}$ decays exponentially as $n
\rightarrow\infty.$ It follows that

$$\int \limits_M {\Phi_f^{+}\times \Phi^{-}}= \int \limits_M {fdm_{\Phi^{-}}},$$
which is what we had to prove.

\subsection{The asymptotics at infinity for H{\"o}lder cocycles}

As was noted above, we identify a finitely-additive measure
$\Phi^{+}\in \mathfrak{B}_c^{+}(X, \omega)$ with a continuous
cocycle over the vertical flow for which, slightly abusing
notation, we keep the same symbol $\Phi ^{+};$ the identification
is given by the formula
$$\Phi^{+}(x, t)=\Phi^{+}([x, h_t^{+}x]).$$
The H{\"o}lder property of a finitely-additive measure is
equivalent to the H{\"o}lder property of the cocycle, that is, to
the requirement that the function $\Phi^{+}(x, t)$ be H{\"o}lder
in $t$ uniformly in $x$. Our next aim is to give H{\"o}lder lower
bounds for the cocycles $\Phi^{+}$ and to investigate the growth
of $\Phi^{+}(x, T)$ as $T\rightarrow +\infty$.

Consider the direct-sum decomposition
$$\mathfrak{B}^{+}(\mathfrak{A})=E_1^{(0)}\oplus\ldots\oplus
E_{l_0}^{(0)}$$
and the corresponding direct-sum decomposition

$$\mathfrak{B}^{+}(X, \omega)=\mathfrak{B}_1^{+}(X,
\omega)\oplus\mathfrak{B}_2^{+}(X,
\omega)\oplus\ldots\oplus\mathfrak{B}_{l_0}^{+}(X, \omega)$$
with
$$\mathfrak{B}_i^{+}(X, \omega)=(\eval_0^{+})^{-1}(E_i^{(0)}),
i=1,\ldots,l_0;$$ of course, we have $$\mathfrak{B}_1^{+}(X,
\omega)=\mathbb{R}\nu^{+}.$$

Take $\Phi^{+}\in\mathfrak{B}^{+}(X, \omega)$ and write
$$\Phi^{+}=\Phi_1^{+}+\ldots+\Phi^{+}_{l_0}$$ with
$\Phi_i^{+}\in\mathfrak{B}^{+}(X, \omega)$. Take the smallest $i$
such that $\Phi_i^{+}\neq0;$ the exponent $\theta_i$ will then be
called {\it the top Lyapunov exponent} of $\Phi^{+};$ similarly,
if $j$ is the largest number such that $\Phi_j^{+}\neq0,$ then
$\theta_j$ will be called {\it the lower Lyapunov exponent} of
$\Phi^{+}.$ We shall now see that the top Lyapunov exponent
controls the growth of $\Phi^{+}(x, t)$ as $t\rightarrow\infty,$
while the lower Lyapunov exponent describes the local H{\"o}lder
behaviour of $\Phi^{+}(x, t).$

\begin{proposition}
\label{convhoeldprop}
Let $r\in\{1, \ldots, l_0\},$ let $\Phi^{+}\in\mathfrak{B}_r^{+},
\Phi^{+}\neq0$, and let $x\in M$ be such that $h_t^{+}x$ is
defined for all $t\in\mathbb{R}.$ Then
$$ \limsup_{|t|\rightarrow\infty} \frac{\log{|\Phi^{+}(x, t)|}}{\log{|t|}} =
\limsup_{|t|\rightarrow 0} \frac{\log{|\Phi^{+}(x, t)|}}{\log{|t|}}=
\theta_r.$$
\end{proposition}

Proof. We first let $t$ tend to $+\infty.$ Let
$v^{(n)}\in\mathfrak{B}^{+}(\mathfrak{A})$ be the equivariant
sequence corresponding to $\Phi^{+};$ we have $v^{(n)}\in
E_r^{(n)}$ and, consequently, for every $\varepsilon>0$ there
exists a constant $C_\varepsilon>0$ and, for every $n\in
\mathbb{N},$ there exists $i(n)\in \{1, \ldots, m\}$ such that

\begin{equation}\label{est-vn}
\left| v_{i(n)}^{(n)} \right| \geq C_\varepsilon
e^{(\theta_r-\varepsilon)n}, n \in \mathbb{N}.
\end{equation}

Now let
$$t_n= \min\{t: t\geq h_{i(n)}^{(n)}, h_{t_n}^{+}x\in
\partial_v^{1}\Pi_{i(n)}^{(n)}\}$$
Informally, $t_n$ is the first such moment that the are $[x,
h_{t_n}^{+}x]$ contains a Markovian arc going all the way through
the rectangle $\Pi_{i(n)}^{(n)}$. It is clear from the
$SB$-property of the sequence $\mathfrak{A}$ that for any
$\varepsilon>0$ there exist constants $C_\varepsilon^\prime,
C_\varepsilon^{\prime\prime}>0$ such that

\begin{equation}\label{est-tn}
C_\varepsilon^\prime e^{(\theta_1 -\varepsilon)n}\leq t_n\leq
C_\varepsilon^{\prime \prime} e^{(\theta_1 +\varepsilon)n}, n\in
\mathbb{N}.
\end{equation}

Now denote $$x^\prime (n)= h_{t_n-h_{i(n)}^{(n)}}^{+} x, x^{\prime
\prime}(n)= h_{t_n}^{+}x.$$

Since $$\Phi^{+}([x, x^{\prime \prime}(n)])=\Phi^{+}([x,
x^{\prime}(n)])+\Phi^{+}([x^{\prime}(n), x^{\prime \prime}(n)]),$$
and $$\Phi^{+}([x^{\prime}(n), x^{\prime
\prime}(n)])=v^{(n)}_{i(n)},$$
it follows from (\ref{est-tn}),
(\ref{est-vn}) that we have
$$\limsup_{n \rightarrow +\infty}
\frac{\max \{\log {|\Phi^{+}([x, x^{\prime}(n)])|},
\log{|\Phi^{+}([x, x^{\prime \prime}(n)])|}\}}{n}=
\frac{\theta_r}{\theta_1},$$

whence also
$$\limsup_{n\rightarrow +\infty} \max \left( \frac{\log {|\Phi^{+}([x,
x^{\prime}(n)])|}}{\log{\nu^{+}([x, x^{\prime}(n)])}}, \frac{\log
{|\Phi^{+}([x^{\prime}(n), x^{\prime \prime}(n)])|}}{\log{\nu^{+}([x, x^{\prime
\prime}(n)])}} \right) = \theta_r,$$
and, finally
$$\limsup_{t\rightarrow +\infty} \frac{\log{|\Phi^{+}(x, t)|}}{\log{t}}=\theta_r.$$

\begin{figure}
\begin{center}
\includegraphics{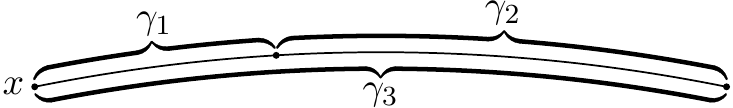}\\
\caption{Proof of the lower bound in Proposition \ref{convhoeldprop}:
$\gamma_1=[x(n),x^{\prime}(n)]$, $\gamma_2=[x^{\prime}(n), x^{\prime\prime}(n)]$.
Either $\gamma_1$ or $\gamma_3$ satisfies the lower bound. }\label{fig:one}
\end{center}
\end{figure}
The desired lower bound is established.
 We illustrate the argument in Figure \ref{fig:one}.

The proof for $t\rightarrow -\infty$ is completely similar, while
the case $t \rightarrow 0$ is obtained by taking $n\rightarrow
-\infty$ and repeating the same argument.

\subsection{Hyberbolic  SB - sequences}
An $SB$ sequence $\mathfrak{A}$ will be called {\it hyperbolic} if
$\mathfrak{B}_c^{+}(\mathfrak{A})=\mathfrak{B}^{+}(\mathfrak{A}).$
It is clear from the definitions that if an abelian differential
$(X, \omega)$ admits an exact Markovian sequence of partitions
such that the corresponding sequence $\mathfrak{A}$ is hyperbolic,
then $\mathfrak{B}_c^{+}(X, \omega)=\mathfrak{B}^{+}(X, \omega).$

In what follows, we shall check that if $\mathbb{P}$ is
probability measure on $\mathcal{H}$, invariant under the
Teichm\"{u}ller flow and ergodic, and such that the
Kontsevich-Zorich cocycle acts isometrically on its neutral
Oseledets subspace, then $\mathbb{P}$-almost every abelian
differential $(X, \omega)$ admits a Markovian sequence of
partitions $\pi_n$ such that the corresponding sequence
$\mathfrak{A}$ of adjacency matrices is a hyperbolic $SB$ -
sequence. It will follow that for $\mathbb{P}$ - almost all
$\omega$ we have $\mathfrak{B}^{+}(X,
\omega)=\mathfrak{B}_c^{+}(X, \omega).$

{\bf Remark.} If $\mathfrak{B}_c^{+}(\mathfrak{A})$ is
strictly larger than $\mathfrak{B}^{+}(\mathfrak{A})$, it does not
follow that $\mathfrak{B}_c^{+}(X, \omega)$ is strictly larger
than $\mathfrak{B}^{+}(X, \omega):$ our constructions do not allow
us to assign a finitely additive measure defined on all arcs of
the vertical flow to a general equivariant sequence
$v^{(n)}\in\mathfrak{B}^{+}(\mathfrak{A})$.

\subsection{Expectation and variance of H{\"o}lder cocycles}
\begin{proposition}
\label{phiexpzero}
For any $\Phi^+\in \BB^+$ and any $t_0\in {\mathbb R}$ we have
$$
\ee_{\nu}(\Phi^+(x,t_0))=\langle \Phi^+, \nu^-\rangle\cdot t_0.
$$
\end{proposition}

Proof: Since the Proposition is clearly valid for $\Phi^+=\nu^+$, it suffices to prove it
in the case $\langle \Phi^+, \nu^-\rangle=0$.
But indeed, if $\ee_{\nu}(\Phi^+(x,t))\neq 0$, then the Ergodic  Theorem implies
$$
\limsup_{T\to\infty} \frac{\log |\Phi^+(x,T)|}{\log T}=1,
$$
and then $\langle \Phi^+, \nu^-\rangle\neq 0$.

\begin{proposition}
\label{phivarnotzero}
For any $\Phi^+\in \BB^+$ not proportional to $\nu^+$  and any $t_0\neq 0$ we have
$$
Var_{\nu} \Phi^+(x,t_0)\neq 0.
$$
\end{proposition}

Taking $\Phi^+-\langle \Phi^+, \nu^-\rangle\cdot\nu^+$ instead of $\Phi^+$, we may
assume $\ee_{\nu}(\Phi^+(x,t_0))=0$. If $Var_{\nu} \Phi^+(x,t_0)=0$, then $\Phi^+(x,t_0)=0$
identically, but then
 $$
\limsup_{T\to\infty} \frac{\log |\Phi^+(x,T)|}{\log T}=0,
$$
whence $\Phi^+=0$, and the Proposition is proved.

{\bf Remark.} In the context of substitutions,
cocycles related to the H{\"o}lder cocycles from ${\mathfrak B}^+$ have been studied by P.~Dumont, T.~Kamae and
S.~Takahashi in \cite{dumontkamae}
as well as by T.~Kamae in \cite{kamae}.

\section{The Teichm{\"u}ller Flow on the Veech Space of Zippered Rectangles.}

\subsection{Veech's space of zippered rectangles}

\subsubsection{Rauzy-Veech induction}

The renormalization action of the Teichm{\"u}ller flow on the spaces $\BB^+$ and $\BB^-$ of H{\"o}lder cocycles will play a main r{\^o}le in the
proof of limit theorems for translation flows. We will use Veech's
representation of abelian differentials by zippered rectangles,
and in this section we recall Veech's construction using the
notation of \cite{bufetov}, \cite{bufgur}.
For a different presentation of the Rauzy-Veech formalism,
see Marmi-Moussa-Yoccoz \cite{MMY}.

We start by recalling the definition of the Rauzy-Veech induction.
Let $\pi$ be a permutation of $m$ symbols, which will always be
assumed irreducible in the sense that $\pi\{1,\dots,k\}=\{1,\dots,k\}$ implies $k=m$.
The Rauzy operations $a$ and $b$ are defined by the formulas
$$
a\pi(j)=\begin{cases}
\pi j,&\text{if $j\leq\pi^{-1}m$,}\\
\pi m,&\text{if $j=\pi^{-1}m+1$,}\\
\pi(j-1),&\text{if $\pi^{-1}m+1<j\le m$;}
\end{cases}
$$
$$
b\pi(j)=\begin{cases}
\pi j,&\text{if $\pi j\leq \pi m$,}\\
\pi j+1,&\text{if $\pi m<\pi j<m$,}\\
\pi m+1,&\text{ if $\pi j=m$.}
\end{cases}
$$

These operations preserve irreducibility. The {\it Rauzy class}
$\mathcal R(\pi)$ is defined as the set of all permutations that can
be obtained from $\pi$ by application of the transformation group
generated by $a$ and $b$. From now on we fix a Rauzy class $\R$ and
assume that it consists of irreducible permutations.

For $i,j=1,\dots,m$, denote by $E^{ij}$ the $m\times m$ matrix whose
$(i,j)th$ entry is $1$, while all others are zeros. Let $E$ be the
identity $m\times m$-matrix. Following Veech \cite{veech}, introduce
the unimodular matrices
\begin{equation} \label{mat_a}
{\mathcal A}(a,\pi)=\sum_{i=1}^{\pi^{-1}m}E^{ii}+E^{m,\pi^{-1}m+1}+
\sum_{i=\pi^{-1}m}^{m-1}E^{i,i+1},
\end{equation}
\begin{equation} \label{mat_b}
{\mathcal A}(b,\pi)=E+E^{m,\pi^{-1}m}.
\end{equation}
For a vector $\la=(\la_1,\dots,\la_m)\in{\mathbb R}^m$, we write
$$
|\la|=\sum_{i=1}^m\la_i.
$$
Let
$$
\Delta_{m-1}=\{\la\in {\mathbb R}^m:|\la|=1,\ \la_i>0 \text{ for }
i=1,\dots,m\}.
$$

One can identify each pair $(\lambda,\pi)$,
$\lambda\in\Delta_{m-1}$, with the {\it interval exchange map} of
the interval $I:=[0,1)$ as follows. Divide $I$ into the
sub-intervals $I_k:=[\beta_{k-1},\beta_k)$, where $\beta_0=0$,
$\beta_k=\sum_{i=1}^k\lambda_i$, $1\le k\le m$, and then place the
intervals $I_k$ in $I$ in the following order (from left to write):
$I_{\pi^{-1}1},\dots, I_{\pi^{-1}m}$. We obtain a piecewise linear
transformation of $I$ that preserves the Lebesgue measure.

The space $\Delta(\R)$ of interval exchange maps corresponding to
$\R$ is defined by
$$
\Delta(\R)=\Delta_{m-1}\times\R.
$$
Denote
$$
\Delta_{\pi}^+=\{\la\in\Delta_{m-1}| \ \la_{\pi^{-1}m}>\la_m\},\ \
\Delta_{\pi}^-=\{\la\in\Delta_{m-1}| \ \la_m>\la_{\pi^{-1}m}\},
$$
$$
\Delta^+(\R)=\bigcup\limits_{\pi\in{\mathcal R}}\{(\pi,\la)|\
\la\in\Delta_{\pi}^+\},$$
$$
\Delta^-(\R)=\bigcup\limits_{\pi\in{\mathcal R}}\{(\pi,\la)|\
\la\in\Delta_{\pi}^-\},$$
$$
\Delta^\pm(\R)=\Delta^+(\R)\cup\Delta^-(\R).
$$
The {\it Rauzy-Veech induction map} ${\mathscr
T}:\Delta^\pm(\R)\to\Delta(\R)$ is defined as follows:
\begin{equation}
\label{te} {\mathscr T}(\la,\pi)=\begin{cases}
(\frac{{\mathcal A}(a,\,\pi)^{-1}\la}{|{\mathcal A}(a,\,\pi)^{-1}\la|},a\pi), &\text{if
$\la\in\Delta_\pi^+$,}\\
(\frac{{\mathcal A}(b,\,\pi)^{-1}\la}{|{\mathcal A}(b,\,\pi)^{-1}\la|},\,b\pi), &\text{if
$\la\in\Delta_\pi^-$} .
\end{cases}
\end{equation}

One can check that ${\mathscr T}(\la,\pi)$ is the interval exchange map
induced by $(\la,\pi)$ on the interval $J=[0,1-\gamma]$, where
$\gamma=\min(\la_m,\la_{\pi^{-1}m})$; the interval $J$ is then stretched to
unit length.

Denote
\begin{equation}
\label{delta+-} \Delta^\infty(\mathcal R)=\bigcap_{n\ge0}\mathscr
T^{-n} \Delta^\pm(\mathcal R).
\end{equation}
Every ${\mathscr T}$-invariant probability measure is concentrated on
$\Delta^\infty(\mathcal R)$. On the other hand, a natural Lebesgue
measure defined on $\Delta(\mathcal R)$, which is finite, but
non-invariant, is also concentrated on $\Delta^\infty(\mathcal R)$.
Veech \cite{veech} showed that $\mathscr T$ has an absolutely
continuous ergodic invariant measure on $\Delta(\R)$, which is,
however, infinite.

We have two matrix cocycles ${\mathcal A}^t$, ${\mathcal A}^{-1}$over $\mathscr T$ defined by
$$
{\mathcal A}^t(n,(\lambda, \pi))={\mathcal A}^t({\mathscr T}^n(\lambda, \pi))\cdot \ldots \cdot {\mathcal A}^t(\lambda,\pi),
$$
$$
{\mathcal A}^{-1}(n,(\lambda, \pi))={\mathcal A}^{-1}({\mathscr T}^n(\lambda, \pi))\cdot \ldots \cdot {\mathcal A}^{-1}(\lambda,\pi).
$$

We introduce the corresponding skew-product transformations
${\mathscr T}^{{\mathcal A}^t}: \Delta(\R)\times {\mathbb R}^m\to\Delta(\R)\times {\mathbb R}^m$,
${\mathscr T}^{{\mathcal A}^{-1}}: \Delta(\R)\times {\mathbb R}^m\to\Delta(\R)\times {\mathbb R}^m$,
$$
{\mathscr T}^{{\mathcal A}^t}((\la,\pi), v)=({\mathscr T}(\la,\pi), {\mathcal A}^t(\la,\pi)v);
$$
$$
{\mathscr T}^{{\mathcal A}^{-1}}((\la,\pi), v)=({\mathscr T}(\la,\pi), {\mathcal A}^{-1}(\la,\pi)v).
$$

\subsubsection{The construction of zippered rectangles}
\label{zipper}

Here we briefly recall the construction of the Veech space of
zippered rectangles. We use the notation of \cite{bufetov}.

{\it Zippered rectangles} associated to the Rauzy class $\R$ are
triples $(\la,\pi,\delta)$, where
$\la=(\la_1,\dots,\la_m)\in{\mathbb R}^m$, $\la_i>0$, $\pi\in{\mathcal
R}$, $\delta=(\delta_1,\dots,\delta_m)\in{\mathbb R}^m$, and the
vector $\delta$ satisfies the following inequalities:
\begin{equation}
\label{deltaone} \delta_1+\dots+\delta_i\leq 0,\ \ i=1,\dots,m-1,
\end{equation}
\begin{equation}
\label{deltatwo} \delta_{\pi^{-1}\,1}+\dots+\delta_{\pi^{-1}\,i}\geq
0, \ \ i=1, \dots, m-1.
\end{equation}
The set of all vectors $\delta$ satisfying (\ref{deltaone}),
(\ref{deltatwo}) is a cone in ${\mathbb R}^m$; we denote it by
$K(\pi)$.

For any $i=1, \dots, m$, set
\begin{equation}
\label{adelta}
a_j=a_j(\delta)=-\delta_1-\dots-\delta_j, \
h_{j}=h_j(\pi,\delta)=-\sum_{i=1}^{j-1} \delta_i+\sum_{l=1}^{\pi(j)-1}\delta_{\pi^{-1}l}.
\end{equation}

\subsubsection{Zippered rectangles and abelian differentials.}

Given a zippered rectangle ${\mathscr X}=(\la,\pi,\delta)$, Veech \cite{veech} takes $m$ rectangles
$\Pi_i=\Pi_i(\la,\pi,\delta)$ of girth $\la_i$ and height $h_i$, $i=1, \dots, m$, and glues them together according to
a rule determined by the permutation $\pi$. This procedure yields a Riemann surface $M$ endowed with a holomorphic
$1$-form $\omega$ which, in restriction to each $\Pi_i$, is simply the form $dz=dx+idy$.
The union of the bases of the rectangles is an interval $I^{(0)}(\la,\pi, \delta)$ of length $|\la|$ on $M$; the
first return map of the vertical flow of the form $\omega$ is precisely the interval exchange ${\bf T}_{(\la,\pi)}$.

A zippered rectangle ${\mathscr X}$ by definition carries a partition $\pi_0=\pi_0({\mathscr X})$ of
the underlying surface $M=M({\mathscr X})$
into $m$ weakly admissible rectangles$\Pi_i$:
$$
\pi_0: M=\Pi_1\sqcup \dots\sqcup \Pi_m.
$$

The {\it area} of a zippered rectangle $(\la,\pi,\delta)$ is given
by the expression
\begin{equation}
\label{area} Area\,(\la,\pi,\delta):=\sum_{r=1}^m\la_rh_r=
\sum_{r=1}^m\la_r(-\sum_{i=1}^{r-1}\delta_i+\sum_{i=1}^{\pi
r-1}\delta_{\pi^{-1}\,i}).
\end{equation}
(Our convention is $\sum_{i=u}^v...=0$ when $u>v$.)

Furthermore, to each rectangle $\Pi_i$ Veech \cite{veechamj} assigns a cycle $\gamma_i(\la,\pi, \delta)$ in the homology
group $H_1(M, {\mathbb Z})$: namely, if $P_i$ is the left bottom corner of $\Pi_i$ and $Q_i$ the left top corner, then the cycle is
the union of the vertical interval $P_iQ_i$ and the horizontal subinterval of $I^{(0)}(\la,\pi, \delta)$ joining $Q_i$ to $P_i$. It is
clear that the cycles $\gamma_i(\la,\pi, \delta)$ span $H_1(M, {\mathbb Z})$.

\subsubsection{The space of zippered rectangles.}
Denote by ${\mathcal V}(\R)$ the space of all zippered rectangles
corresponding to the Rauzy class $\R$, i.e.,
$$
{\mathcal V}(\R)=\{(\la,\pi,\delta):\la\in{\mathbb
R}^m_+,\,\pi\in\R,\,\delta\in K(\pi)\}.
$$
Let also
$$
{\mathcal V}^+(\R)=\{(\la,\pi,\delta)\in{\mathcal
V}(\R):\la_{\pi^{-1}m}>\la_m\},
$$
$$
{\mathcal V}^-(\R)=\{(\la,\pi,\delta)\in{\mathcal
V}(\R):\la_{\pi^{-1}m}<\la_m\},
$$
$$
{\mathcal V}^\pm(\R)={\mathcal V}^+(\R)\cup{\mathcal V}^-(\R).
$$
Veech \cite{veech} introduced the flow $\{P^s\}$ acting on ${\mathcal
V}(\R)$ by the formula
$$
P^s(\la,\pi,\delta)=(e^{s}\la,\pi,e^{-s}\delta),
$$
and the map $\U:{\mathcal V}^\pm(\R)\to{\mathcal V}(\R)$, where
$$
{\U}(\la,\pi,\delta)=\begin{cases}({{\mathcal A}}(\pi,a)^{-1}\la,a\pi,{{\mathcal A}}(\pi,
a)^{-1}\delta),
&\text{if $\la_{\pi^{-1}m}>\la_m$,}\\
({{\mathcal A}}(\pi,b)^{-1}\la,b\pi,{{\mathcal A}}(\pi,b)^{-1}\delta), &\text{if
$\la_{\pi^{-1}m}<\la_m$.}
\end{cases}
$$
(The inclusion $\U{\mathcal V}^\pm(\R)\subset{\mathcal V}(\R)$ is proved in \cite{veech}.) The
map $\U$ and the flow $\{P^s\}$ commute on ${\mathcal V}^\pm(\R)$. They also
preserve the area of a zippered rectangle (see (\ref{area})) and
hence can be restricted to the set
$$
{\mathcal V}^{1,\pm}(\R):=\{(\la,\pi,\delta)\in{\mathcal V}^\pm(\R):
Area(\la,\pi,\delta)=1\}.
$$

For $(\la,\pi)\in\Delta(\R)$, denote
\begin{equation}
\label{taulambda}
\tau^0(\la,\pi)=:-\log(|\la|-\min(\la_m,\la_{\pi^{-1}m})).
\end{equation}
From (\ref{mat_a}), (\ref{mat_b}) it follows that if
$\la\in\Delta_\pi^+\cup\Delta_\pi^-$, then
\begin{equation}
\label{tau1} \tau^0(\la,\pi)=-\log|{{\mathcal A}}^{-1}(c,\pi)\la|,
\end{equation}
where $c=a$ when $\la\in\Delta_\pi^+$, and $c=b$ when
$\la\in\Delta_\pi^-$.

Next denote
$$\Y_1(\R):=\{x=(\la,\pi,\delta)\in{\mathcal
V}(\R):|\la|=1,\ Area(\la,\pi,\delta)=1\},
$$
\begin{equation*}
\label{tau3}\tau(x):=\tau^0(\la,\pi)\text{ for
}x=(\la,\pi,\delta)\in\Y_1(\R),
\end{equation*}
\begin{equation}
\label{phase} {\mathcal V}_{1,\tau}(\R):=\bigcup_{x\in\Y_1(\R),\ 0\leq
s\leq \tau(x)}P^sx.
\end{equation}
 Let
$$
{\mathcal V}^{1,\pm}_{\ne}(\R):=\{(\la,\pi,\delta)\in {\mathcal
V}^{1,\pm}(\R):a_m(\delta)\ne 0\},
$$
$$
{\mathcal V}_\infty(\R):=\bigcap_{n\in\mathbb Z}\U^n{\mathcal
V}^{1,\pm}_{\ne}(\R).
$$
Clearly
$\U^n$ is well-defined on ${\mathcal
V}_\infty(\R)$ for all $n\in\mathbb Z$.

We now set
$$
{\mathcal Y}^{\prime}(\R):={\mathcal Y}_1(\R)\cap\mathcal V_\infty(\R),\ \ {\tilde{\mathcal V}}(\R):=
{\mathcal V}_{1,\tau}(\R)\cap{\mathcal V}_\infty(\R).
$$
The above identification enables us to define on $\tilde{\mathcal
V}(\R)$ a natural flow, for which we retain the notation $P^s$
(although the bounded positive function $\tau$ is not separated from
zero, the flow $P^s$ is well-defined).

Note that for any $s\in {\mathbb R}$ we have a natural ``tautological" map
\begin{equation}
\label{def-t-s}
{\mathfrak t}_s: M({\mathscr X})\to M(P^s{\mathscr X})
\end{equation}
which on each rectangle $\Pi_i$ is simply expansion by $e^s$ in the horizontal direction and contraction
by $e^s$ in the vertical direction. By definition, the map ${\mathfrak t}_s$ sends the vertical and the
horizontal foliations of ${\mathscr X}$ to those of $P^s{\mathscr X}$.

Introduce the space
$$
{\mathfrak X}{\tilde{\mathcal V}}(\R)=\{({\mathscr X}, x): {\mathscr X}\in {\tilde{\mathcal V}}(\R),  x\in M({\mathscr X}) \}
$$
and endow the space ${\mathfrak X}{\tilde{\mathcal V}}(\R)$ with the flow $P^{s, {\mathfrak X}}$ given by the formula
$$
P^{s, {\mathfrak X}}({\mathscr X}, x)=(P^s{\mathscr X}, {\mathfrak t}_sx).
$$

The flow $P^s$ induces on the transversal ${\mathcal Y}_1(\R)$ the first-return map ${\overline {\mathscr T}}$ given by the
formula
\begin{equation}
{\overline {\mathscr T}}(\la,\pi, \delta)={\mathcal U}P^{\tau^0(\la,\pi)}(\la,\pi, \delta).
\end{equation}

Observe that, by definition, if ${\overline {\mathscr T}}(\la,\pi,\delta)=(\la^{\prime}, \pi^{\prime}, \delta^{\prime})$,
then $(\la^{\prime}, \pi^{\prime})={\mathscr T}(\la,\pi)$.

For $(\la,\pi, \delta)\in {\tilde{\mathcal V}}(\R)$, $s\in {\mathbb R}$, let ${\tilde n}(\la,\pi,\delta, s)$ be defined by the formula
$$
{\mathcal U}^{{\tilde n}(\la,\pi,\delta, s)}(e^s\la,\pi,e^{-s}\delta)\in {\mathcal V}_{1,\tau}(\R).
$$

Endow the space ${\tilde{\mathcal V}}(\R)$  with a matrix cocycle ${\overline {\mathscr {{\mathcal A}}}}^t$
over the flow $P^s$ given by the formula
$$
{\overline {\mathcal {A}}}^t(s, (\la,\pi,\delta))={{\mathcal A}}^t({\tilde n}(\la,\pi,\delta, s), (\la,\pi))
$$

and introduce the corresponding skew-product flow
$$
P^{s,\overline{\mathcal{A}}^t}:\widetilde{\mathcal{V}}(\mathcal{R})\times\mathbb{R}^m\to
\widetilde{\mathcal{V}}(\mathcal{R})\times\mathbb{R}^m
$$
by the formula
$$
P^{s,\overline{\mathcal{A}}^t}({\mathscr X}, v)=
(P^s{\mathscr X},{\overline{\mathcal{A}}^t}({\mathscr X}, s)v).
$$
We also have a natural cocycle ${\overline {\mathcal A}}$
over the inverse flow $P^{-s}$ given by the formula
$$
{\overline {\mathcal A}}({\mathscr X}, s)=({\overline {\mathcal A}}^t(P^{-s}{\mathscr X}, s))^t
$$
and the natural skew-product flow
$$
P^{-s,\overline{\mathcal{A}}}:\widetilde{\mathcal{V}}(\mathcal{R})\times\mathbb{R}^m\to \widetilde{\mathcal{V}}(\mathcal{R})\times\mathbb{R}^m
$$
defined by the formula
$$
P^{-s,\overline{\mathcal{A}}}({\mathscr X}, v)=(P^{-s}{\mathscr X},{\overline{\mathcal{A}}}({\mathscr X}, s)v).
$$
The strongly unstable Oseledets bundle of the cocycle ${\overline {\mathcal A}}$ will
be seen to describe all the measures
$\Phi^-\in \BB^-$ in the same way in which
 the strongly unstable Oseledets bundle of the
cocycle $\overline{\mathcal{A}}^t$ describes all the measures
$\Phi^+\in \BB^+$.

{\bf Remark.} The Kontsevich-Zorich cocycle is isomorphic to the inverse of its dual (see, e.g.,
Statement 2 in
Proposition \ref{veechtokz} below).
This ``self-duality'' is, however, not used in
the construction and characterization of finitely-additive invariant measures.
The duality between the spaces $\BB^+$ and $\BB^-$ corresponds to
the duality between the cocycle and its transpose, that is,
in our notation, between $\overline{\mathcal{A}}^t$ and $\overline{\mathcal{A}}$ :
such duality takes place for any invertible matrix-valued cocycle over any measure-preserving flow.

\subsubsection{The correspondence between cocycles.}

To a connected component $\HH$ of the space $\modk$ one can assign
a Rauzy
class $\R$ in such a way that the following is true \cite{veech, kz}.

\begin{theorem}[\rm Veech]
\label{zipmodule} There exists a finite-to-one measurable map
$\pi_{\R}:\tilde{\mathcal V}(\R)\to\HH$ such that $\pi_{\R}\circ P^t=g_t\circ \pi_{\R}$.
The image of $\pi_{\R}$ contains all abelian differentials whose vertical and
horizontal foliations are both minimal.
\end{theorem}

As before, let ${\mathbb H}^1(\HH)$ be the fibre bundle over $\HH$ whose fibre at a point $(M, \omega)$ is the cohomology group
$H^1(M, {\mathbb R})$. The Kontsevich-Zorich cocycle ${\bf A}_{KZ}$ induces
a skew-product flow $g_s^{{\bf A}_{KZ}}$ on ${\mathbb H}^1(\HH)$ given by the formula
$$
g_s^{{\bf A}_{KZ}}({\bf X}, v)=({\bf g}_s{\bf X}, {\bf A}_{KZ}v), \ {\bf X}\in\HH, v\in H^1(M, {\mathbb R}).
$$

Following Veech \cite{veech}, we now explain the connection between the Kontsevich-Zorich cocycle
${\bf A}_{KZ}$ and the cocycle  ${\overline {\mathcal A}}^t$.

For any irreducible permutation $\pi$ Veech \cite{veechamj} defines an alternating matrix $L^{\pi}$ by
setting $L_{ij}^{\pi}=0$ if $i=j$ or if $i<j, \pi i<\pi j$, $L_{ij}^{\pi}=1$ if $i<j, \pi i>\pi j$,
$L^{\pi}_{ij}=-1$ if $i>j, \pi i<\pi j$ and denotes by $N(\pi)$ the kernel of $L^{\pi}$ and by $H(\pi)=L^{\pi}({\mathbb R}^m)$ the image of $L^{\pi}$. The dimensions of $N(\pi)$ and $H(\pi)$ do not change as $\pi$ varies in $\R$, and, furthermore,
Veech \cite{veechamj} establishes the following properties of the spaces $N(\pi)$, $H(\pi)$.
\begin{proposition}
\label{veechtokz}
Let $c=a$ or $b$. Then
\begin{enumerate}
\item $H(c\pi)={\mathcal A}^t(c,\pi)H(\pi)$,
$N(c\pi)={\mathcal A}^{-1}(c,\pi)N(\pi)$;
\item the diagram
$$
\begin{CD}
{\mathbb R}^m/N(\pi)@ >L^{\pi}>> H(\pi) \\
@ VV {\mathcal A}^{-1}(\pi, c)V          @VV {\mathcal A}^t(\pi, c) V   \\
{\mathbb R}^m/N(c\pi)@ >L^{c\pi}>> H(c\pi) \\
\end{CD}
$$
is commutative and each arrow is an isomorphism.

\item For each $\pi$ there exists a basis ${\bf v}_{\pi}$ in $N(\pi)$ such that
the map ${\mathcal A}^{-1}(\pi, c)$ sends every element of ${\bf v}_{\pi}$ to an element of ${\bf v}_{c\pi}$.
\end{enumerate}
\end{proposition}

Each space $H^{\pi}$ is thus endowed with a natural anti-symmetric bilinear form ${\mathcal L}_{\pi}$
defined, for $v_1, v_2\in H(\pi)$, by the formula
\begin{equation}
\label{bilpi}
{\mathcal L}_{\pi}(v_1, v_2)=\langle v_1, (L^{\pi})^{-1}v_2\rangle.
\end{equation}

(The vector $(L^{\pi})^{-1}v_2$ lies in ${\mathbb R}^m/N(\pi)$; since for all $v_1\in H(\pi)$, $v_2\in N(\pi)$
by definition we have $\langle v_1, v_2\rangle=0$, the right-hand side is well-defined.)

Consider the ${\mathscr T}^{{\mathcal A}^t}$-invariant subbundle
${\mathscr H}(\Delta(\R))\subset \Delta(\R)\times {\mathbb R}^m$ given by the formula
$$
{\mathscr H}(\Delta(\R))=\{((\la,\pi), v), (\la,\pi)\in\Delta(\R), v\in H(\pi)\}.
$$
as well as a quotient bundle
$$
{\mathscr N}(\Delta(\R))=\{((\la,\pi), v), (\la,\pi)\in\Delta(\R), v\in {\mathbb R}^m/N(\pi)\}.
$$

The bundle map ${\mathscr L}_{\R}: {\mathscr H}(\Delta(\R))\to {\mathscr N}(\Delta(\R))$ given by
${\mathscr L}_{\R}((\la,\pi), v)=((\la,\pi), L^{\pi}v)$ induces a bundle isomorphism between
${\mathscr H}(\Delta(\R))$ and ${\mathscr N}(\Delta(\R))$.

Both bundles can be naturally lifted to bundles ${\mathscr H}({\tilde {\mathcal V}}(\R))$, ${\mathscr N}({\tilde {\mathcal  V}}(\R))$
over the space ${\tilde {\mathcal V}}(\R)$ of zippered rectangles; they are naturally invariant
under the corresponding skew-product flows  $P^{s, \overline {\mathcal A}^t}$, $P^{-s, \overline {\mathcal A}}$,
and the map ${\mathscr L}_{\R}$ lifts to a bundle isomorphism between
${\mathscr H}({\tilde {\mathcal V}}(\R))$ and ${\mathscr N}({\tilde {\mathcal V}}(\R))$.

Take ${\mathscr X}\in {\tilde {\mathcal V}}(\R)$ and write  $\pi_{\R}({\mathscr X})=(M({\mathscr X}), \omega({\mathscr X}))$.
Veech \cite{veechamj2} has shown that the map $\pi_{\R}$ lifts to a bundle epimorphism
 ${\tilde \pi}_{\R}$
from ${\mathscr H}({\tilde {\mathcal V}}(\R))$ onto ${\mathbb H}^1(\HH)$ that intertwines
the cocycle ${\overline {\mathcal A}}^t$ and the Kontsevich-Zorich cocycle ${\bf A}_{KZ}$:

\begin{proposition}[Veech]
\label{identif-coc-veech}

For almost every ${\mathscr X}\in {\tilde {\mathcal V}}(\R)$, ${\mathscr X}=(\la,\pi,\delta)$,
there exists an isomorphism ${\mathcal I}_{\mathscr X}: H(\pi)\to H^1(M({\mathscr X}), {\mathbb R})$
such that
\begin{enumerate}
\item
the map
${\tilde \pi}_{\R}: {\mathscr H}({\tilde {\mathcal V}}(\R)) \to {\mathbb H}^1(\HH)$ given by
$$
{\tilde \pi}_{\R}({\mathscr X}, v)=(\pi_{\R}({\mathscr X}), {\mathcal I}_Xv)
$$
induces a measurable bundle epimorphism from ${\mathscr H}({\tilde {\mathcal V}}(\R))$ onto ${\mathbb H}^1(\HH)$,
which is an isomorphic on each fibre;
\item the diagram
$$
\begin{CD}
{\mathscr H}({\tilde {\mathcal V}}(\R))@ > {\tilde \pi}_{\R}  >> {\mathbb H}^1(\HH) \\
@ VVP^{s, \overline {\mathcal A}^t} V          @VV g_s^{{\bf A}_{KZ}} V   \\
 {\mathscr H}({\tilde {\mathcal V}}(\R))@ > {\tilde \pi}_{\R}  >> {\mathbb H}^1(\HH)\\
\end{CD}
$$
is commutative;
\item for ${\mathscr X}=(\la,\pi,\delta)$, the isomorphism ${\mathcal I}_X$ takes  the bilinear form ${\mathcal L}_{\pi}$ on $H(\pi)$,
defined by (\ref{bilpi}), to the cup-product on  $H^1(M({\mathscr X}), {\mathbb R})$.
\end{enumerate}
\end{proposition}

Proof: Recall that to each rectangle $\Pi_i$ Veech \cite{veechamj} assigns a cycle $\gamma_i(\la,\pi, \delta)$ in the homology
group $H_1(M, {\mathbb Z})$: if $P_i$ is the left bottom corner of $\Pi_i$ and $Q_i$ the left top corner, then the cycle is
the union of the vertical interval $P_iQ_i$ and the horizontal subinterval of $I^{(0)}(\la,\pi, \delta)$ joining $Q_i$ to $P_i$. It is
clear that the cycles $\gamma_i(\la,\pi, \delta)$ span $H_1(M, {\mathbb Z})$; furthermore, Veech shows that the cycle
$t_1\gamma_1+\dots +t_m\gamma_m$ is homologous to $0$ if and only if $(t_1, \dots, t_m)\in N(\pi)$.
We thus obtain an identification of ${\mathbb R}^m/N(\pi)$ and $H_1(M, {\mathbb R})$.
Similarly, the subspace of ${\mathbb R}^m$ spanned  by the vectors $(f(\gamma_1), \dots, f(\gamma_m))$, $f\in H^1(M, {\mathbb R})$,
is precisely $H(\pi)$.
The  identification of the bilinear form ${\mathcal L}_{\pi}$ with the cup-product is established in Proposition 4.19 in \cite{V3}.

Let ${\Prob}_{\mathcal V}$ be an ergodic $P^s$-invariant probability measure for the flow $P^s$ on ${\mathcal V}(\R)$
and let $\Prob_{\HH}=(\pi_{\R})_*{\Prob}_{\mathcal V}$ be the corresponding ${\bf g}_s$-invariant measure on $\HH$.
Let ${\mathcal E}^u_{{\Prob}_{\mathcal V}}({\tilde {\mathcal V}}(\R))$
be the strongly unstable bundle of the cocycle
${\overline {\mathcal A}}^t$.
 By Proposition \ref{veechtokz}, the bundle
${\mathcal E}^u_{{\Prob}_{\mathcal V}}({\tilde {\mathcal V}}(\R))$ is a
subbundle of ${\mathscr H}({\tilde {\mathcal V}}(\R))$.
It therefore follows from Proposition \ref{identif-coc-veech} that
the map ${\tilde \pi}_{\R}$ isomorphically identifies
the strongly unstable bundles of the cocycles ${\overline {\mathcal A}^t}$ and ${{\bf A}_{KZ}}$; this identification
is equivariant with respect to the natural actions of the skew-product flows
$P^{s, \overline {\mathcal A}^t}$ and $g_s^{{\bf A}_{KZ}}$  on the corresponding bundles.

\subsubsection{The correspondence between measures.}

\begin{proposition}
\label{cor-mes}

Let $\Prob$ be an ergodic ${\bf g}_s$--invariant probability measure on $\HH$. Then there exists
an ergodic $P^s$-invariant probability measure ${\Prob}_{\mathcal V}$
on ${\mathcal V}(\R)$ such that $$
\Prob=(\pi_{\R})_*{\Prob}_{\mathcal V}.
$$
\end{proposition}

This proposition is a corollary of the following general statement.

\begin{proposition}
\label{finfibre}
Let $Z_1, Z_2$ be standard Borel spaces, let $g^1_s:Z_1\to Z_1$, $g^2_s: Z_2\to Z_2$ be measurable flows,
and let $\pi_{12}:Z_1\to Z_2$ be a Borel measurable   map such that
\begin{enumerate}
\item for any $z_2\in Z_2$ the preimage $\{\pi_{12}^{-1}(z_2)\}$ of $z_2$ is finite;
\item the map $\pi_{12}$ intertwines the flows $g^1_s$, $g^2_s$ in the sense that
the diagram
$$
\begin{CD}
Z_1@ > {\pi}_{12}  >> Z_2 \\
@ VVg^1_s V          @VV g_s^{2} V   \\
Z_1@ > {\pi}_{12}  >> Z_2\\
\end{CD}
$$
is commutative.
\end{enumerate}

Then for any Borel $g_s^2$-invariant ergodic probability measure $\Prob_2$ on $Z_2$,
there exists a Borel $g_s^1$-invariant ergodic probability measure $\Prob_1$  on $Z_1$
such that $$(\pi_{12})_*(\Prob_1)=\Prob_2.$$
\end{proposition}

The proof of Proposition \ref{finfibre} is routine: first, note that, by ergodicity,
for $\Prob_2$-almost every $z_z\in Z_2$ the cardinality of the
preimage $\{\pi_{12}^{-1}(z_2)\}$ of $z_2$ is constant; now consider the normalized product ${\tilde \Prob}_1$
of $\Prob_2$ and the counting measure in the preimage; the measure ${\tilde \Prob}_1$ is by definition
$g_s^1$-invariant, and for the measure $\Prob_1$ one may take
an ergodic component, in fact, almost every ergodic component, of the measure ${\tilde \Prob}_1$.

\subsection{A strongly biregular sequence of partitions corresponding to a zippered rectangle}

Given a zippered rectangle ${\mathscr X}$, we shall speak of its vertical and horizontal foliations, H\"{o}lder cocycles and so on, meaning the corresponding objects for the underlying abelian differential, and we shall use the notation $\mathfrak{B}^+({\mathscr X})$, $\mathfrak{B}^-({\mathscr X})$, $\mathfrak{B}^+_c({\mathscr X})$, $\mathfrak{B}^-_c({\mathscr X})$ for the corresponding spaces of finitely-additive measures.

Recall that, by construction, a zippered rectangle carries the partition
$$\pi_0({\mathscr X})=\Pi^{(0)}_1\sqcup\ldots\sqcup\Pi^{(0)}_m$$
into weakly admissible rectangles.

The Rauzy--Veech expansion of a zippered rectangle now yields a Markovian sequence of partitions $\pi_n$, $n\in\mathbb{Z}$. To construct
it, first take ${\mathscr X}\in{\mathcal Y}(\mathcal{R})$ and recall that we have a natural ``tautological''
identification map
$$\mathfrak{t}_{\overline{{\mathscr T}}}: M({\mathscr X})\to M(\overline{{\mathscr T}}{\mathscr X}).$$
Now set
\begin{equation}
\pi_n({\mathscr X})=\left(\mathfrak{t}_{\overline{{\mathscr T}}}\right)^{-n}\,\pi_0\, \left({\overline{{\mathscr T}}}^n{\mathscr X}\right)
\end{equation}

By definition, the sequence $\pi_n({\mathscr X})$, $n\in\mathbb{Z}$, is Markovian. Minimality
of the horizontal and vertical flows implies exactness of the sequence $\pi_n({\mathscr X})$.
For a general zippered rectangle ${\mathscr X}'\in\mathcal{V}_{1,\tau}(\mathcal{R})$,
write $$
{\mathscr X}'=P^{s_0}\,{\mathscr X}, {\mathscr X}\in{\mathcal Y}_1(\mathcal{R}), 0\leqslant s_0<\tau({\mathscr X})
$$
and set $$\pi_n({\mathscr X})=\mathfrak{t}_{s_0}\left(\pi_n({\mathscr X})\right).$$
(informally, carry over sequence $\pi_n$ from the ``closest'' zippered rectangle lying on the transversal
${\mathcal Y}_1(\mathcal{R})$.)

\begin{lemma}
\label{sb-zip}
For $\mathbb{P}_{\mathcal{V}}$-almost every zippered rectangle ${\mathscr X}$ there exists a sequence $n_k\in\mathbb{Z}$, $n_0=0$, $n_k<n_{k+1}$, $k\in\mathbb{Z}$, such that the
sequence $\mathfrak{A}=\left(A\left(\pi_{n_k},\pi_{n_{k+1}}\right)\right)$ of adjacency matrices
of the exact Markovian subsequence $\pi_{n_k}({\mathscr X})$, $k\in\mathbb{Z}$
satisfies the following.
\begin{enumerate}
\item $\mathfrak{A}$ is an SB-sequence,
\item the space $\mathfrak{B}^+(\mathfrak{A})$ coincides with the strongly unstable space of
the cocycle $\overline{\mathcal{A}}^t$ at the point ${\mathscr X}$.
\end{enumerate}
\end{lemma}

The proof of the Lemma is routine: one chooses a Rauzy-Veech matrix $Q$ of the form
$$
Q=Q_1\,Q_2,
$$
where $Q_1$ and $Q_2$ are Rauzy-Veech matrices all whose entries are positive and such that $\mathbb{P}_{\mathcal{V}}$-almost all zippered rectangles ${\mathscr X}$ contain infinitely many occurrences of the matrix $Q$ both in the past and in the future. The sequence $n_k$ is then the sequence of consecutive occurrences of the matrix $Q$.
Each adjacency matrix $A\left(\pi_{n_k},\pi_{n_{k+1}}\right)$ now has the form $Q_2\widetilde{A}Q_1$, where $\widetilde{A}$ is an integer matrix with non-negative entries. It follows from the Oseledets
Multiplicative Ergodic
Theorem and the Oseledets-Pesin Reduction Theorem (Theorem 3.5.5 on p.77 in
\cite{pesinbarreira})
that $\mathfrak{A}$ is an SB-sequence and that $\mathfrak{B}^+(\mathfrak{A})$ coincides with the strongly
unstable space of the cocycle $\overline{\mathcal{A}}^t$ at the zippered rectangle
${\mathscr X}$. The proof of the Lemma is complete.

\subsection{The renormalization action of the Teichm{\"u}ller flow on the space of finitely-additive measures.}

We have the evaluation map
$$\eval^+_{\mathscr X}\colon\mathfrak{B}^+({\mathscr X})\to\mathbb{R}^m$$

which  to a finitely-additive measure $\Phi^+\in\mathfrak{B}^+$ assigns the vector of its values on vertical arcs
of the rectangles $\Pi_i^{(0)}$, $i=1,\ldots,m$. We must now check that the map $\eval^+_{\mathscr X}$ is
indeed an isomorphism between the space $\mathfrak{B}^+({\mathscr X})$ and the strongly unstable space of
the cocycle ${\overline{\mathcal{A}}}^t$.

Introduce a measurable fibre bundle $\mathfrak{B}^+\widetilde{\mathcal{V}}(\mathcal{R})$ over the Veech space $\mathcal{V}(\mathcal{R})$ by setting
$$
\mathfrak{B}^+\widetilde{\mathcal{V}}(\mathcal{R})=\left\{\bigl({\mathscr X},\Phi^+\bigr)\colon{\mathscr X}\in\widetilde{\mathcal{V}}(\mathcal{R}), \Phi^+\in\mathfrak{B}^+({\mathscr X})\right\}\,.
$$

Extend the map $\eval^+_{\mathscr X}$ to a bundle morphism
$$\eval^+\colon\mathfrak{B}^+\widetilde{\mathcal{V}}(\mathcal{R})\to\widetilde{\mathcal{V}}(\mathcal{R})\times\mathbb{R}^m\,,$$
given by the formula:
$$\eval^+\bigl({\mathscr X},\Phi^+\bigr)=\bigl({\mathscr X},\eval^+_{\mathscr X}(\Phi^+)\bigr)\,.$$

By definition, the map $\eval^+$ intertwines the action of the flow $P^{s,{\mathscr X}}$ on the
bundle $\mathfrak{B}^+\widetilde{\mathcal{V}}(\mathcal{R})$ with that of the flow
$P^{s,\overline{\mathcal{A}}^t}$ on the trivial bundle $\widetilde{\mathcal{V}}(\mathcal{R})\times\mathbb{R}^m$.

Recall that for any $s\in {\mathbb R}$ we have a natural ``tautological" map
$$
{\mathfrak t}_s: M({\mathscr X})\to M(P^s{\mathscr X}).
$$
given by (\ref{def-t-s}). The bundle
$
{\mathfrak B}^+{\tilde{\mathcal V}}(\R)$ is now
endowed with a natural renormalization flow $P^{s, {\mathfrak B}^+}$ given by the formula
$$
P^{s, {\mathfrak B}^+}({\mathscr X}, \Phi^+)=(P^s{\mathscr X}, ({\mathfrak t}_s)_*\Phi^+).
$$

We furthermore have a bundle morphism
$$
\eval^+: {\mathfrak B}^+{\tilde{\mathcal V}}(\R)\to {\tilde{\mathcal V}}(\R)\times {\mathbb R}^m
$$
given by the formula
$$
\eval^+({\mathscr X}, \Phi^+)=({\mathscr X}, \eval_{{\mathscr X}}^+(\Phi^+).
$$
The identification of cocycles now gives us the following
\begin{proposition}
\label{hypzip}
Let ${\Prob}_{\mathcal V}$ be an ergodic $P^s$-invariant
probability measure for the flow $P^s$ on ${\mathcal V}(\R)$.
We have a commutative diagram
$$
\begin{CD}
{\mathfrak B}^+{\tilde {\mathcal V}}(\R)@ > \eval^+  >> {\tilde{\mathcal V}}(\R)\times {\mathbb R}^m \\
@ VVP^{s, {\mathfrak B}^+} V          @VVP^{s, \overline {\mathcal A}^t} V   \\
{\mathfrak B}^+{\tilde {\mathcal V}}(\R)@ > \eval^+  >> {\tilde{\mathcal V}}(\R)\times {\mathbb R}^m \\
\end{CD}
$$
The map $\eval^+$ is injective in restriction each fibre.
For ${\Prob}_{\mathcal V}$-almost every ${\mathscr X}\in  {\mathfrak B}^+{\tilde{\mathcal V}}(\R)$, the map
$\eval^+$ induces an isomorphism between the space ${\mathfrak B}^+({\mathscr X})$ and
the strongly unstable Oseledets subspace of the cocycle $\overline {\mathcal A}^t$ at the point ${\mathscr X}$.
\end{proposition}
Proof.
Let $\pi_{n_k}$ be the sequence of partitions given by Lemma \ref{sb-zip}, and let
${\mathfrak A}$ be the corresponding SB-sequence of matrices.
Since $\mathfrak{A}$ is an SB-sequence, the map $\eval^+_{\mathscr X}$ induces an
isomorphism between $\mathfrak{B}^+({\mathscr X})$ and $\mathfrak{B}^+(\mathfrak{A})$ (recall here that $n_0=0$). Since $\mathfrak{B}^+(\mathfrak{A})$ coincides with the unstable space of the
cocycle $\overline{\mathcal{A}}^t$, the Proposition is proved completely.

Using Proposition \ref{hypzip}, we will identify
the action of $P^{s, {\mathfrak B}^+}$ on ${\mathfrak B}^+({\tilde {\mathcal V}}(\R)$
with the action of  $P^{s, \overline {\mathcal A}^t}$ on the strongly unstable Oseledets subbundle of
${\tilde{\mathcal V}}(\R)\times {\mathbb R}^m$, and speak of the action of the cocycle
  ${\overline {\mathcal A}}^t$ on the space of finitely-additive measures in this sense.

This renormalization action of the flow $P^s$ on the space of finitely-additive measures will
play a key r{\^o}le in the proof of the limit theorems in the next Section. We close this section by giving
a sufficient condition for the equality ${\mathfrak B}^+(X, \omega)={\mathfrak B}^+_c(X, \omega)$.

\subsection{A sufficient condition for the equality ${\mathfrak B}^+(X, \omega)={\mathfrak B}^+_c(X, \omega)$}

Let $({\mathcal X}, \mu)$ be a probability space endowed with a $\mu$-preserving transformation
$T$ or flow $g_s$ and an integrable linear cocycle $A$ over $g_s$  with values in ${\mathrm GL}(m, {\mathbb R})$.

For $p\in{\mathcal X}$ let $E_{0, x}$ be the the {\it neutral} subspace of $A$ at $p$, i.e., the
Lyapunov subspace of the cocycle ${A}$ corresponding to the
Lyapunov exponent $0$.
We say  that ${A}$ {\it acts isometrically on its neutral subspaces}
if for almost any $p$ there exists
an inner product $\langle\cdot\rangle_{p}$ on ${\mathbb R}^m$ which depends on
$p$ measurably and satisfies
$$
\langle { A}(1, p)v,{ A}(1, p)v\rangle_{g_sp}=\langle v, v\rangle_{p}, \ v\in E_{0, p}
$$
for all $s\in {\mathbb R}$ (again, in the case of a transformation, $g_s$ should be replaced by $T$ in this formula).

The third statement of Proposition \ref{veechtokz} has the following immediate
\begin{corollary}
\label{hypzip1}
Let ${\Prob}_{\mathcal V}$ be a Borel ergodic $P^s$-invariant probability measure  on ${\mathcal V}(\R)$
and let $\Prob=(\pi_{\R})_*{\Prob}_{\mathcal V}$ be the corresponding ${\bf g}_s$-invariant measure on $\HH$.
If the Kontsevich-Zorich cocycle acts isometrically on its neutral subspace with respect to $\Prob$, then the
cocycle ${\overline {\mathcal A}}^t$  also acts isometrically on its neutral subspace with respect to $\Prob_{\mathcal V}$.
\end{corollary}

Note that the hypothesis of Corollary \ref{hypzip1} is satisfied, in particular, for the Masur-Veech
smooth measure on the moduli space of abelian differentials.

The following proposition is clear from the definitions.
\begin{proposition}
\label{hyperboliccriterion}
Let ${\Prob}_{\mathcal V}$ be a Borel ergodic $P^s$-invariant probability measure on ${\mathcal V}(\R)$
such that the
cocycle ${\overline {\mathcal A}}^t$  acts isometrically on its neutral subspace with respect to $\Prob_{\mathcal V}$.
Let $\Prob=(\pi_{\R})_*{\Prob}_{\mathcal V}$ be the corresponding ${\bf g}_s$-invariant ergodic measure on $\HH$.
Then for $\Prob$-almost every abelian differential $(M, \omega)$, we have the equality
$$
{\mathfrak B}^+(M, \omega)={\mathfrak B}^+_c(M, \omega).
$$
\end{proposition}

In other words,  if the
cocycle ${\overline {\mathcal A}}^t$  acts isometrically on its neutral subspace with respect
to $\Prob_{\mathcal V}$, then any continuous finitely-additive measure must in fact be H{\"o}lder.
Note that the assumptions of the proposition are verified,
in particular, for the Masur-Veech
smooth measure on the moduli space of abelian differentials.
To prove Proposition \ref{hyperboliccriterion} we use Proposition \ref{veechtokz}, which implies that
if the
cocycle ${\overline {\mathcal A}}^t$  acts isometrically on its neutral subspace with respect to
$\Prob_{\mathcal V}$, then $\Prob$-almost every abelian differential
$(M, \omega)$ admits an exact Markovian sequence of partitions
whose sequence of adjacency matrices is a hyperbolic SB-sequence: which, in turn, is sufficient for the equality
$$
{\mathfrak B}^+(M, \omega)={\mathfrak B}^+_c(M, \omega).
$$

\section{Proof of the Limit Theorems.}
\subsection{Outline of the proof.}

The main element in the proof of the limit theorem is the
renormalization action of the Teichm{\"u}ller flow $P^s$
on the bundle ${\mathfrak B}^+{\tilde {\mathcal V}}({\mathcal R})$.

Start with the case when the second Lyapunov exponent of the cocycle
${\overline {\mathcal A}}^t$ is positive and simple with respect to a $P^s$-invariant
ergodic probability measure $\Prob_{{\mathcal V}}$ on ${\tilde {\mathcal V}}({\mathcal R})$.
Then, by Theorem \ref{multiplicmoduli},  for $\Prob_{{\mathcal V}}$-almost every zippered rectangle ${\mathscr X}$, and
a generic weakly Lipschitz function $f$ of zero average the ergodic integral
$$
\int_0^{\exp(s)}f\circ h_t^+(x)dt
$$
is approximated by an expression of the form
$$
const\cdot \Phi^+_{2, {\mathscr X}}(x, e^s),
$$
where the constant depends on $f$, and
$\Phi^+_{2, {\mathscr X}}\in {\mathfrak B}^+({\mathscr X})$
is a cocycle belonging to the second Lyapunov subspace of the cocycle
${\overline {\mathcal A}}^t$.
Note that the cocycle  $\Phi^+_{2, {\mathscr X}}$ is defined up to multiplication
by a scalar; the double cover $\HH^{\prime}$ over the space $\HH$ in the formulation of
the limit theorem is considered precisely in order to distinguish between positive and negative scalars.

Now,  Proposition \ref{hypzip} implies that the normalized distribution
of the random variable $\Phi^+_{2, {\mathscr X}}(x, e^s)$ (considered as a function of $x$ with fixed $s$)
coincides with the normalized distribution of the random variable  $\Phi^+_{2, P^s{\mathscr X}}(x, 1)$.
Assigning to a zippered rectangle ${\mathscr X}$ the normalized distribution
of the random variable $\Phi^+_{2, {\mathscr X}}(x, e^s)$ (considered as a function of $x$ with fixed $s$)
now  yields the desired map ${\mathcal D}_2^+$ from the space of zippered rectangles
(more precisely, from its double cover) to the space of distributions.
The fact that the normalized distributions of the ergodic integrals are approximated by
the image under the map ${\mathcal D}_2^+$ of the orbit of our zippered rectangle
under the action of the Teichm{\"u}ller flow $P^s$ follows now from the
asymptotic expansion of Theorem \ref{multiplicmoduli}.

\subsection{The case of the simple second Lyapunov exponent.}
\subsubsection{The leading term in the asymptotic for the ergodic integral.}
We fix a $P^s$-invariant
ergodic probability measure $\Prob_{{\mathcal V}}$ on ${\tilde {\mathcal V}}({\mathcal R})$
 and start with the case in which the second Lyapunov exponent of the cocycle ${\overline {\mathcal A}}^t$
is positive and simple with respect to the measure $\Prob_{{\mathcal V}}$.
Consider the Oseledets subspace $E_{1, \oomega}^u={\mathbb R}h_{\oomega}$ corresponding to the top Lyapunov exponent $1$
and the one-dimensional Oseledets subspace $E_{2, \oomega}^u$ corresponding to the second Lyapunov exponent.
Furthermore, let $E_{\geq 3, \oomega}^u$ be the subspace corresponding to the remaining Lyapunov exponents.

We have then the decomposition
$$
E_{\oomega}^u=E_{1, \oomega}^u\oplus E_{2, \oomega}^u\oplus E_{\geq 3, \oomega}^u.
$$
Denote $\BB_{1, \oomega}^+={\mathbb R}\nu^+$,
$\BB_{2, \oomega}^+$, $\BB_{\geq 3, \oomega}^+$ the corresponding spaces of H{\"o}lder cocycles.

A similar decomposition holds for the dual space ${\tilde E}^u$, the strongly unstable space
of the cocycle ${\overline {\mathcal A}}$:
$$
{\tilde E}_{\oomega}^u={\tilde E}_{1, \oomega}^u\oplus {\tilde E}_{2, \oomega}^u\oplus {\tilde E}_{\geq 3, \oomega}^u.
$$
Again, denote $\BB_{1, \oomega}^-={\mathbb R}\nu^-$,
$\BB_{2, \oomega}^-$, $\BB_{\geq 3, \oomega}^-$ the corresponding spaces of H{\"o}lder cocycles.

Choose $\Phi_2^+\in \BB_{2, \oomega}^+$, $\Phi_2^-\in\BB_{2, \oomega}^-$ in such a way that
$$
\langle \Phi_2^+, \Phi_2^-\rangle =1.
$$
Take $f\in Lip_{w}^+({\mathscr X})$, $x\in {\mathscr X}$, $T\in {\mathbb R}$  and observe that the expression
\begin{equation}
\label{mphitwo}
m_{\Phi_2^-}(f)\Phi_2^+(x,T)
\end{equation}
does not depend on the precise choice of $\Phi_2^{\pm}$ (we have the freedom of multiplying $\Phi_2^+$
by an arbitrary scalar, but then $\Phi_2^-$ is divided by the same scalar).

Now for $f\in Lip_{w}^+({\mathscr X})$ write
$$
\Phi_f^+(x,T)=(\int_{\mathscr X} fd\nu)\cdot T + m_{\Phi_2^-}(f)\Phi_2^+(x,T)+\Phi_{3, f}^+(x,T),
$$
where $\Phi_{3, f}^+\in {\mathcal I}_{\oomega}({\tilde E}_{\geq 3, \oomega}^u)$.

In particular, there exist two positive constants  $C$ and $\alpha$ depending only on $\Prob$ such that
for any  function $f$ satisfying
$$
f\in Lip_{w}^+({\mathscr X}), \int_{\mathscr X} fd\nu=0,
$$
 we have the estimate
\begin{equation}
\label{phitwoapprox}
\left|\int_0^T f\circ h_t^+(x) dt - m_{\Phi_2^-}(f)\Phi_2^+(x,T)\right|\leq C||f||_{Lip} T^{{\overline \theta_2}-\alpha}.
\end{equation}

\subsubsection{The growth of the variance.}

In order to estimate the variance of the random
variable $\int_0^T f\circ h_t^+(x) dt$, we start by studying the growth of
the variance of the random variable $\Phi_{2, \oomega}^+(x,T)$ as $T\to\infty$.

Recall that $\ee_{\nu}\Phi_{2, \oomega}^+(x,T)=0$ for all $T$, while
$Var_{\nu}\Phi_{2, \oomega}^+(x,T)\neq 0$ for $T\neq 0$.
Recall that for a cocycle $\Phi^+\in\BB_{\oomega}^+$,
$\Phi^+={\mathcal I}^+_{\oomega}(v)$,
we have defined its norm $|\Phi^+|$ by the formula $|\Phi^+|=|v|$.
Introduce a multiplicative cocycle $H_2(s, \oomega)$ over the flow $P^s$ by the formula
\begin{equation}
H_2(s, \oomega)=\frac{|{\overline {\mathcal A}}^t(s, \oomega)v|}{|v|}, \ v\in E_{2, \oomega}^u,\  v\neq 0.
\end{equation}
 Observe that the right-hand side does not depend on the specific choice of $v\neq 0$.

By definition, we now have
\begin{equation}
\lim_{s\to\infty}\frac{\log H_2(s, \oomega)}{s}={\overline \theta}_2.
\end{equation}

\begin{proposition}
There exists a positive measurable
function $V:\Oomega\to {\mathbb R}_+$ such that
the following equality holds for $\Prob_{{\mathcal V}}$-almost all $\oomega\in\Oomega$:
\begin{equation}
{Var_{\nu} {\Phi_2^+}(x,T)}=V(P^s\oomega)|\Phi_2^+|^2(H_2(s, \oomega))^2.
\end{equation}
\end{proposition}
Indeed, the function $V(\oomega)$ is given by
\begin{equation}
V(\oomega)=\frac{{Var_{\nu} {\Phi_2^+}(x,1)}}{|\Phi^+_2|^2},
\end{equation}
and the Proposition is an immediate corollary of Proposition \ref{hypzip}.
Observe that the right-hand side does not depend on  a particular choice of $\Phi_2^+\in {\BB}_{2, \oomega}^+$,
$\Phi_2^+\neq 0$.

Using (\ref{phitwoapprox}), we now proceed to
estimating the growth of the variance of the ergodic integral  $$\int_0^T f\circ h_t^+(x) dt.$$

We use the same notation as in the Introduction: for
$\tau\in [0,1]$, $s\in {\mathbb R}$, a real-valued
$f\in Lip_{w,0}^+({\mathscr X})$ we write
\begin{equation}
\label{sfstaux-symb}
{\mathfrak S}[f,s;\tau, x]=\int_0^{\tau\exp(s)} f\circ h^{+}_t(x)dt.
\end{equation}

As before, let $\nu$ be the Lebesgue measure on the surface $M({\mathscr X})$ corresponding to the zippered rectangle ${\mathscr X}$.
As before, as $x$ varies
in the probability space $(M({\mathscr X}), \nu)$, we obtain a random element
of $C[0,1]$. In other words, we have a
 random variable
\begin{equation}
{\mathfrak S}[f,s]: (M({\mathscr X}), \nu)\to C[0,1]
\end{equation}
defined by the formula (\ref{sfstaux-symb}).

For any fixed $\tau\in [0,1]$  the formula (\ref{sfstaux-symb}) yields
a real-valued random variable
\begin{equation}
{\mathfrak S}[f,s; \tau]: (M({\mathscr X}), \nu)\to {\mathbb R},
\end{equation}
whose expectation, by definition, is zero.

\begin{proposition}
\label{varf}
There exists $\alpha>0$ depending only on $\Prob_{\mathcal V}$ and a positive measurable
function $C:\Oomega\times\Oomega\to {\mathbb R}_+$ such that
the following holds for $\Prob_{\mathcal V}$-almost all $\oomega\in\Oomega$ and all $s>0$.
Let $\Phi_{2, \oomega}^+\in\BB^+_{2, \oomega}$, $\Phi_{2, \oomega}^-\in\BB^-_{2, \oomega}$
be chosen in such a way that
$\langle \Phi^+_{2, \oomega}, \Phi^-_{2, \oomega}\rangle=1$.
Let $f\in Lip_{w}^+(\oomega)$ be such that
$$
\int_{M({\mathscr X})} f d\nu=0,\  m_{\Phi_{2, \oomega}^-}(f)\neq 0.
$$
Then
\begin{equation}
\left|\frac{Var_{\nu} {\mathfrak S}[f,s;1]}{V(P^s\oomega) (m_{\Phi_2^-}(f)|\Phi^+_2|H_2(s, \oomega))^2}-1\right|\leq
C(\oomega, P^s\oomega)\exp(-\alpha s).
\end{equation}
\end{proposition}
{\bf Remark.} Observe that the quantity $(m_{\Phi_2^-}(f)|\Phi^+_2|)^2$ does
not depend on the specific choice of $\Phi_2^+\in\BB^+_2$, $\Phi_2^-\in\BB^-_2$ such that $\langle \Phi^+_2, \Phi^-_2\rangle=1$.

Indeed, the proposition is immediate from Theorem \ref{multiplicmoduli},  the inequality
$$
|\ee(\xi_1^2)-\ee(\xi_2^2)|\leq {\rm sup}|\xi_1+\xi_2| \cdot \ee|\xi_1-\xi_2|,
$$
which holds for any two bounded random variables $\xi_1, \xi_2$ on any probability space,
and the following clear Proposition, which, again, is an
immediate corollary of Theorem \ref{multiplicmoduli}.
\begin{proposition}
There exists a constant  $\alpha>0$ depending only on $\Prob_{{\mathcal V}}$, a positive measurable
function $C:\Oomega\times\Oomega\to {\mathbb R}_+$ and a positive measurable function
$V^{\prime}:\Oomega\to {\mathbb R}_+$ such that for all $s>0$ we have
\begin{equation}
\max\limits_{x\in M} |\Phi_2^+(x,e^s)|=V^{\prime}(P^s\oomega)H_2(s, \oomega);
\end{equation}
\begin{equation}
 \left|\frac{\max\limits_{x\in M} {\mathfrak S}[f,s;1](x)}{V^{\prime}(P^s\oomega) (m_{\Phi_2^-}(f)|\Phi^+|H_2(s, \oomega))^2}-1\right|\leq
C(\oomega, P^s\oomega)\exp(-\alpha s).
\end{equation}

\end{proposition}

\subsubsection{Conclusion of the proof.}

We now turn to the asymptotic behaviour of the distribution of the
random variable ${\mathfrak S}[f,s]$ as $s\to\infty$.

Again, we will use the notation
$\mm[f,s]$ for the distribution of the normalized random variable
\begin{equation}
\frac{{\mathfrak S}[f,s]}{{\sqrt{Var_{\nu} {\mathfrak S}[f,s;1]}}}.
\end{equation}
The measure $\mm[f,s]$ is thus a probability distribution on the space $C[0,1]$ of continuous functions on the unit interval.

For $\tau\in {\mathbb R}$, $\tau\neq 0$, we again let $\mm[f,s; \tau]$ be the distribution of the ${\mathbb R}$-valued random variable
\begin{equation}
\frac{{\mathfrak S}[f,s; \tau]}{{\sqrt{Var_{\nu} {\mathfrak S}[f,s; \tau]}}}.
\end{equation}
If $f$ has zero average, then, by definition,  $\mm[f,s; \tau]$ is a measure on ${\mathbb R}$ of
expectation $0$ and variance $1$.
Again, as in the Introduction, we take the space $C[0,1]$
of continuous functions on the unit interval endowed with
the Tchebyshev topology, we let $\MM$ be the space of Borel probability
measures on the space $C[0,1]$
endowed with the weak topology (see \cite{bogachev} or the Appendix).

Consider the space $\Oomega^{\prime}$ given by
the formula
$$
\Oomega^{\prime}=\{\oomega^{\prime}=(\oomega, v), v\in E_{2, \oomega}^+, |v|=1\}.
$$
The  flow $P^s$ is lifted to $\Oomega^{\prime}$ by the formula
$$
P^{s,\prime}({\mathscr X}, v)=\left(P^s{\mathscr X}, \frac{{\overline {{\mathcal A}}^t}(s,\oomega)v}
{|{\overline {{\mathcal A}}^t}(s,\oomega)v|}\right).
$$

Given $\oomega^{\prime}\in {\Oomega}^{\prime}$, $\oomega^{\prime}=(\oomega, v)$,
write
$$
\Phi_{2, \oomega^{\prime}}^+={\mathcal I}_{\oomega}(v).
$$
As before, write
$$
V(\oomega^{\prime})=Var_{\nu}   \Phi_{2, \oomega^{\prime}}^+(x,1).
$$

Now introduce the map
$$
{\mathcal D}_2^+: \Oomega^{\prime}\to \MM
$$

by setting    ${\mathcal D}_2^+(\oomega^{\prime})$ to be the distribution
of the $C[0,1]$-valued normalized random variable
$$
\frac{\Phi^+_{2,\oomega^{\prime}}(x, \tau)}{\sqrt{V(\oomega^{\prime})}}, \  \tau\in[0,1].
$$

Note here that, by Proposition \ref{phivarnotzero},
for any $\tau_0\neq 0$ we have $Var_{\nu} \Phi^+_{2, \oomega}(x, \tau_0)\neq 0$,
so, by definition, we have
${\mathcal D}_2^+(\oomega^{\prime})\in\MM_1$.

Now, as before, we take a function  $f\in Lip_{w,\oomega}^+$ such that
 $$
 \int_{M({\mathscr X})} f d\nu=0,\  m_{\Phi_{2, \oomega}^-}(f)\neq 0
 $$
As before, $d_{LP}$  stands for the L{\'e}vy-Prohorov metric on $\MM$, $d_{KR}$  for the Kantorovich-Rubinstein
metric on $\MM$.
\begin{proposition}
\label{limthmmarkcomp-simple}
Let $\Prob_{{\mathcal V}}$ be a $P^s$-invariant ergodic Borel probability measure on
${\tilde {\mathcal V}}({\mathcal R})$
such that the second Lyapunov exponent of the
cocycle ${\overline {\mathcal A}^t}$ is positive and simple with respect to $\Prob_{{\mathcal V}}$.
There exists a positive measurable function $C: \Oomega\times \Oomega\to {\mathbb R}_+$
and a positive constant $\alpha$ depending only on $\Prob_{{\mathcal V}}$
such that
for $\Prob_{{\mathcal V}}$-almost every $\oomega^{\prime}\in\Oomega^{\prime}$, $\oomega^{\prime}=(\oomega, v)$,
and any $f\in Lip_{w,0}^+({\mathscr X})$
satisfying
$m_{2, {\mathscr X}^{\prime}}^-(f)>0$ we have
\begin{equation}
d_{LP}(\mm[f,s], {\mathcal D}_2^+(P^{s, \prime}\oomega^{\prime}))\leq C(\oomega, P^s\oomega)\exp(-\alpha s).
\end{equation}
\begin{equation}
d_{KR}(\mm[f,s], {\mathcal D}_2^+(P^{s,\prime}\oomega^{\prime}))\leq C(\oomega, P^s\oomega)\exp(-\alpha s).
\end{equation}
\end{proposition}

Proof: We start with the simple inequality
$$
\left|\frac{a}{b}-\frac{c}{d}\right| \leq
|a|\cdot\left|\frac{b-d}{bd}\right|+\frac{|a-c|}{d}
$$
valid for any real numbers $a,b,c,d$.
For any pair of random variables $\xi_1, \xi_2$  taking values in an arbitrary Banach space
and any positive real numbers $M_1, M_2$
we consequently have
\begin{equation}
\label{ineqlimthm}
\sup \left| \frac{\xi_1}{M_1} - \frac{\xi_2}{M_2}\right|\leq
\sup|\xi_1|\cdot\left|\frac{M_1-M_2}{M_1M_2}\right|+\frac{\sup|\xi_1-\xi_2|}{M_2}.
\end{equation}

We apply the inequality (\ref{ineqlimthm}) to the $C[0,1]$-valued random variables
$$
\xi_1={\mathfrak S}[f,s], \ \xi_2=\Phi_{2, P^s\oomega}^+(x,\tau\cdot e^s),
$$
letting $M_1$, $M_2$ be the corresponding normalizing variances: $M_1=Var_{\nu}{\mathfrak S}[f,s; 1]$,
$M_2=Var_{\nu}\mm[f,s;1]$.

Now take $\varepsilon>0$ and  let ${\tilde \xi}_1, {\tilde \xi}_2$ be
two random variables on an arbitrary probability space $(\Omega, \Prob)$ taking values in a complete metric space
and such that the distance between their values does not exceed $\varepsilon$.
In this case both the L{\'e}vy-Prohorov and the Kantorovich-Rubinstein distance between their distributions $({\tilde \xi}_1)_*\Prob$, $({\tilde \xi}_2)_*\Prob$ also does not exceed $\varepsilon$
(see Lemma \ref{dist-images} in the Appendix).

Proposition \ref{limthmmarkcomp-simple} is now immediate from
Equation (\ref{phitwoapprox}) and Proposition \ref{varf}.

It remains to derive Proposition \ref{limthmmoduli-simple} from Proposition \ref{limthmmarkcomp-simple}.
To do so, note that the map ${\mathcal D}_2^+$, originally defined on the double cover ${\Oomega}^{\prime}$
of the space of zippered rectangles, naturally descends to a map, for which we keep the same
symbol  ${\mathcal D}_2^+$, defined on the double cover $\HH^{\prime}$ of the connected component
$\HH$ of the moduli space of abelian differentials.  Indeed, it is immediate from the definitions
that the image ${\mathcal D}_2^+({\mathscr X}^{\prime})$ of an
element  ${\mathscr X}^{\prime}\in{\Oomega}^{\prime}$,
  ${\mathscr X}^{\prime}=({\mathscr X}, v)$ only depends on the underlying
element $(M({\mathscr X}), \omega({\mathscr X}), v)$ of  the space $\HH^{\prime}$.
Proposition \ref{limthmmarkcomp-simple} is now proved completely.

\subsection{Proof of Corollary \ref{nonentwo}.}

For $\oomega^{\prime}\in\Oomega^{\prime}$, $\Phi^+\in\BB^+_{\oomega}$
let ${\mathfrak m}[\Phi^+, \tau]$ be the distribution
of the normalized ${\mathbb R}$-valued random variable
$$
\frac{\Phi^+(x, \tau)}{\sqrt{Var_{{\nu}}\Phi^+(x, \tau)}}.
$$

\begin{proposition}
Let $\Prob_{{\mathcal V}}$ be a $P^s$-invariant ergodic Borel probability measure on
${\tilde {\mathcal V}}({\mathcal R})$.
For $\Prob_{{\mathcal V}}$-almost every $\oomega$ and any $\Phi^+\in \BB^+_{\oomega}$, $\Phi^+\neq 0$,  the correspondence
$$
\tau\rightarrow \mm[\Phi^+, \tau]
$$
yields a continuous map from ${\mathbb R}\setminus 0$ to $\MM({\mathbb R})$.
\end{proposition}
Proof. This is immediate from the H{\"o}lder property of the cocycle $\Phi^+$ and the nonvanishing of the
variance  $Var_{\nu} \Phi^+(x, \tau)$ for $\tau\neq 0$, which is  guaranteed by Proposition \ref{phivarnotzero}.

As usual, by the {\it {omega-limit set}} of a parameterized curve $p(s)$, $s\in {\mathbb R}$, taking values
in a  metric space, we mean the set of all accumulation points of our curve as $s\to\infty$.

We now use the following general statement.
\begin{proposition}
\label{omlimset}
Let $(\Omega, \mathcal{B})$ be a standard Borel space, and let $g_s$ be a measurable
flow on $\Omega$ preserving an ergodic Borel probability measure $\mu$. Let $Z$ be a separable metric space, and
let $\varphi: \Omega\rightarrow Z$\:be a measurable map such that for $\mu$-almost every $\omega\in \Omega$
the curve $\varphi(g_s\omega)$ is continuous in $s\in\mathbb{R}$. Then there exists a closed set
$\mathfrak{N}\subset Z$ such that for $\mu$-almost every $\omega\in \Omega$ the set $\mathfrak{N}$
is the omega-limit set of the curve $\varphi(g_s\omega)$, $s\in\mathbb{R}$.
\end{proposition}
The proof of Proposition \ref{omlimset} is routine.
We choose a countable base ${\mathscr U}=\{U_n\}_{n\in\mathbb{N}}$  of open sets in $Z$.
By ergodicity of $g_s$, continuity of the curves  $\varphi(g_s\omega)$ and countability of the family ${\mathscr U}$,
there exists a subset of full measure $\Omega^{\prime}\subset \Omega$, $\mu(\Omega^{\prime})=1$,
such that for any $U\in {\mathscr U}$ and any $\omega\in\Omega^{\prime}$
 the following conditions are satisfied:
\begin{enumerate}
\item if $\mu(U)>0$, then there exists an infinite sequence $s_n\rightarrow\infty$ such that $\varphi(g_{s_n}\omega)\in U$;
 \item if $\mu(U)=0$, then there exists $s_0>0$ such that $\varphi(g_s\omega)\notin U$ for all $s>s_0$.
\end{enumerate}

Now let $\mathfrak{N}$  be the set of all points $z\in Z$ such that $\mu(U)>0$ for any open set
$U\in {\mathscr U}$ containing the point $z$.
By construction, for any $\omega\in\Omega^{\prime}$, the set $\mathfrak{N}$
is precisely the omega-limit set of the curve $\varphi(g_s\omega)$.
The Proposition is proved.

Proposition \ref{omlimset} with $\Omega=\HH^{\prime}$, $\varphi={\mathcal D}_2^+$ and
$\mu$ an ergodic component of $\Prob^{\prime}$ together with the Limit Theorem given by Propositions
\ref{limthmmoduli-simple}, \ref{limthmmarkcomp-simple} immediately  implies Corollary \ref{nonentwo}.

\subsection{The general case}
\subsubsection{The fibre bundles ${\bf{S}}^{(i)}\Oomega$ and
the flows $P^{s, {\bf{S}}^{(i)}}$ corresponding to the strongly unstable Oseledets subspaces.}
Let $\mathbb{P}_{\mathcal V}$ be an ergodic $P^s$-invariant probability measure on
${\Oomega}$ and let $$\theta_1=1>\theta_2>\dots>\theta_{l_0}>0$$ be
the distinct positive Lyapunov exponents of $\overline{\mathcal{A}}^t$ with
respect to $\mathbb{P}$. We assume $l_0\geq 2$.

For $\oomega\in\Oomega,$ let
$$E_{\oomega}^u=\mathbb{R}h_{\oomega}^{(0)}
+E_{2,\oomega}\oplus\dots\oplus
E_{l_0,\oomega}$$ be the corresponding direct-sum
decomposition into Oseledets subspaces, and let
$$\mathfrak{B}_{\oomega}^+=\mathbb{R}\nu_{\oomega}^+\oplus
\mathfrak{B}_{2,\oomega}^+\oplus\ldots\oplus
\mathfrak{B}_{l_0,\oomega}^+$$ be the corresponding
direct sum decomposition of the space
$\mathfrak{B}_{\oomega}^+.$

For $f\in Lip_{w}^+({{{\mathscr X}}})$ we now write
$$\Phi_f^+=\Phi_{1,f}^++\Phi_{2,f}^++\ldots +\Phi_{l_0,f}^+,$$
where $\Phi_{i,f}^+\in\mathfrak{B}_{i,\oomega}^+$ and,
of course,
$$\Phi^+_{1,f}=(\int_{M(\oomega)}fd\nu)\cdot\nu^+.$$

For each $i=2,\ldots,l_0$ introduce a measurable fibre bundle
$${\bf{S}}^{(i)}\Oomega=\{(\oomega,v):\oomega \in\Oomega,v\in E_{i,\oomega}^+,|v|=1\}.$$

The flow $P^s$ is naturally lifted to the space
${\bf{S}}^{(i)}\Oomega$ by the formula
$$
P^{s, {\bf{S}}^{(i)}}(\oomega,v)=\left(P^s \oomega,
\frac{{\overline {\mathcal A}}^t(s,\oomega)v}
{|{\overline {\mathcal {A}}}^t(s,\oomega)v|}\right).
$$

\subsubsection{Growth of the variance.}
The growth of the norm of
vectors $v\in E_i^+$ is controlled by the multiplicative cocycle
$H_i$ over the flow $P^{s, \bf{S}^{(i)}}$ defined by the formula
$$
H_i(s, (\oomega,v))=\frac{|{\overline {\mathcal A}}^t(s,\oomega)v|}{|v|}.
$$
The growth of the variance of ergodic integrals is also,
similarly to the previous case, described by the cocycle $H_i.$

For
$\oomega\in\Oomega$ and $f\in
Lip_{w,0}^+({{\mathscr X}})$ we write
\begin{equation}
\label{defif}
i(f)=\min\{j:\Phi_{f,j}^+\neq0\}.
\end{equation}
We now define a vector $v_f\in E_{i(f), {{\oomega}}}^u$ by the formula
\begin{equation}
\label{defvf}
\mathcal{I}_{{\oomega}}^+(v_f)=\frac{\Phi_{f,i(f)}^+}{|\Phi_{f,i(f)}^+|}.
\end{equation}

\begin{proposition} There exists $\alpha>0$
depending only on $\mathbb{P}_{\mathcal V}$ and, for any $i=2,\ldots,l_0,$
positive measurable functions
$$V^{(i)}:{\bf{S}}^{(i)}\Oomega\to\mathbb{R}_+,
C^{(i)}:\Oomega\times\Oomega\to\mathbb{R}_+$$ such that for
$\mathbb{P}_{\mathcal V}$-almost every $\oomega\in\Oomega$, any  $f\in
Lip_{w,0}^+({{\mathscr X}})$
and all $s>0$ we have
$$\left|\frac{Var_{\nu}(\mathfrak{S}[f,e^s;1])}
{V^{(i(f))}(P^{s,{\bf{S}}^{(i)}}(\oomega,v_{f}))(H_i(s, (\oomega,v_{f})
))^2}-1\right|\leqslant
C^{(i)}(\oomega,P^s\oomega)e^{-\alpha s}.$$
\end{proposition}

Indeed, similarly to the case of a simple Lyapunov exponent, for $v\in E_{\oomega}^i$ we
write $\Phi_v^+=\mathcal{I}_{\oomega}^+(v)$ and set
$$V^{(i)}(\oomega,v)=Var_{\nu}
\Phi_v^+(x,1).$$

The Proposition follows now in the same way as in the case of the simple second Lyapunov exponent:
the pointwise approximation of the ergodic integral by the corresponding H{\"o}lder cocycle implies also
that the variances of these random variables are exponentially close.

\subsection{Proof of Theorem \ref{limthmmoduli}.}
For $i=2,\ldots,l_0,$
introduce a map
$$\mathcal{D}_i^+:{\bf{S}}^{(i)}\Oomega\to\mathfrak{M}$$
by setting $\mathcal{D}_i^+(\oomega,v)$ to be the
distribution of the $C[0,1]$-valued random variable
$$\frac{\Phi^+_v(x,\tau)}{\sqrt{Var_{\nu}
(\Phi_v^+(x,1))}},\tau\in[0,1].$$

As before, by definition we have
$\mathcal{D}_i^+(\oomega,v)\in\mathfrak{M}_1.$ The measure
$\mathfrak{m}[f,s]\in \MM$ is, as before, the distribution of the
$C[0,1]$-valued random variable
$$
{\displaystyle{\frac{\int\limits_0^ {\tau\exp(s)} f\circ h_t^+(x)dt}
{\sqrt{Var_ {\nu}(\int\limits_0^ {\exp(s)}f\circ h_t^+(x)dt)}}}},\ \tau\in[0,1].
$$

As before, let
$l_0=l_0(\mathbb{P}_{\mathcal V})$ be the number of distinct positive Lyapunov
exponents of the measure $\mathbb{P}_{\mathcal V}$. For $f\in
Lip_{w,0}^+({{\mathscr X}})$ we define the number $i(f)$ by (\ref{defif})
and the vector $v_f$ by (\ref{defvf}).

\begin{theorem}\label{limthmmarkcomp}
Let $\mathbb{P}_{\mathcal V}$ be a Borel $P^s$-invariant ergodic probability
measure on ${\Oomega}$ satisfying $l_0(\mathbb{P}_{\mathcal V})\geq 2$.
There exists a constant $\alpha>0$ depending only on $\mathbb{P}$
and a positive measurable map
$C:\Oomega\times\Oomega\to\mathbb{R}_+$ such
that for $\mathbb{P}_{\mathcal V}$-almost every
$\oomega\in\Oomega$ and any $f\in
Lip_{w,0}^+({{\mathscr X}})$ we have
$$d_{LP}(\mathfrak{m}[f,s],D_{i(f)}^+(P^{s, {\bf{S}}^{(i(f))}}(\oomega,v_f)))
\leqslant C(\oomega,P^s\oomega)e^{-\alpha
s},$$
$$d_{KR}(\mathfrak{m}[f,s],D_{i(f)}^+(P^{s,{\bf{S}}^{(i(f))}}(\oomega,v_f)))
\leqslant C(\oomega,P^s\oomega)e^{-\alpha
s}.
$$
\end{theorem}

The proof is similar to the proof of Proposition \ref{limthmmarkcomp-simple}.
Again, the ergodic integral is uniformly approximated by the
corresponding cocycle; the uniform bound on the difference yields
the uniform bound on the difference and the ratio of variances of
the ergodic integral and the cocycle considered as random
variables; we proceed, as before, by using the inequality
(\ref{ineqlimthm})
with $\xi_1=\mathfrak{m}[f,s],\ \xi_2=\Phi^+_{f,i(f)}(x, \tau)$, and
$M_1,M_2$ the corresponding normalizing variances. We conclude, again, by
noting that a uniform bound on the difference between two random
variables implies the same bound on the L{\'e}vy-Prohorov or
Kantorovich-Rubinstein distance between the distributions of the
random variables (using Lemma \ref{dist-images} in the Appendix ).

Theorem \ref{limthmmarkcomp} now implies Theorem \ref{limthmmoduli} in the same way in which
Proposition \ref{limthmmarkcomp-simple} implies Proposition \ref{limthmmoduli-simple}.

\subsection{Atoms of limit distributions.}
Let ${\mathscr X}$ be a zippered rectangle, and let $(M, \omega)=(M({\mathscr X}), \omega({\mathscr X}))$
be the underlying abelian differential.
For $x\in M({\mathscr X})$, let $\gamma_{\infty}^+(x)$ stand for
the leaf of the vertical foliation containing $x$, and let  $\gamma_{\infty}^-(x)$ stand for
the leaf of the horizontal foliation containing $x$.
Our next aim is to show that atoms of limit distributions
occur at all ``homoclinic times'', that is,
moments of time $t_0$ such that there exists a  point
${\tilde x}\in M$ satisfying $h_{t_0}^+(x)\in \gamma_{\infty}^-({\tilde x})$.

\begin{proposition}
\label{atom-eq-pr}
Let ${\mathscr X}$ be a zippered rectangle, and let $(M, \omega)=(M({\mathscr X}), \omega({\mathscr X}))$
be the underlying abelian differential.
Let ${\tilde x}\in M$ and assume that ${\tilde x}$ does  lies  neither on
a horizontal nor on a vertical leaf passing through a singularity
of the abelian differential $\omega({\mathscr X})$.
Let $t_0\in {\mathbb R}$  be such that
$h_{t_0}^+{\tilde x}\in\gamma_{\infty}^-({\tilde x})$.
Then there exists a rectangle $\Pi$ of positive area such that
for any $x\in\Pi$ and any $\Phi^+\in\B^+({\mathscr X})$ we have
\begin{equation}
\label{atomeq}
\Phi^+(x,t_0)=\Phi^+({\tilde x}, t_0).
\end{equation}
\end{proposition}
Proof.
Let ${\hat x}=h_{t_0}^+({\tilde x})$ and write ${\hat x}=h_{t_1}^-({\tilde x})$.
Start with the case $t_0>0, t_1>0$. By our assumptions, for sufficiently small positive $t_2$, $t_3$,
the rectangles
$$
\Pi_1=\Pi({\tilde x}, t_2, t_1+t_3), \Pi_2={\Pi}({\tilde x}, t_0+t_2, t_3)
$$
are both admissible.

The desired rectangle $\Pi$ can now be taken of the form
$$
\Pi=\Pi({\tilde x}, t_2, t_3).
$$
Indeed, take $x\in\Pi$. Our aim is to check the equality  (\ref{atomeq}). Write $x=h_t^+x_1$, where
$x_1\in \partial_h^0(\Pi)$. We first check the equality
   \begin{equation}
\label{atomeq2}
\Phi^+(x,t_0)=\Phi^+(x_1, t_0).
\end{equation}
But indeed, $\Phi^+(x_1, t)=\Phi^+(h_{t_0-t}^+x, t)$ since $\Pi_1$ is admissible, whence
$$
\Phi^+(x, t_0)=\Phi^+(x, t_0-t)+\Phi^+(h_{t_0-t}^+x, t)=\Phi^+(x_1, t)+\Phi^+(h_t^+x_1, t_0-t)=\Phi^+(x_1, t_0),
$$
as desired.
The equality
 \begin{equation}
\label{atomeq3}
\Phi^+(x_1,t_0)=\Phi^+({\tilde x}, t_0)
\end{equation}
is a direct corollary of admissibility of $\Pi_2$. Combining (\ref{atomeq2}) with (\ref{atomeq3}),
we arrive at the desired equality  (\ref{atomeq2}), and Proposition \ref{atom-eq-pr} is proved.

For a fixed zippered rectangle $\oomega$ both whose vertical and horizontal flows are minimal,
the set of ``homoclinic times'' $t_0$ for which there
exist ${\tilde x}, {\hat x}\in X$  satisfying
${\tilde x}\in \gamma_{\infty}^+({\hat x})$,
${\tilde x}\in \gamma_{\infty}^-({\hat x})$, ${\hat x}=h_{t_0}^+{\tilde x}$, is
countable and dense in ${\mathbb R}$.
Proposition \ref{atom-eq-pr} now implies the following
\begin{corollary}
Let $\mathbb{P}_{\mathcal V}$ be a Borel $P^s$-invariant ergodic probability  measure on
${\Oomega}$.
For $\Prob_{\mathcal V}$-almost every $\oomega\in\Oomega$, there exists a dense set of times $t_0\in {\mathbb R}$ such that for
any $\Phi^+\in\B^+$  the distribution of the random variable $\Phi^+(x, t_0)$ has an atom.
\end{corollary}

Our next step is to show that atoms of weight arbitrarily close to
$1$ occur for limit distributions of our H{\"o}lder cocycles.
Informally, such atoms exist when one admissible rectangle
occupies most of our surface.
More precisely, we have the following
\begin{proposition}
\label{bigatom}
Let $\oomega\in\Oomega$ satisfy $\la_{1}^{(0, \oomega)}>1/2$. Then there exists a set ${\Pi}\subset M(\oomega)$
such that
\begin{enumerate}
\item $\nu_{\oomega}({\Pi})\geq (2\la_1^{(0, \oomega)}-1)h_1^{(0, \oomega)}$;
\item for any $\Phi^+\in {\mathfrak B}^+(\oomega)$, the function $\Phi^+(x, h_1^{(0, \oomega)})$ is constant
on ${\Pi}$.
\end{enumerate}
\end{proposition}

Proof. We consider $\oomega$ fixed and omit it from notation.
Consider the partition
$$
\pi_0({\mathscr X})=\Pi^{(0)}_1\sqcup\dots\sqcup \Pi^{(0)}_m
$$
of the zippered rectangle ${\mathscr X}$.
Let $I_k$ be the interval forming lower horizontal boundaries of the rectangles $\Pi^{(0)}_k$,
$k=1, \dots, m$, and set
$$
I=I_1\sqcup\dots\sqcup I_m.
$$

The flow transversal $I$ carries the Lebesgue measure $\nu_I$ invariant under
the first-return map of the flow $h_t^+$ on $I$. We recall that the first return map
is simply the interval exchange transformation $(\la,\pi)$ of the zippered rectangle
${\mathscr X}=(\la,\pi,\delta)$.
We recall that $\la^{(0)}_k$ is the length of $I_k$, and that $h_k^{(0)}$ is the height of
$\Pi^{(0)}_k$.
For brevity, denote $t_1=h_1^{(0)}$. By definition,  $h^+_{t_1}I_1\subset I$ and we have
$$
\nu_I(I_1\bigcap h^+_{t_1}I_1)\geq 2\la_1^{(0)}-1>0.
$$
Introduce the set
$$
\Pi=\{h^+_{\tau}x, 0<\tau<t_1, x\in I_1, h_{t_1}x\in I_1\}.
$$
The first statement of the Proposition is clear, and we proceed to the proof of the second.
Note first that for any $\Phi^+\in {\mathfrak B}^+(\oomega)$ and any $\tau, 0\leq\tau\leq t_1$
the quantity $\Phi^+(x, \tau)$
is constant as long as $x$ varies in $I_1$.

Fix $\Phi^+\in {\mathfrak B}^+(\oomega)$ and take an arbitrary ${\tilde x}\in \Pi$.
Write ${\tilde x}=h^+_{\tau_1}x_1$, where $x_1\in I_1$,
$0<\tau_1<t_1$. We have $h^+_{t_1-\tau_1}{\tilde x}\in I_1$, whence
$$
\Phi^+(h^+_{t_1-\tau_1}{\tilde x}, \tau_1)=\Phi^+(x_1, \tau_1)
$$
and
$$
\Phi^+({\tilde x}, t_1)=\Phi^+({\tilde x}, t_1-\tau_1)+\Phi^+(h^+_{t_1-\tau_1}{\tilde x}, \tau_1)=
\Phi^+(h^+_{\tau_1}x_1, t_1-\tau_1)+\Phi^+(x_1, \tau_1)=\Phi^+(x_1, t_1),
$$
which concludes the proof of  the Proposition.

\begin{figure}
\begin{center}
\includegraphics{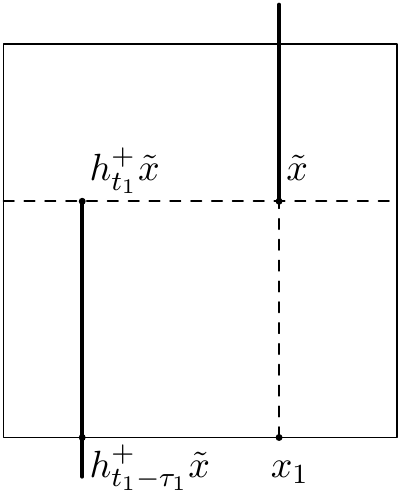}\\
\caption{Atoms of limit distributions}\label{fig:two}
\end{center}
\end{figure}
We illustrate the proof by Figure \ref{fig:two}.

\subsection{Accumulation at zero for limit distributions}

Recall that for $\oomega^{\prime}\in\Oomega^{\prime}$, $\Phi^+\in\BB^+_{\oomega}$, $\Phi^+\neq 0$,
and $\tau\in {\mathbb R}$, $\tau\neq 0$,
the measure ${\mathfrak m}[\Phi^+, \tau]$ is the distribution
of the normalized ${\mathbb R}$-valued random variable
$$
\frac{\Phi^+(x, \tau)}{\sqrt{Var_{{\nu}}\Phi^+(x, \tau)}}.
$$
As before, let $\MM({\mathbb R})$ be the space of probability measures on ${\mathbb R}$ endowed with the weak topology,
and let $\delta_0\in\MM({\mathbb R})$ stand for the delta-measure at zero.
Similarly to the Introduction, we need the following additional assumption on our
$P^s$-invariant ergodic probability  measure $\mathbb{P}_{\mathcal V}$ on ${\Oomega}$.
\begin{assumption}
\label{fatzip}
For any $\varepsilon>0$ we have
$$
\Prob_{\mathcal V}(\{\oomega: \la_1^{(\oomega)}>1-\varepsilon, h_1^{(\oomega)}>1-\varepsilon \})>0.
$$
\end{assumption}

By Proposition \ref{bigatom}, in view of the ergodicity of $\Prob_{\mathcal V}$,
for almost every $\oomega\in\Oomega$
and every $\Phi^+\in\BB^+_{\oomega}$, $\Phi^+\neq 0$, the sequence of measures ${\mathfrak m}[\Phi^+, \tau]$
admits atoms of weight arbitrarily close to $1$.
The next simple Proposition shows that the corresponding measures must then accumulate at {\it zero}
(rather than at another point of the real line).

\begin{proposition} \label{atomatzero}
Let $\mu_0$ be a probability
measure on $\mathbb{R}$ such that
$$\int_{\mathbb{R}}xd\mu_0(x)=0,\ \ \int_{\mathbb{R}}x^2d\mu_0(x)=1.$$
Let $x_0\in\mathbb{R}$ and assume that
$$\mu_0(\{x_0\})=\beta.$$
Then $$|x_0|^2\leqslant\frac{1-\beta}{\beta^2}.$$
\end{proposition}

Proof. If $x_0=0$, then there is nothing to
prove, so assume $x_0>0$ (the remaining case $x_0<0$ follows by
symmetry). We have $$\int_0^{+\infty}xd\mu_0(x)\geqslant\beta
x_0,$$ and, consequently,
$$\int_{-\infty}^0xd\mu_0(x)\leqslant-\beta x_0.$$

Using the Cauchy-Bunyakovsky-Schwarz inequality, write
$$\frac{1}{\mu_0((-\infty,0))}\int_{-\infty}^0x^2d\mu_0(x)\geqslant\left(\frac{1}{\mu_0((-\infty,0))}\int_{-\infty}^0xd\mu_0(x)\right)^2,$$
whence, recalling that the variance of $\mu_0$ is equal to $1$, we obtain
$$\mu_0((-\infty,0))\geqslant\left(\int_{-\infty}^0xd\mu_0(x)\right)^2$$
and, finally, $$1-\beta\geqslant\beta^2x_0^2,$$ which is what we
had to prove.

As before, the
symbol $\Rightarrow$ denotes weak convergence of probability
measures.

\begin{proposition}\label{convtodelta} Let $\mathbb{P}_{\mathcal V}$ be a Borel  ergodic
$P^s$-invariant probability measure on $\Oomega$ satisfying
Assumption \ref{fatzip}. Then for $\mathbb{P}_{\mathcal V}$-almost
every ${{{\mathscr X}}}\in\Oomega$ there exists a sequence
$\tau_n\in\mathbb{R}_+$ such that for any
$\Phi^+\in{\mathfrak{B}}^+({{{\mathscr X}}})$ we have
$$\mathfrak{m}[\Phi^+,\tau_n]\Rightarrow\delta_0\ \mathrm{as}\
n\to\infty.$$
\end{proposition}

This is immediate from Proposition \ref{bigatom} and Proposition \ref{atomatzero}.

\begin{corollary}\label{nonconvergence}
Let $\mathbb{P}_{\mathcal V}$ be a Borel ergodic
$P^s$-invariant probability measure on $\Oomega$ satisfying
Assumption \ref{fatzip}. Then for $\mathbb{P}_{\mathcal V}$-almost
every ${{{\mathscr X}}}\in\Oomega$ there exists a sequence
$s_n\in\mathbb{R}_+$ such that for any $f\in
Lip_{w,0}^+({{{\mathscr X}}})$ satisfying $\Phi_f^+\neq0$ we have
$$\mathfrak{m}[f,s_n;1]\Rightarrow\delta_0\ \mathrm{as}\
n\to\infty.$$

Consequently, if $f\in Lip_{w,0}^+({{{\mathscr X}}})$ satisfies
$\Phi_f^+\neq0,$ then the family of measures
$\mathfrak{m}[f,s;1]$ does not converge in the weak topology on $\MM({\mathbb R})$ as
$s\to\infty$ and the family of measures
$\mathfrak{m}[f,s]$ does not converge in the weak topology on $\MM(C[0,1])$ as
$s\to\infty$.
\end{corollary}

Proof. The first claim is clear from
Proposition \ref{convtodelta} and the Limit Theorem \ref{limthmmarkcomp}.
The second claim is obtained from the Limit
Theorem \ref{limthmmarkcomp}  in the following way.

First note that the set
\begin{equation}
\label{phitausphere}
\{\mathfrak{m}[\Phi^+,1], \Phi^+\in\mathfrak{B}^+({{{\mathscr X}}}),|\Phi^+|=1\}
\end{equation}
is compact in the weak topology (indeed, it is clear from the {\it uniform}
convergence on spheres in the Oseledets Multiplicative Ergodic Theorem
that the map $$\Phi^+\rightarrow\mathfrak{m}[\Phi^+,\tau]$$ is
continuous in restriction to the set $\{\Phi^+: |\Phi^+|=1\}$ whose image is therefore compact).
In particular, the set (\ref{phitausphere}) is bounded away from $\delta_0$, and the function
$$
\kappa(\oomega)=\inf\limits_{{\Phi^+: |\Phi^+|=1}} d_{LP}(\mm[\Phi^+, 1], \delta_0)
$$
is a positive measurable function on $\Oomega$. Consequently, there exists $\kappa_0>0$ such that
$$
\Prob_{\mathcal V}(\{\oomega: \kappa(\oomega)>\kappa_0\})>0.
$$
From ergodicity of the measure $\Prob_{\mathcal V}$ and the Limit  Theorem \ref{limthmmarkcomp} it follows that
the family  $\mathfrak{m}[f,s;1]$, $s\in {\mathbb R}$, does not converge to $\delta_0$.
On the other hand, as we have seen, the measure $\delta_0$
is an accumulation point for the family. It follows that the measures  $\mathfrak{m}[f,s;1]$ do
not converge in $\MM({\mathbb R})$ as $s\to\infty$, and, a fortiori,
that the measures $\mathfrak{m}[f,s]$ do not converge in $\MM(C[0,1])$ as $s\to\infty$.

Corollary \ref{nonconvergence} is proved completely.

\section{Appendix: Metrics on the Space of Probability  Measures.}
\subsection{The Weak Topology.}

In this Appendix, we collect some standard facts about the weak
topology on the space of probability measures. For a detailed
treatment, see, e.g., \cite{bogachev}.

Let $(\emph{X}, {d})$ be a complete separable metric
space, and let  $\mathfrak{M}(\emph{X})$ be the space of
Borel probability measures on \emph{X}. The {\it weak topology}
on $\mathfrak{M}(\emph{X})$ is defined as follows. Let
$\varepsilon>0$, $\nu_0\in \mathfrak{M}(\emph{X})$, and let
$f_1,..,f_k: \emph{X} \rightarrow \mathbb{R} $ be bounded
continuous functions. Introduce the set
$$U(\nu_0, \varepsilon,f_1,..,f_k)=\{\nu\in\mathfrak{M}(\emph{X}):
 |  \int\limits_X f_i d\nu-\int\limits_X f_i d\nu_0|<\varepsilon,
 i=1,..,k\}.$$

The basis of neighbourhoods for the weak topology is given
precisely by sets of the form $U(\nu_0, \varepsilon,f_1,..,f_k)$, for
all $\varepsilon>0,$ $\nu_0\in \mathfrak{M}(\emph{X}),$
$f_1,..,f_k$ continuous and bounded.

The weak topology is metrizable and there are several natural
metrics on $\mathfrak{M}(\emph{X})$ inducing the weak topology.

\subsection{The Kantorovich-Rubinstein metric}

Let $$Lip_1^1=\{f:\emph{X}\rightarrow \mathbb{R}\ \ :\
\sup_X|f|\leqslant 1,\  |f(x_1)-f(x_2)|\leqslant d(x_1,x_2)\  \mathrm{for}\
\mathrm{all} \ x_1,x_2\in\emph{X}\}.$$

The Kantorovich-Rubinstein metric is defined, for
$\nu_1,\nu_2\in\mathfrak{M}(\emph{X}),$ by the formula
$$d_{KR}(\nu_1,\nu_2)=\sup_{f\in Lip_1^1(\emph{X})} |
\int\limits_{\emph{X}} f d\nu_1-\int\limits_{\emph{X}} f
d\nu_2|.$$

The Kantorovich-Rubinstein metric induces the weak topology on
$\mathfrak{M}(\emph{X}).$ By the Kantorovich-Rubinstein
Theorem, for bounded metric spaces, the Kantorovich-Rubinstein metric admits the following
equivalent dual description. Given
$\nu_1,\nu_2\in\mathfrak{M}(\emph{X}),$ let $\rm
Join(\nu_1,\nu_2)\in \mathfrak{M}(\emph{X}\times\emph{X})$ be the
set of probability measures $\eta$ on $\emph{X}\times\emph{X}$
such that projection of $\eta$ on the first coordinate is equal to
$\nu_1,$ the projection of $\eta$ on the second coordinate is
equal to $\nu_2.$ The Kantorovich-Rubinstein Theorem claims that
$$d_{KR}(\nu_1,\nu_2)=\inf_{\eta\in\rm
Join(\nu_1,\nu_2)}\int\limits_{\emph{X}\times\emph{X}}d(x_1,x_2)d\eta.$$

\subsection {The L{\'e}vy-Prohorov metric.}

Let $\mathcal{B}_\emph{X}$ be the $\sigma$-algebra of Borel
subsets of $\emph{X}.$ For $B\in\mathcal{B}_\emph{X},
\varepsilon>0,$ set $$B^\varepsilon=\{x\in\emph{X}:\ \inf_{y\in
B}d(x,y)\leqslant\varepsilon\}.$$

Given $\nu_1,\nu_2\in \mathfrak{M}(\emph{X})$, introduce the
L{\'e}vy-Prohorov distance between them by the formula
$$d_{LP}(\nu_1,\nu_2)=\inf\{\varepsilon>0: \nu_1(B)\leqslant\nu_2(B^\varepsilon)+\varepsilon,
\nu_2(B)\leqslant\nu_1(B^\varepsilon)+\varepsilon \ \mathrm{for}\ \mathrm{any}\ B\in\mathcal{B}\}.$$

The L{\'e}vy-Prohorov metric also induces the weak topology on
$\mathfrak {M}(\emph{X}).$

\subsection{ An estimate on the distance between images of
measures.}

Let $(\Omega,\mathfrak{B}_\Omega,\mathbb{P})$ be a probability
space, and let $\xi_1,\xi_2:\Omega\rightarrow\emph{X}$ be two
measurable maps.

In the proof of the limit theorems, we use the following
simple estimate on the L{\'e}vy-Prohorov and the Kantorovich-Rubinstein
distance between the push-forwards
$(\xi_1)_\ast\mathbb{P},(\xi_2)_\ast\mathbb{P}$ of the measure
$\mathbb{P}$ under the mappings $\xi_1,\xi_2.$

\begin{lemma}\label{dist-images} Let $\varepsilon>0$ and assume
that for $\mathbb{P}-almost$ all $\omega\in\Omega$ we have
$d(\xi_1(\omega),\xi_2(\omega))\leqslant\varepsilon.$

Then we have
$$d_{KR}((\xi_1)_\ast\mathbb{P},(\xi_2)_\ast\mathbb{P})\leqslant\varepsilon,$$

$$d_{LP}((\xi_1)_\ast\mathbb{P},(\xi_2)_\ast\mathbb{P})\leqslant\varepsilon.$$
\end{lemma}

The proof of the lemma is immediate from the definitions of the
Kantorovich-Rubinstein and the L{\'e}vy-Prohorov metric.

\end{document}